\documentclass[11pt]{article}

\usepackage[T1]{fontenc}
\usepackage[utf8]{inputenc} 
\usepackage[english]{babel}

\usepackage[dvipsnames,svgnames]{xcolor}

\usepackage{amsmath, amssymb, amsthm, mathtools}
\usepackage{mathrsfs}

\usepackage{hyperref}
\hypersetup{
    colorlinks=true,
    citecolor=RoyalBlue,
    linkcolor=PineGreen,
    filecolor=magenta,      
    urlcolor=Orange,
    pdftitle={Overleaf Example},
    pdfpagemode=FullScreen,
    }

\usepackage[round]{natbib}

\usepackage{comment}
\usepackage{nicefrac}

\usepackage{microtype}
\usepackage[a4paper,margin=1in]{geometry}

\theoremstyle{plain}
\newtheorem{thm}{Theorem}[section]
\newtheorem{lem}[thm]{Lemma}
\newtheorem{prop}[thm]{Proposition}
\newtheorem{cor}[thm]{Corollary}

\theoremstyle{remark}
\newtheorem{ex}[thm]{Example}
\newtheorem{remark}[thm]{Remark}
\newtheorem{defi}[thm]{Definition}

\DeclareMathOperator*{\argmin}{arg\,min}
\DeclareMathOperator*{\oequivalent}{o}
\DeclareMathOperator*{\Oequivalent}{\mathcal{O}}

\newcommand\bx{{\bf x}}
\newcommand\by{{\bf y}}
\newcommand\bX{{\bf X}}

\newcommand{\appendixlink}{Appendix \ref{appn}}
\newcommand{\supplementary}{the \hyperref[appn]{Appendix}}
\newcommand{\supplementSecTwo}{Appendix \ref{app:lemmas}}
\newcommand{\supplementSecFive}{Appendix \ref{sec:consistency}}

\title{On the convergence of PINNs}

\author{
  Nathan Doum\`eche$^{1}$ \and
  G\'erard Biau$^{1,2}$ \and
  Claire Boyer$^{1,2}$
}

\date{
  $^{1}$Sorbonne Universit\'e, CNRS, LPSM, F-75005 Paris, France\\
  $^{2}$Institut universitaire de France (IUF)\\[4pt]
  \texttt{nathan.doumeche@sorbonne-universite.fr},
  \texttt{gerard.biau@sorbonne-universite.fr},
  \texttt{claire.boyer@sorbonne-universite.fr}
}

\begin{document}
\maketitle

\begin{abstract}
Physics-informed neural networks (PINNs) are a promising approach that combines the power of neural networks with the interpretability of physical modeling. PINNs have shown good practical performance in solving partial differential equations (PDEs) and in hybrid modeling scenarios, where physical models enhance data-driven approaches. However, it is essential to establish their theoretical properties in order to fully understand their capabilities and limitations. In this study, we highlight that classical training of PINNs can suffer from systematic overfitting.  This problem can be addressed by adding a ridge regularization to the empirical risk, which ensures that the resulting estimator is risk-consistent for both linear and nonlinear PDE systems. However, the strong convergence of PINNs to a solution satisfying the physical constraints requires a more involved analysis using tools from functional analysis and calculus of variations. In particular,
for linear PDE systems, an implementable Sobolev-type  regularization 
allows to reconstruct a solution that not only achieves statistical accuracy but also maintains consistency with the underlying physics.
\end{abstract}


\section{Introduction}

\paragraph{Physics-informed machine learning}
Advances in machine learning and deep learning have led to significant breakthroughs in almost all areas of science and technology. However, despite remarkable achievements, modern machine learning models are difficult to interpret and do not necessarily obey the fundamental governing laws of physical systems \citep{linardatos2021explainability}. Moreover, they often fail to extrapolate scenarios beyond those on which they were trained \citep{xu2021extrapolation}. 
On the contrary, numerical or pure physical methods struggle to capture nonlinear relationships in complex and high-dimensional systems, while lacking flexibility and being prone to computational problems.  
This state of affairs has led to a growing consensus that data-driven machine learning methods need to be coupled with prior scientific knowledge based on physics. 
This emerging field, often called {physics-informed machine learning} \citep{raissi2019PINN}, seeks to combine the predictive power of machine learning techniques with the interpretability and robustness of physical modeling. 
The literature in this field is still disorganized, with a somewhat unstable nomenclature. In particular, the terms physics-informed, physics-based, physics-guided, and theory-guided are used interchangeably. 
For a comprehensive account, we refer to the reviews by \citet{rai2020review}, \citet{karniadakis2021piml}, \citet{cuomo2022scientific}, and \citet{Hao2022review}, which survey some of the prevailing trends in embedding physical knowledge in machine learning, present some of the current challenges, and discuss various applications.

\smallskip
\paragraph{Vocabulary and use cases} Depending on the nature of the interaction between machine learning and physics, physics-informed machine learning is usually achieved by preprocessing the features \citep{rai2020review}, by designing innovative network architectures that incorporate the physics of the problem \citep{karniadakis2021piml}, or by forcing physics infusion into the loss function \citep{cuomo2022scientific}. 
It is this latter approach, which is most often referred to as physics regularization \citep{rai2020review}, to which our article is devoted.  Note that other names are possible, including physics consistency penalty \citep{wang2020superResolution}, knowledge-based loss term \citep{vonRueden2021informed}, and physics-guided neural networks \citep{cunha2022review}. In the following, we will focus more specifically on neural networks incorporating a physical regularization, called PINNs (for physics-informed neural networks, \citealt{raissi2019PINN}).
Such models have been successfully applied to $(i)$ model hybrid learning tasks, where the data-driven loss is regularized to satisfy a  physical prior, and $(ii)$ design efficient solvers of partial differential equations (PDEs).
A significant advantage of PINNs is that they are easy to implement compared to other PDE solvers, and that they rely on the backpropagation algorithm, resulting in reasonable computational cost.
Although $(i)$ and $(ii)$ are different facets of the same mathematical problem, they differ in their geometry and the nature of the data on which they are based, as we will see later. 

\smallskip
\paragraph{Related work and contributions} Despite a rapidly growing literature highlighting the capabilities of PINNs in various real-world applications, there are still few theoretical guarantees regarding {the} overfitting, consistency, and error analysis of the approach. Most existing theoretical work focuses either on intractable modifications of PINNs \citep{cuomo2022scientific} or on negative results, such as in \citet{krishnapriyan2021characterizing} and \citet{wang2022ntk}. 

Our goal in the present article is to provide a comprehensive theoretical analysis of the mathematical forces driving PINNs, in both the hybrid modeling and PDE solver settings, 
with the constant concern to provide approaches that can be implemented in practice. Our results complement those of \citet{shin2020convergence}, \citet{shin2020errorEstimates}, \citet{mishra2022generalization}, \citet{ryck2022kolmogorov},   \citet{wu2022convergence}, and \citet{qian2023error} for the PDE solver problem. \citet{shin2020convergence} and \citet{wu2022convergence} focus on modifications of PINNs using the Hölder norm of the neural network in the loss function, which is unfortunately intractable in practice. In the context of linear PDEs, \citet{shin2020errorEstimates} analyze the expected generalization error of PINNs using the Rademacher complexity of the image of the neural network class by a differential operator. However, this Rademacher complexity does not obviously vanish with increasing sample size. 
Similarly, \citet{mishra2022generalization} bound the generalization error by a quadrature rule depending on the Hölder norm of the neural network, which does not necessarily tend to zero as the number of training points tends to infinity. \citet{ryck2022kolmogorov} derive bounds on the expectation of the $L^2$ error, provided that the weights of the neural networks are bounded. 
In contrast to this series of works, we consider models and assumptions that can be practically verified or implemented. Moreover, our approach includes hybrid modeling, for which, as pointed out by \citet{karniadakis2021piml}, no theoretical guarantees have been given so far. Preliminary interesting results on the statistical consistency of a regression function penalized by a PDE are reported in \citet{arnone2022spatialRegression}. 
The original point of our approach lies in the use of a mix of statistical and functional analysis arguments \citep{evans2010partial} to characterize the PINN problem.

\smallskip
\paragraph{Overview} After correctly defining the PINN problem in Section \ref{TPF}, we show in Section \ref{POF} that an additional regularization term is needed in the loss, otherwise PINNs can overfit. This first important result is consistent with the approach of \citet{shin2020convergence}, which penalizes PINNs by Hölder norms to ensure their convergence, and with the experiments of \citet{nabian2019engineering}, which improve performance by adding an extra-regularization term. In Section \ref{sec:consistency}, we establish the consistency of ridge PINNs by proving in Theorem \ref{thm:generalization_error} that a slowly va\-ni\-shing ridge penalty is sufficient to prevent overfitting.
Finally, in Section \ref{sec:functional}, we show that an additional level of regularization is sufficient in order to guarantee the strong convergence of PINNs (Theorem \ref{thm:functionalCv}).
We also prove that an adapted tuning of the hyperparameters allows to reconstruct the solution in the PDE solver setting (Theorem \ref{prop:pdeSolverFunctional}), as well as to ensure both statistical and physics consistency in the hybrid modeling setting (Theorem \ref{cor:sPINNsConsistency}). All proofs are postponed to \supplementary.
The code of all the numerical experiments can be found at \citet{supplement2023code} or at \url{https://github.com/NathanDoumeche/Convergence_and_error_analysis_of_PINNs}.

\section{The PINN framework}
\label{TPF}

In its most general formulation, the PINN method can be described as an empirical risk minimization problem, penalized by a PDE system.

\smallskip
\paragraph{Notation} Throughout this article, the symbol $\mathbb E$ denotes expectation and $\|\cdot\|_2$ (resp.,  $\langle \cdot, \cdot\rangle$) denotes the Euclidean norm (resp., scalar product) in $\mathbb R^d$, where $d$ may vary depending on the context. 
Let $\Omega \subset \mathbb R^{d_1}$ be a bounded Lipschitz domain with boundary $\partial \Omega$ and closure $\bar \Omega$, and let $(\bX,Y) \in \Omega \times \mathbb{R}^{d_2}$ be a pair of random variables. Recall that Lipschitz domains are a general category of open sets that includes bounded convex domains (such as $]0,1[^{d_1})$ and usual manifolds with $C^1$ boundaries (see \appendixlink). 
This level of generality with respect to the domain $\Omega$ is necessary to encompass most of the physical problems, such as those presented in \citet{arzani2021uncovering}, which use non-trivial (but Lipschitz) geometries. 
For $K \in \mathbb{N}$, the space of functions from $\Omega$ to $\mathbb{R}^{d_2}$ that are $K$ times continuously differentiable is denoted by $C^K(\Omega, \mathbb{R}^{d_2})$. 

Let $C^\infty(\Omega, \mathbb{R}^{d_2}) = \cap_{K \geqslant 0}C^K(\Omega, \mathbb{R}^{d_2})$ be the space of infinitely differentiable functions. 
The space $C^K(\Omega, \mathbb{R}^{d_2})$ is endowed with the Hölder norm $\|\cdot\|_{C^K(\Omega)}$, 
defined for any $u$ by $\|u\|_{C^K(\Omega)} = \max_{|\alpha|\leqslant K} \|\partial^\alpha u\|_{\infty, \Omega}$. 
The space $C^\infty(\bar{\Omega}, \mathbb{R}^{d_2})$ of smooth functions is defined as the subspace of continuous functions $u:\bar{\Omega} \to \mathbb{R}^{d_2}$ satisfying $u|_\Omega \in C^\infty(\Omega, \mathbb{R}^{d_2})$ and, for all $K\in \mathbb{N}$, $\|u\|_{C^K(\Omega)} < \infty$. 
 A differential operator $\mathscr{F} : C^\infty(\Omega, \mathbb{R}^{d_2}) \times \Omega \to \mathbb{R}$ is said to be of order $K$ if it can be expressed as a function over the partial derivatives of order less than or equal to $K$. For example, the operator $\mathscr{F}(u, \bx) = \partial_1 u(\bx)  \partial^2_{1,2} u(\bx) + u(\bx)\sin(\bx) $ has order 2. 
 A summary of the mathematical notation used in this paper is to be found in \appendixlink.

\smallskip
\paragraph{Hybrid modeling} As in classical regression analysis, we are interested in estimating the unknown regression function $u^\star$ such that $Y = u^\star({\bf X}) +\varepsilon$, for some random noise $\varepsilon$ that satisfies $\mathbb E(\varepsilon|\bX)=0$. 
What makes the problem original is that the function $u^{\star}$ is assumed to satisfy (at least approximately) a collection of $M \geqslant 1$ PDE-type constraints of order at most $K$, denoted in a standard form by $\mathscr{F}_k(u^\star,\bx)\simeq 0$ for $1 \leqslant k \leqslant M$. 
It is therefore assumed that $u^\star$ can be derived $K$ times. 
Moreover, there exists some subset $E \subseteq \partial \Omega$ and an boundary/initial condition function $h:E\to \mathbb R^{d_2}$ such that, for all $\bx\in E$, $u^\star(\bx) \simeq h(\bx)$. 
We stress that $E$ can be strictly included in $\Omega$, as shown in 
Example \ref{ex:spatioTemp} for a spatio-temporal domain $\Omega$. 
The specific case $E = \partial \Omega$ corresponds to Dirichlet boundary conditions. 

These constraints model some a priori physical information about $u^{\star}$. However, this knowledge may be incomplete (e.g., the PDE system may be ill-posed and have no or multiple solutions) and/or imperfect (i.e., there is some modeling error, that is, $\mathscr{F}_k(u^{\star},\bx)\neq 0$ and $u^\star|_E \neq h$).
This again emphasizes that $u^\star$ is not necessarily a solution of the system of differential equations. 
\begin{ex}[Maxwell equations]
Let $\bx = (x, y, z, t) \in \mathbb{R}^3 \times \mathbb{R}_+$, and consider  Maxwell equations describing the evolution of an electro-magnetic field $u^\star = (E^\star, B^\star)$ in vacuum, defined by
\[
\left\{\begin{array}{rcl}
      \mathscr{F}_1(u^\star, \bx) &=& \mathrm{div} E^\star(\bx)\\
      \mathscr{F}_2(u^\star, \bx) &=& \mathrm{div} B^\star(\bx)\\
      (\mathscr{F}_{3},\ \mathscr{F}_{4},\ \mathscr{F}_{5}) (u^\star, \bx) &=& \partial_t E^\star(\bx)  - \mathrm{curl} B^\star(\bx)\\
      (\mathscr{F}_{6},\ \mathscr{F}_{7},\ \mathscr{F}_{8})(u^\star, \bx) &=& \partial_t B^\star(\bx) + \mathrm{curl} E^\star(\bx),\\
\end{array}\right.
\]
where $E^\star\in C^1(\mathbb{R}^4, \mathbb{R}^3)$ is the electric field,  $B^\star \in C^1(\mathbb{R}^4, \mathbb{R}^3)$ the magnetic field, and the $\mathrm{div}$ and $\mathrm{curl}$ operators are respectively defined for $F = (F_x, F_y, F_z) \in C^1(\mathbb{R}^4, \mathbb{R}^3)$ by
\[ 
\mathrm{div} F = \partial_{x} F_x + \partial_{y} F_y + \partial_{z} F_z \quad \mbox{and} \quad
\mathrm{curl} F =  (\partial_y F_z - \partial_z F_y,\ \partial_z F_x - \partial_x F_z,\ \partial_x F_y - \partial_y F_x).
\]
In this case, $d_1=4$, $d_2 = 6$, and $M = 8$. 
\end{ex}

\begin{ex}[Spatio-temporal condition function]
\label{ex:spatioTemp}
    Assume that the domain $\Omega \subseteq \mathbb{R}^{d_1}$ is of the form $\Omega=\Omega_1 \times ]0,T[$, 
where $\Omega_1 \subseteq \mathbb{R}^{d_1-1}$ is a bounded Lipschitz domain and $T\geqslant 0$ is a finite time horizon. The spatio-temporal PDE system admits (spatial) boundary conditions specified by a function $f:\partial \Omega_1 \to \mathbb R^{d_2}$, i.e.,
\[
\forall x \in \partial \Omega_1, \ \forall t \in [0, T], \quad u^\star(x, t) = f(x),
\]
and a (temporal) initial condition specified by a function $g: \Omega_1 \to \mathbb R^{d_2}$, that is
\[
\forall x \in \Omega_1, \quad u^\star(x, 0) = g(x).
\]
The set on which the boundary and initial conditions are defined is $E = (\Omega_1\times \{0\}) \cup (\partial \Omega_1 \times [0,T]) $, and the associated condition function $h : E \to \mathbb{R}^{d_2}$ is
\[
h(\bx) = \left\{\begin{array}{lll}
       f(x) &\text{if} &\bx= (x,t) \in \partial \Omega_1 \times [0,T]\\
       g(x) &\text{if} &\bx=(x,t) \in \Omega_1\times \{0\}.
\end{array}\right.
\]
Notice that $E\subsetneq \partial \Omega$. 
\end{ex}
\hfill \\
In order to estimate $u^\star$, we assume to have at hand three sets of data:
\begin{itemize}
\item[$(i)$] A collection of i.i.d.~random variables $(\bX_1,Y_1), \hdots, (\bX_n,Y_n)$ distributed as $(\bX,Y) \in \Omega \times \mathbb R^{d_2}$, the distribution of which is \textit{ unknown};
\item[$(ii)$] A collection of i.i.d.~random variables $\bX^{(e)}_1, \hdots, \bX^{(e)}_{n_e}$ distributed according to some \textit{known} distribution $\mu_E$ on $E$;
\item[$(iii)$] A sample of i.i.d.~random variables $\bX^{(r)}_1, \hdots, \bX^{(r)}_{n_r}$ \textit{ uniformly distributed} on $\Omega$.
\end{itemize}
The function $u^\star$ is then estimated by minimizing the empirical risk function
\begin{align}
R_{n, n_e, n_r}(u_{\theta}) &= \frac{\lambda_d}{n}\sum_{i=1}^{n} \|u_\theta(\bX_i)-Y_i\|_2^2 + \frac{\lambda_e}{n_e}\sum_{j=1}^{n_e} \|u_\theta(\bX^{(e)}_j)-h(\bX^{(e)}_j)\|_2^2 \nonumber\\
& \quad + \frac{1}{n_r}\sum_{k=1}^M \sum_{\ell=1}^{n_r}  \mathscr{F}_k(u_\theta, \bX^{(r)}_\ell)^2\label{lossgenerale}
\end{align}
over the class $\text{NN}_H(D):=\{u_\theta, \theta \in \Theta_{H,D}\}$ 
of feedforward neural networks with $H$ hidden layers of common width $D$ (see below for a precise definition), where $(\lambda_d, \lambda_e) \in \mathbb{R}_+^2\backslash (0,0)$ are hyperparameters that establish a tradeoff between the three terms. 
In practice, one often encounters the case where $\lambda_e=0$ (data + PDEs).
Another situation of interest is when $\lambda_d=0$ (PDEs + boundary/initial conditions), which corresponds to the special case of a PDE solver. 
Setting \eqref{lossgenerale} is more general as it includes all the combinations data + PDEs + boundary/initial conditions. 
Since a minimizer of the empirical risk function \eqref{lossgenerale} does not necessarily exist, we denote by $(\hat \theta(p, n_e, n_r, D))_{p\in \mathbb{N}} \in \Theta_{H,D}^\mathbb{N}$ any minimizing sequence, i.e.,
\[ 
\lim_{p \to \infty}R_{n, n_e, n_r}(u_{\hat \theta(p, n_e, n_r, D)}) = \inf_{\theta \in \Theta_{H,D}}\,R_{n, n_e, n_r}(u_\theta).
\]
In practice, such a sequence is usually obtained by implementing some optimization procedure, the exact description of which is not important for our purpose.

On the practical side, simulations using hybrid modeling have been successfully applied to model image denoising \citep{wang2020superResolution},
turbulence \citep{wang2020turbulence},
blood streams \citep{arzani2021uncovering},
wave propagation \citep{davini2021using},
and ocean streams \citep{wolff2021ocean}. Experiments with real data have been performed to assess 
the sea temperature \citep{bezenac2017processes},
subsurface transport \citep{he2020subsurface},
fused filament fabrication \citep{kapusuzoglu2020manufacturing},
seismic response \citep{zhang2020seismic},
glacier dynamic \citep{riel2021glacier}, 
lake temperature \citep{daw2022lake},
thermal modeling of buildings \citep{gokhale2022thermal},  
 blasts \citep{pannell2022blast}, and
 heat transfers \citep{ramezankhani2022multifidelity}.
 The generality and flexibility of the empirical risk function \eqref{lossgenerale} allows it to encompass most PINN-like problems. For example, the case $M \geqslant 2$ is considered in \citet{bezenac2017processes} and \citet{riel2021glacier}, while \citet{zhang2020seismic} and \citet{wang2020turbulence} assume that $d_1 = d_2 = 3$. 
Importantly, the situation where $\lambda_d > 0$ {\it and} $\lambda_e>0$ (data + boundary conditions + PDEs) is also interesting from a physical point of view. This is, for example, the approach advocated by \citet{arzani2021uncovering}, which uses both data and boundary conditions (see also \citealp{cuomo2022scientific}, and \citealp{Hao2022review}).

\smallskip
\paragraph{The PDE solver case} The particular case $\lambda_d=0$ deserves a special comment. 
In this setting, without physical measures $(\bX_i, Y_i)$, the function $u^{\star}$ is viewed as the unknown solution of the system of PDEs $\mathscr{F}_1, \hdots, \mathscr{F}_M$ with boundary/initial conditions $h$. 
The goal is to estimate the solution $u^{\star}$ of the PDE problem
\[
    \left\{\begin{array}{lrcl}
      \forall k, \, \forall \bx \in \Omega, & \mathscr{F}_k(u^{\star},\bx) &=& 0 \\
      \forall \bx \in E, & u^{\star}(\bx) &=& h(\bx),
\end{array}\right.
\]
with neural networks from $\mathrm{NN}_H(D)$.
In this case, the empirical risk function \eqref{lossgenerale} becomes 
\[
R_{n_e, n_r}(u_\theta) = \frac{\lambda_e}{n_e}\sum_{j=1}^{n_e} \|u_\theta(\bX^{(e)}_j)-h(\bX^{(e)}_j)\|_2^2  + \frac{1}{n_r}{\sum_{k=1}^M}\sum_{\ell=1}^{n_r} \mathscr{F}_k(u_\theta, \bX^{(r)}_\ell)^2,
\]
where the boundary and initial conditions $(\bX^{(e)}_1,h(\bX^{(e)}_1)), \hdots, (\bX^{(e)}_{n_e},h(\bX^{(e)}_{n_e}))$ are sampled on $E\times \mathbb{R}^{d_2}$ according to some known distribution $\mu_E$, and $(\bX^{(r)}_1, \hdots, \bX^{(r)}_{n_r})$ are uniformly distributed on $\Omega$.
Note that, for simplicity, we write $R_{n_e,n_r}(u_{\theta})$ instead of $R_{n, n_e,n_r}(u_{\theta})$ because no $\bX_i$ is involved in this context. Since no confusion is possible, the same convention is used for all subsequent risk functions throughout the paper. 
The first term of $R_{n_e,n_r}(u_{\theta})$ measures the gap between the network $u_\theta$ and the condition function $h$ on $E$, while the second term forces $u_{\theta}$ to obey the PDE in a discretized way. 
Since both the condition function $h$ and the distribution $\mu_E$ are known, it is reasonable to think of $n_e$ and $n_r$ as large (up to the computational resources).
In this scientific computing perspective, PINNs have been successfully applied to solve a wide variety of linear and nonlinear problems, including motion, advection, heat, Euler, high-frequency Helmholtz, Schr\"odinger, Blasius, Burgers, and Navier-Stokes equations, covering various fields ranging from classical (mechanics, fluid dynamics, thermodynamics, and electromagnetism) to quantum physics \citep[e.g.,][]{cuomo2022scientific, li2023aphysics}.

\smallskip
\paragraph{The class of neural networks}  
A fully-connected feedforward neural network with $H\in\mathbb{N}^\star$ hidden layers of sizes $(L_1, \hdots, L_H) :=(D,\hdots ,  D) \in (\mathbb{N}^\star)^H$ and activation $\tanh$, is a function from $\mathbb R^{d_1}$ to $\mathbb R^{d_2}$, defined by
\begin{equation*}
    u_{\theta} = \mathcal{A}_{H+1}\circ (\tanh \circ \mathcal{A}_H) \circ \cdots \circ (\tanh \circ \mathcal{A}_1),
\end{equation*}
where the hyperbolic tangent function $\tanh$ is applied element-wise.  Each $\mathcal{A}_k : \mathbb{R}^{L_{k-1}} \rightarrow \mathbb{R}^{L_{k}}$ is an affine function of the form $\mathcal{A}_k(\bx) = W_k \bx + b_k$, with $W_k$ a ($L_{k-1} \times L_k$)-matrix,  $b_k \in \mathbb{R}^{L_k}$ a vector, 
$L_0 = d_1$, and $L_{H+1} = d_2$.
The neural network $u_{\theta}$ is parameterized by $\theta = (W_1, b_1, \hdots, W_{H+1}, b_{H+1}) \in \Theta_{H,D}$, where $\Theta_{H,D}=\mathbb{R}^{\sum_{i=0}^H (L_i+1) \times L_{i+1}}$.
Throughout, we let $\text{NN}_H(D)=\{u_{\theta}, \, \, \theta \in \Theta_{H,D}\}$.
We emphasize that the $\tanh$ function is the most common activation in PINNs \citep[see, e.g.,][]{cuomo2022scientific}.
It is preferable to the classical $\text{ReLU }(x)=\max(x,0)$ activation. In fact, since ReLU neural networks are a subset of piecewise linear functions, their high derivatives vanish and therefore cannot be captured by the penalty term $\frac{1}{n_r}\sum_{k=1}^M \sum_{\ell=1}^{n_r} \mathscr{F}_k(u_\theta, \bX^{(r)}_\ell)^2$.

The parameter space $\text{NN}_H(D)$ must be chosen large enough to approximate both the solutions of the PDEs and their derivatives. This property is encapsulated in Proposition \ref{prop:densite}, which shows that for any number $H \geqslant 2$ of hidden layers, the set $\text{NN}_H:=\cup_{D}\text{NN}_H(D)$ is dense in the space $(C^\infty(\bar{\Omega}, \mathbb{R}^{d_2}), \|\cdot\|_{C^K(\Omega)})$. This generalizes
Theorem 5.1 in \citet{ryck2021approximation} which states that $\text{NN}_2$ is dense in $(C^\infty([0,1]^{d_1}, \mathbb{R}), \|\cdot\|_{C^K(]0,1[^{d_1})})$ for all $d_1 \geqslant 1$ and $K \in \mathbb{N}$. 
\begin{prop}[Density of neural networks in Hölder spaces]
\label{prop:densite}
Let $K \in \mathbb{N}$, $H\geqslant 2$, and $\Omega \subseteq \mathbb{R}^{d_1}$ be a bounded Lipschitz domain. Then $\mathrm{NN}_H:=\cup_{D}\mathrm{NN}_H(D)$ is dense in $(C^\infty(\bar{\Omega}, \mathbb{R}^{d_2}), \|\cdot\|_{C^K(\Omega)})$, i.e., for any function $u\in C^\infty(\bar{\Omega}, \mathbb{R}^{d_2})$, there exists a sequence $(u_{p})_{p\in \mathbb{N}}\in \mathrm{NN}_H^{\mathbb{N}}$ such that $\lim_{p \to \infty} \|u-u_p\|_{C^K(\Omega)} = 0$. 
\end{prop}
In the remainder of the article, the number $H$ of hidden layers is considered to be fixed. 
\citet{krishnapriyan2021characterizing} use $\text{NN}_4(50)$, \citet{xu2021extrapolation} take $\text{NN}_5(100)$, whereas \citet{arzani2021uncovering} employ $\text{NN}_{10}(100)$. 
It is worth noting that in this series of papers the width $D$ is much larger than $H$, as in Proposition \ref{prop:densite}.

\section{PINNs can overfit}
\label{POF}
Our goal in this section is to show through two examples how learning with standard PINNs can lead to severe overfitting problems. This weakness has already been noted in \citet{costabal2020physics},  \citet{nabian2019engineering},  \citet{chandrajit2023recipes}, and \citet{esfahani2023adatadriven}, which propose to improve the performance of their models by resorting to an additional regularization strategy. The pathological cases that we highlight both rely on neural networks with exploding derivatives. 

The theoretical risk function is defined by 
\begin{equation}
\mathscr{R}_n(u) = \frac{\lambda_d}{n} \sum_{i=1}^n \|u(\bX_i) - Y_i\|_2^2 + \lambda_e \mathbb{E}\|u(\bX^{(e)})-h(\bX^{(e)})\|_2^2 + \frac{1}{|\Omega|}\sum_{k=1}^M \int_\Omega \mathscr{F}_k(u, \bx)^2 d\bx.\label{lossTheorical}
\end{equation}
Observe that in $\mathscr{R}_n(u)$ we take expectation with respect to $\mu_E$ (for the boundary/initial condition part) and integrate with respect to the uniform measure on $\Omega$ (for the PDE part), but keep the term $\sum_{i=1}^n \|u_{{\theta}}(\bX_i) - Y_i\|_2^2$ intact. 
This regime corresponds to the limit of the em\-pi\-ri\-cal risk function \eqref{lossgenerale}, holding $n$ fixed and letting  $n_e, n_r \to \infty$. The rationale is that while the random samples $(\bX_i,Y_i)$ may be limited in number (e.g., because their acquisition is more delicate and require physical measurements), this is not the case for $\bX^{(e)}_j$ or $\bX^{(r)}_j$, which can be freely sampled (up to computational resources).
Note however that in the PDE solver setting, the first term is not included. 

Given any minimizing sequence $(\hat{\theta}(p, n_e,n_r, D))_{p \in \mathbb{N}}$ of the empirical risk, satisfying
\[\lim_{p\to \infty} R_{n, n_e,n_r}(u_{\hat{\theta}(p, n_e,n_r, D)}) = \inf_{\theta \in \Theta_{H,D}}R_{n, n_e,n_r}(u_\theta),
\] 
a natural requirement, called risk-consistency, is that 
\[
\lim_{n_e,n_r \to \infty}\lim_{p \to \infty}
\mathscr{R}_n(u_{\hat{\theta}(p, n_e,n_r, D)}) = \inf_{u \in \text{NN}_H(D)} \mathscr{R}_n(u).
\]
We show below that standard PINNs can dramatically fail to be risk-consistent, through two counterexamples, one in the hybrid modeling context and one in the specific PDE solver setting.

\smallskip
\paragraph{The case of dynamics with friction} 
Consider the following ordinary differential constraint, defined on the domain $\Omega = ]0, T[$ (with closure ${\bar{\Omega}} = [0,T]$) by
\begin{align}
\label{eq:friction_constraint}
\forall u\in C^2({\bar \Omega},\mathbb{R}), \ \forall \bx \in \Omega, \quad \mathscr{F}(u, \bx) &= mu''(\bx) + \gamma u'(\bx).
\end{align} 
This models the dynamics of an object of mass $m > 0$, subjected to a fluid force of friction coefficient $\gamma >0$. 
The goal is to reconstruct the real trajectory $u^\star$ by taking advantage of the model $\mathscr{F}$ and the noisy observations $Y_i$ at the $\bX_i$. This is an example where the modeling is perfect, i.e., $\mathscr{F}(u^\star, \cdot) = 0$, but the challenge is that the physical model is incomplete because the boundary conditions are unknown.
Following the hybrid modeling framework, the trajectory $u^{\star}$ is estimated by minimizing over the space $\text{NN}_H(D)$ the empirical risk function
\[
R_{n, n_r}(u_\theta) = \frac{\lambda_d}{n}\sum_{i=1}^{n} |u_\theta(\bX_i)-Y_i|^2 + \frac{1}{n_r}\sum_{\ell=1}^{n_r} \mathscr{F}(u_\theta, \bX^{(r)}_\ell)^2.
\] 
\begin{prop}[Overfitting] 
\label{prop:friction}
    Consider the dynamics with friction model \eqref{eq:friction_constraint}, and assume that there are two observations such that $Y_i \neq Y_j$. Then, whenever $D \geqslant n-1$, for any integer $n_r$, for all $\bX^{(r)}_1, \hdots, \bX^{(r)}_{n_r}$, there exists a minimizing sequence $(u_{\hat{\theta}(p, n_r, D)})_{p\in \mathbb{N}} \in \mathrm{NN}_H(D)^{\mathbb{N}}$ such that 
    $\lim_{p \to \infty} R_{n,n_r}(u_{\hat{\theta}(p, n_r, D)}) = 0$ but $ \lim_{p \to \infty} \mathscr{R}_n(u_{\hat{\theta}(p, n_r, D)}) = \infty$. So, this PINN estimator is not consistent.  
\end{prop}
\begin{figure}
    \centering
    \includegraphics[width = 0.7\textwidth]{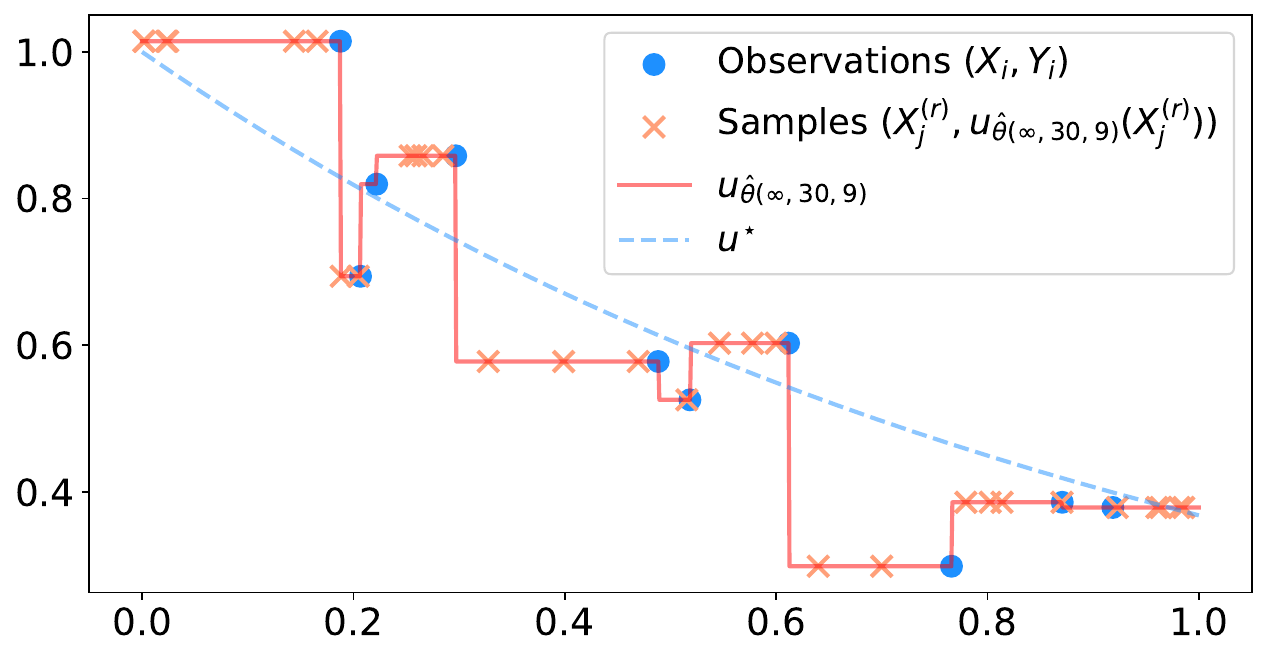}
    \caption{An inconsistent PINN estimator in hybrid modeling with $m = \gamma = 1$, $\varepsilon \sim \mathcal{N}(0, 10^{-2})$, and $n=10$. }
    \label{fig:failure_hybrid}
\end{figure}
Proposition \ref{prop:friction} illustrates how fitting a PINN by minimizing the empirical risk alone can lead to a catastrophic situation, where the empirical risk of the minimizing sequence is (close to) zero, while its theoretical risk is infinite. This phenomenon is explained by the existence of piecewise constant functions interpolating the observations $\bX_1,\hdots, \bX_n$, whose derivatives are null at the points $\bX^{(r)}_1, \hdots, \bX^{(r)}_{n_r}$, but diverge between these points (see Figure \ref{fig:failure_hybrid}). These functions correspond to neural networks $u_\theta$ such that $\|\theta\|_2 \to \infty$.

\smallskip
\paragraph{PDE solver: The heat propagation case} Consider the heat propagation differential operator defined on the domain $\Omega = ]-1, 1[ \times ]0,T[$ (with closure ${\bar \Omega} = [-1, 1] \times [0,T]$) by 
\begin{equation}
    \label{eq:wave_constraint}
    \forall  u \in C^2({\bar \Omega}, \mathbb{R}), \ \forall \bx \in \Omega, \quad \mathscr{F}(u,\bx) = \partial_{t}u(\bx) - \partial^2_{x,x}u(\bx),
\end{equation}
associated with the boundary conditions
\[\forall t \in [0,T], \quad u(-1, t) = u(1, t) = 0,\]
and the initial condition defined, for all $x \in [-1, 1]$, by
\[u(x,0) = \tanh^{\circ H}(x+0.5)-\tanh^{\circ H}(x-0.5) + \tanh^{\circ H}(0.5) - \tanh^{\circ H}(1.5). \]
The notation $\tanh^{\circ k}$ stands for the function recursively defined by $\tanh^{\circ 1} = \tanh$ and $\tanh^{\circ (k+1)} = \tanh \circ \tanh^{\circ k}$.  The unique solution $u^\star$ of the PDE is shown in Figure \ref{fig:failure_PDEsolving} (right). It models the time evolution of the temperature of a wire, whose extremities at $x=-1$ and $x=1$ are maintained at zero temperature. Note that the initial condition corresponds to a bell-shaped function, which belongs to $\text{NN}_H(2)$.
However, the setting can be extended to arbitrary initial conditions that take the form of a neural network function, given the boundary condition $u(\partial \Omega \times [0,T]) = \{0\}$. 

To solve the PDE \eqref{eq:wave_constraint}, we use $n_e$ i.i.d.~samples $\bX^{(e)}_1, \hdots, \bX^{(e)}_{n_e}$ on  $E = ([-1,1]\times\{0\}) \cup (\{-1, 1\}\times [0,T])$, distributed according to $\mu_E$, together with $n_r$ i.i.d.~samples $\bX^{(r)}_1, \hdots, \bX^{(r)}_{n_r}$, uniformly distributed on $\Omega$. Let $(\hat{\theta}(p, n_e, n_r, D))_{p \in \mathbb{N}}$ be a sequence of parameters  minimizing the empirical risk function
\[
R_{n_e, n_r}(u_\theta) = \frac{\lambda_e}{n_e}\sum_{j=1}^{n_e} |u_\theta(\bX^{(e)}_j)-h(\bX^{(e)}_j)|^2  + \frac{1}{n_r}\sum_{\ell=1}^{n_r} \mathscr{F}(u_\theta, \bX^{(r)}_\ell)^2,
\] 
over the space $\text{NN}_H(D)$. The theoretical counterpart of this empirical risk is
\[\mathscr{R}(u) =  \lambda_e \mathbb{E}|u(\bX^{(e)})-h(\bX^{(e)})|^2 + \frac{1}{|\Omega|}\int_\Omega \mathscr{F}(u, \bx)^2 d\bx.  \] 

\begin{prop}[PDE solver overfitting]
\label{prop:wave}
Consider the heat propagation model \eqref{eq:wave_constraint}. Then, 
 whenever $D \geqslant 4$, for any pair $(n_e, n_r)$, for all $\bX^{(e)}_1, \hdots, \bX^{(e)}_{n_e}$ and for all $\bX^{(r)}_1, \hdots, \bX^{(r)}_{n_r}$, there exists a minimizing sequence $(u_{\hat{\theta}(p, n_e, n_r, D)})_{p \in \mathbb{N}} \in \mathrm{NN}_H(D)^{\mathbb{N}}$ such that $\lim_{p\to \infty}R_{n_e, n_r}(u_{\hat{\theta}(p, n_e, n_r, D)}) = 0$ but $\lim_{p \to \infty}\mathscr{R}(u_{\hat{\theta}(p, n_e, n_r, D)}) = \infty$. So, this PINN estimator is not consistent. 
\end{prop}
\begin{figure}
    \includegraphics[width = \textwidth]{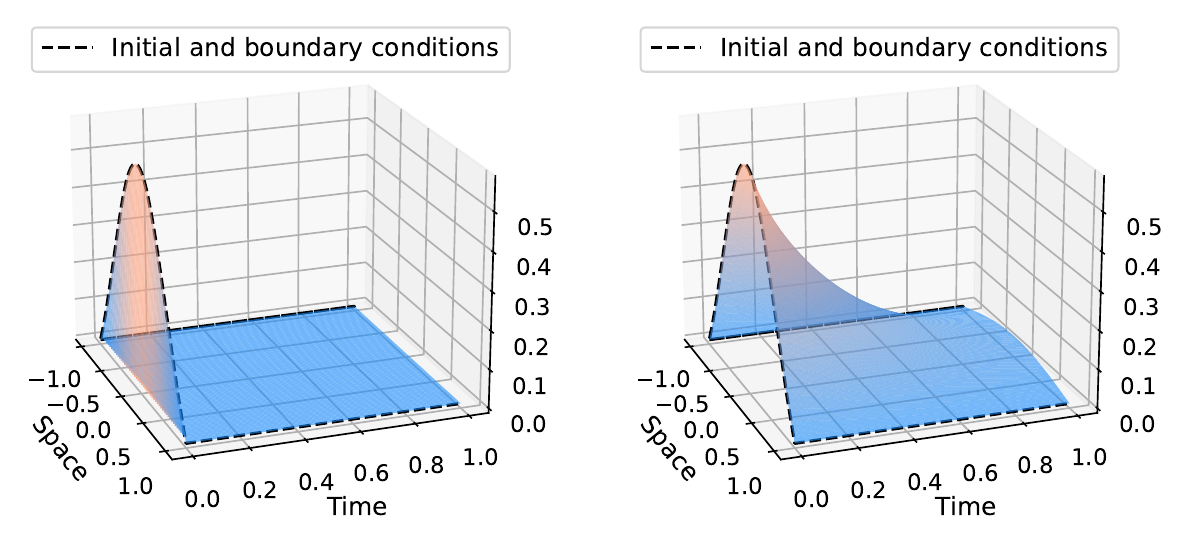}
    \caption{Inconsistent PINN (left) compared to the solution $u^\star$ of the PDE (right) for the heat propagation case.}
    \label{fig:failure_PDEsolving}
\end{figure}
Figure \ref{fig:failure_PDEsolving} (left) shows an example of an inconsistent PINN estimator.
Such an estimator corresponds to a function that equals zero on $\Omega$ (and thus satisfies the linear PDE), while satisfying the  initial condition on $\partial \Omega$. This function corresponds to a limit of neural networks $u_\theta$ such that $\|\theta\|_2\to \infty$.

The proof strategy of Propositions \ref{prop:friction} and \ref{prop:wave} does not depend on the geometry of the points $\bX^{(r)}$ and the points $\bX^{(e)}$, which could therefore be sampled along a grid, or by any quasi Monte Carlo method.
We emphasize that the two negative examples of Propositions \ref{prop:friction} and \ref{prop:wave} are no exceptions. In fact, their proofs can be easily generalized to differential operators $\mathscr{F}$ such that the following property holds: for all $\bx \in \Omega$, for all $u \in C^\infty(\Omega, \mathbb{R}^{d_2})$, if $\nabla u$ vanishes on an open set containing $\bx$, then $\mathscr{F}(u, \bx) = 0$.
This property is satisfied in the case of motion with friction, advection, heat, wave propagation, Schr\"odinger, Maxwell and Navier-Stokes equations, which are so as many cases that will suffer from overfitting.

\section{Consistency of regularized PINNs for linear and nonlinear PDE systems}
\label{sec:consistency}
Training PINNs can be tricky because it can lead to the type of pathological situations highlighted in Section \ref{POF}. To avoid such an overfitting behavior, a standard approach in machine learning is to resort to ridge regularization, where the empirical risk to be minimized is penalized by the $L^2$ norm of the parameters $\theta$. This technique has been shown to improve not only the optimization convergence 
during the training phase, but also the generalization ability of the resulting predictor \citep{krogh1991a, guo2017on}. Ridge regularization is available in most deep learning libraries (e.g., \texttt{pytorch} or \texttt{keras}),  where it is implemented using the so-called weight decay \citep{loshchilov2019decoupled}. 
Interestingly, the ridge re\-gu\-la\-ri\-za\-tion of a slight modification of PINNs, using adaptive activation functions, has been studied in \citet{jagtap2020adaptive}, which shows that gradient descent algorithms manage to generate an effective minimizing sequence of the penalized empirical risk. In this section, we formalize ridge PINNs and study their risk-consistency. 
\begin{defi}[Ridge PINNs]
    The ridge risk function is defined by
\begin{equation}
    R_{n, n_e,n_r}^{(\mathrm{ridge)}}(u_\theta) = R_{n, n_e,n_r}(u_\theta) + \lambda_{(\mathrm{ridge})} \|\theta\|_2^2, \label{eq:regPINN}
\end{equation}
where $\lambda_{(\mathrm{ridge})} >0$ is the ridge hyperparameter. We denote by $(\hat \theta_{(p, n_e, n_r, D)}^{(\mathrm{ridge})})_{p\in\mathbb{N}}$ a minimizing sequence of this risk, i.e., 
\[ 
\lim_{p \to \infty}R^{(\mathrm{ridge)}}_{n, n_e, n_r}(u_{\hat \theta_{(p, n_e, n_r, D)}^{(\mathrm{ridge})}}) = \inf_{\theta \in \Theta}\,R^{(\mathrm{ridge})}_{n, n_e, n_r}(u_\theta).
\]
\end{defi}

Our next Proposition \ref{prop:bounding} states that the $L^2$ norm of the parameters $\theta$ bounds the Hölder norm of the neural network $u_\theta$. This result is interesting in itself because it establishes a connection between the $L^2$ norm of a fully connected neural network and its regularity. (Note that, by equivalence of the norms, this result also holds if the ridge penalty is replaced by $\|\theta\|_p^p$.) In the present paper it plays a key role in the risk-consistency analysis.

\begin{prop}[Bounding the norm of a neural network by the norm of its parameter]
\label{prop:bounding}
    Consider the class $\mathrm{NN}_H(D)=\{u_\theta, \theta\in\Theta_{H,D}\}$. Let $K \in \mathbb{N}$.
    Then there exists a constant $C_{K,H} >0$, depending only on $K$ and $H$,
    such that, for all $ \theta \in \Theta_{H, D}$,
    \[
    \|u_\theta\|_{C^K(\mathbb{R}^{d_1})} \leqslant C_{K,H}(D+1)^{HK+1}(1+\|\theta\|_2)^{HK}\|\theta\|_2.
    \]
    Moreover, this bound is tight with respect to $\|\theta\|_2$, in the sense that, for all $H,D \geqslant 1$ and all $K\in \mathbb{N}$, there exists a sequence $(\theta_p)_{p \in \mathbb{N}} \in \mathrm{NN}_H(D)$ and a constant $\bar{C}_{K,H}>0$ such that $(i)$ $\lim_{p\to\infty}\|\theta_p\|_2 = \infty$ and $(ii)$ $\|u_{\theta_p}\|_{C^K(\mathbb{R}^{d_1})} \geqslant \bar{C}_{K,H}\|\theta_p\|_2^{HK+1}.$
\end{prop}
In order to study the generalization capabilities of regularized PINNs, we need to restrict the PDEs to a class of smooth differential operators, which we call polynomial operators (Definition \ref{defi:polyop} below). This class includes the most common PDE systems, as shown in the following example with the Navier-Stokes equations. 

\begin{ex}[Navier-Stokes equations]
\label{ex:navierS} Let $\Omega=\Omega_1 \times ]0,T[$, 
where $\Omega_1 \subseteq \mathbb{R}^{3}$ is a bounded Lipschitz domain and $T\geqslant 0$ is a finite time horizon.
The incompressible Navier-Stokes system of equations is defined for all $u = (u_x,u_y,u_z,p)\in C^2(\bar{\Omega}, \mathbb{R}^4)$ and for all $ \bx = (x,y,z,t) \in \Omega,$ by
\[ \left\{\begin{array}{rcl}
\mathscr{F}_1(u, \bx) &= &\partial_t u_x - (u_x \partial_x + u_y \partial_y + u_z \partial_z) u_x - \eta (\partial^2_{x,x}+\partial^2_{y,y}+\partial^2_{z,z}) u_x+ \rho^{-1} \partial_x p   \\
\mathscr{F}_2(u, \bx) &=& \partial_t u_y - (u_x \partial_x + u_y \partial_y + u_z \partial_z) u_y - \eta (\partial^2_{x,x}+\partial^2_{y,y}+\partial^2_{z,z}) u_y+ \rho^{-1} \partial_y p \\
         \mathscr{F}_3(u, \bx) &= &\partial_t u_z - (u_x \partial_x + u_y \partial_y + u_z \partial_z) u_z - \eta (\partial^2_{x,x}+\partial^2_{y,y}+\partial^2_{z,z}) u_z+ \rho^{-1} \partial_z p+g(\bx)\\
         \mathscr{F}_4(u, \bx) &= &\partial_x u_x + \partial_y u_y + \partial_z u_z,
    \end{array}\right. \]
where $\eta, \rho >0$ and $g\in C^{\infty}(\bar{\Omega}, \mathbb{R})$. Observe that $\mathscr{F}_1, \mathscr{F}_2, \mathscr{F}_3$, and $\mathscr{F}_4$ are polynomials in $u$ and its derivatives, with coefficients in $C^\infty(\bar{\Omega}, \mathbb{R})$. For example, $\mathscr{F}_3(u,\bx) = P_3(u_x, u_y, u_z, \partial_x u_z, \partial_y u_z, 
\partial_z u_z,$ $ \partial_t u_z, \partial^2_{x,x} u_z, \partial^2_{y,y} u_z, \partial^2_{z,z} u_z, \partial_z p)(\bx)$, where the polynomial $P_3\in C^\infty(\bar \Omega, \mathbb{R})[Z_1, \hdots, Z_{11}]$ is defined by $P_3(Z_1, \hdots, Z_{11}) = Z_7- Z_1Z_4 - Z_2Z_5 - Z_3Z_6- \eta (Z_8+Z_9+Z_{10}) + \rho^{-1}Z_{11} + g$.
\end{ex}

The above example can be generalized with the following definition. 
\begin{defi}[Polynomial operator]
\label{defi:polyop}
    An operator $\mathscr{F} : C^K(\bar{\Omega}, \mathbb{R}^{d_2})\times \Omega \to \mathbb{R}$  is  a polynomial operator
    of order $K \in \mathbb{N}$ if there exists an integer $s \in\mathbb{N}$ and multi-indexes $(\alpha_{i,j})_{1\leqslant i\leqslant d_2, 1\leqslant j\leqslant s} \in (\mathbb{N}^{d_1})^{sd_2}$ such that 
    \[\forall u=(u_1,\hdots,u_{d_2}) \in C^K(\bar{\Omega}, \mathbb{R}^{d_2}), \quad \mathscr{F}(u, \cdot) = P((\partial^{\alpha_{i,j}}u_i)_{1\leqslant i\leqslant d_2, 1\leqslant j\leqslant s}),\]
    where $P \in C^\infty(\bar{\Omega}, \mathbb{R})[Z_{1,1}, \hdots, Z_{d_2,s}]$ is a polynomial with smooth coefficients. 
\end{defi}
In other words, $\mathscr{F}$ is a polynomial operator if it is of the form
\[\mathscr{F}(u, \bx) = \sum_{k=1}^{N(P)} \phi_k \times \prod_{i=1}^{d_2}\prod_{j=1}^s (\partial^{\alpha_{i,j}}u_i(\bx))^{I(i,j,k)}, \]
where $N(P) \in \mathbb{N}^\star$, $\phi_k \in C^\infty(\bar{\Omega}, \mathbb{R})$, and $I(i,j,k) \in \mathbb{N}$. 
And the associated polynomial is $P(Z_{1,1}, \hdots, Z_{d_2, s})$ $= \sum_{k=1}^{N(P)} \phi_k \times \prod_{i=1}^{d_2}\prod_{j=1}^s Z_{i,j}^{I(i,j,k)}$ (recall that $\partial^\alpha u_i = u_i$ when $\alpha = 0$). 
\begin{defi}[Degree] \label{defi:deg}
    The degree of the polynomial operator $\mathscr{F}$ is \[\mathrm{deg}(\mathscr{F}) = \max_{1\leqslant k \leqslant N(P)}\sum_{i=1}^{d_2}\sum_{j=1}^{s} (1+|\alpha_{i,j}|)I(i,j,k).\] 
\end{defi}
As an illustration, in Example \ref{ex:navierS}, one has $\deg(\mathscr{F}_3) = 3$, and this degree is reached in both the terms  $u_z\partial_z u_z$ and $\partial^2_{z,z}u_z$. 
 Note that $\deg(P_3) = 2$ but $\deg(\mathscr{F}_3) = 3$.
 To compute $\deg(\mathscr{F}_3)$, we first count the number of terms in each monomial ($u_z\partial_zu_z$ has two terms while $\partial^2_{z,z}u_z$ has one term), which is $\sum_{i=1}^{d_2}\sum_{j=1}^{s} I(i,j,k)$ for the $k$th monomial, and add the number of derivatives involved in the product ($u_z\partial_zu_z$ contains a single $\partial_z$ operator while $\partial^2_{z,z}u_z$ contains two derivatives in $\partial_z$), which corresponds to $\sum_{i=1}^{d_2}\sum_{j=1}^{s} |\alpha_{i,j}|I(i,j,k)$ for the $k$th monomial. Thus, for each monomial $k$, the total sum is $\sum_{i=1}^{d_2}\sum_{j=1}^{s} (1+|\alpha_{i,j}|)I(i,j,k)$.

We emphasize that this class includes a large number of PDEs, such as linear PDEs (e.g.,  advection, heat, and Maxwell equations), as well as some nonlinear PDEs (e.g., Blasius, Burger's, and Navier-Stokes equations). Proposition \ref{prop:bounding} is a key ingredient to uniformly bound the risk of PINNs involving polynomial PDE operators 
\supplementSecFive.
This in turn can be used to establish the risk-consistency of these PINNs when $n_e$ and $n_r$ tend to $\infty$, 
as follows.

\begin{thm}[Risk-consistency of ridge PINNs]
\label{thm:generalization_error}
Consider the ridge PINN problem \eqref{eq:regPINN}, over the  class $\mathrm{NN}_H(D)=\{u_\theta, \theta\in\Theta_{H,D}\}$,  where $H \geqslant 2$. 
Assume that the condition function $h$ is Lipschitz and that $\mathscr{F}_1, \hdots, \mathscr{F}_M$ are polynomial operators. 
Assume, in addition, that the ridge parameter is of the form
\[
\lambda_{(\mathrm{ridge})} = \min(n_e, n_r)^{-\kappa}, \quad \text{where} \quad  \kappa=\frac{1}{12+4H(1+(2+H)\max_k\deg(\mathscr{F}_k))}.
\] 

Then, almost surely, 
\[
\lim_{n_e, n_r \to \infty} \lim_{p\to \infty}\mathscr{R}_n(u_{\hat{\theta}^{(\mathrm{ridge})}(p, n_e, n_r, D)}) =  \inf_{u \in \mathrm{NN}_H(D)} \mathscr{R}_n(u).
\] 
\end{thm}
Thus, minimizing the ridge empirical risk \eqref{eq:regPINN} over $\Theta_{H,D}$ amounts to minimizing the theoretical risk \eqref{lossTheorical} over $\Theta_{H,D}$ in the asymptotic regime $n_e, n_r \to \infty$.
This fundamental result is complemented by the following one, which resorts to another asymptotics in the width $D$. This ensures that the choice of the neural architecture $\mathrm{NN}_H \subseteq C^\infty(\bar{\Omega}, \mathbb{R}^{d_2})$ does not introduce any asymptotic bias. 

\begin{thm}[The ridge PINN is asymptotically unbiased] 
    \label{thm:approximation}
    Under the same assumptions as in Theorem \ref{thm:generalization_error}, one has, almost surely, 
    \[
    \lim_{D \to \infty} \lim_{n_e, n_r \to \infty} \lim_{p\to \infty}\mathscr{R}_n(u_{\hat{\theta}^{(\mathrm{ridge})}(p, n_e, n_r, D)}) =  \inf_{u \in C^{\infty}(\bar{\Omega}, \mathbb{R}^{d_2})} \mathscr{R}_n(u).
    \]
\end{thm}
In other words, minimizing the ridge empirical risk over $\Theta_{H,D}$ and letting $D, n_e, n_r \to \infty$ amounts to minimizing the theoretical risk \eqref{lossTheorical} over the entire class $C^\infty(\bar{\Omega}, \mathbb{R}^{d_2})$.
We emphasize that these two theorems hold independently of the values of the hyperparameters $\lambda_d, \lambda_e \geqslant 0$. 
 Therefore, our results cover the general hybrid modeling framework \eqref{lossgenerale}, which includes the PDE solver. 
To the best of our knowledge, these are the first results that provide theoretical guarantees for PINNs regularized with a standard penalty. 
They complement the state-of-the-art approaches of  \citet{shin2020convergence}, \citet{shin2020errorEstimates}, \citet{mishra2022generalization}, and  \citet{wu2022convergence}, which consider re\-gu\-la\-ri\-za\-tion strategies that are unfortunately not feasible in practice.

It is worth noting that Theorem \ref{thm:approximation} still holds by choosing $D$ as a function of $n_e$ and $n_r$. 
In fact, an easy modification of the proofs reveals that one can take $D(n_e,n_r) = \min(n_e, n_r)^{\xi}$, where $\xi$ is a constant depending only on $H$ and $\max_k\deg(\mathscr{F}_k)$. Thus, in this setting, 
\[
\lim_{n_e, n_r \to \infty} \lim_{p\to \infty}\mathscr{R}_n(u_{\hat{\theta}^{(\mathrm{ridge})}(p, n_e, n_r, D(n_e,n_r)}) =  \inf_{u \in C^{\infty}(\bar{\Omega}, \mathbb{R}^{d_2})} \mathscr{R}_n(u).
\]

\begin{remark}[Dirichlet boundary conditions] Theorems \ref{thm:generalization_error} and \ref{thm:approximation} can be easily adapted to PINNs with Von Neumann conditions instead of Dirichlet boundary conditions. This is achieved by substituting the term $n_e^{-1}\sum_{j=1}^{n_e} \|u_\theta(\bX^{(e)}_j)-h(\bX^{(e)}_j)\|_2^2$ in the PINN definition \eqref{lossgenerale} by $n_e^{-1}\sum_{j=1}^{n_e}\|\partial_{\overrightarrow{n}}u_\theta(\bX^{(e)}_j)\|_2^2$, where $\overrightarrow{n}$ is the normal to $\partial \Omega$.
\end{remark}

\smallskip
\paragraph{Practical considerations}
The decay rate of $\lambda_{(\mathrm{ridge})} = \min(n_e, n_r)^{-\kappa}$ does not depend on the dimension $d_1$ of $\Omega$. This is consistent with the results of \citet{karniadakis2021piml} and \citet{ryck2022kolmogorov}, which suggest that PINNs can overcome the curse of dimensio\-nality, opening up interesting perspectives for efficient solvers of high-dimensional PDEs.
We also emphasize that $\lambda_{(\mathrm{ridge})}$ depends only on the degree of the polynomial PDE operator,  the depth $H$, and the sample sizes $n_e$ and $n_r$.
All these quantities are known, which makes this hyperparameter immediately useful for practical applications.    
For example, in Navier-Stokes equations of Example \ref{ex:navierS}, one has $\max_k \deg(\mathscr{F}_k) = 3$. 
Thus, for a neural network of depth, say $H=2$, the ridge hyperparameter $\lambda_{(\mathrm{ridge})} = \min(n_e, n_r)^{-1/116}$ is sufficient to ensure consistency. 
It is also interesting to note that the bound on $\lambda_{(\mathrm{ridge})}$ in the theorems deteriorates with increasing  depth $H$. 
This confirms the preferential use of shallow neural networks in the experimental works of \citet{arzani2021uncovering}, \citet{karniadakis2021piml}, and \citet{xu2021extrapolation}. 
The bound also deteriorates as $\max_k \deg \mathscr{F}_k$ increases. 
This is in line with the empirical results of \citet{davini2021using}, which was able to improve the performance of PINNs by reformulating their polynomial differential equation of degree $3$ as a system of two polynomial differential equations of degree $2$.

It is also interesting to note that Theorems \ref{thm:generalization_error} and \ref{thm:approximation} hold for any ridge hyperparameter $\lambda_{(\mathrm{ridge})} \geqslant \min(n_e, n_r)^{-\kappa}$ such that $\lim_{n_e, n_r \to \infty} \lambda_{(\mathrm{ridge})} = 0$. 
However, if $n_e$ and $n_r$ are fixed, choosing too large a $\lambda_{(\mathrm{ridge})}$ will lead to a bias toward parameters of $\Theta_{H,D}$ with a low $L^2$ norm. 
Therefore, there is a trade-off between taking $\lambda_{(\mathrm{ridge})}$ as small as possible to reduce this bias, but large enough to avoid overfitting, as illustrated in Section \ref{POF}. 
Moreover, our choice of $\lambda_{(\mathrm{ridge})}$ may be suboptimal, since these results rely on inequalities involving a general class of polynomial operators. 
When studying a particular PDE, the consistency results of Theorems \ref{thm:generalization_error} and \ref{thm:approximation} should eventually hold with a smaller $\lambda_{(\mathrm{ridge})}$. 
To tune $\lambda_{(\mathrm{ridge})}$ in practice, one could, for example, monitor the overfitting gap  $\mathrm{OG}_{n, n_e, n_r} = |R_{n, n_e, n_r} - \mathscr{R}_n|$ for a ridge estimator $\hat{\theta}^{(\mathrm{ridge})}(p, n_e, n_r, D)$, by standard validation strategy (e.g., by sampling $\tilde{n}_r$ and $\tilde{n}_e$ new points to estimate $\mathscr{R}_n(u_{\hat{\theta}^{(\mathrm{ridge})}(p, n_e, n_r, D)})$ at a $\min(\tilde{n}_r, \tilde{n}_e)^{-1/2}$-rate given by the central limit theorem), and then choose the smallest parameter $\lambda_{(\mathrm{ridge})}$ to introduce as little bias as possible. More information about the relevance of $\mathrm{OG}_{n, n_e, n_r}$ is given in \supplementSecTwo.

\section{Strong convergence of PINNs for linear PDE systems}
\label{sec:functional}
Beyond risk-consistency concerns, the ultimate goal of PINNs is to learn a physics-informed regression function $u^\star$, or, in the PDE solver setting, to strongly approximate the unique solution $u^\star$ of a PDE system. 
Thus, what we want is to have guarantees regarding the convergence of $u_{\hat{\theta}^{(\mathrm{ridge})}(p, n_e, n_r, D)}$ to $u ^\star$ for an adapted norm. This requirement is called strong convergence in the functional analysis literature.
This is however not guaranteed under the sole convergence of the theo\-re\-ti\-cal risk $(\mathscr{R}_n(u_{\hat{\theta}^{(\mathrm{ridge})}(p, n_e, n_r, D)}))_{p,n_e, n_r,D \in \mathbb{N}}$, as shown in the following two examples.

\begin{ex}[Lack of data incorporation in the hybrid modeling setting]
\label{ex:dataIncorpo}
Suppose $M =1$, $d_1 = 2$, $d_2 = 1$, $\Omega = ]0,1[\times ]0,T[$, $h(x,0) = 1$ and $h(0, t) = 1$,  and let $\mathscr{F}(u, \bx) = \partial_x u(\bx) + \partial_t u(\bx)$. 
This corresponds to the assumption that the solution should approximately follow the advection equation and that it should be close to $1$. 
For any $\delta >0$, let the sequence $(u_{\delta, p})_{p \in \mathbb{N}}\in \mathrm{NN}_H(2n)^{\mathbb{N}}$ be defined by \[u_{\delta, p}(x,t) = 1+\sum_{i=1}^n \frac{Y_i}{2} \big(\tanh_p^{\circ H}(x-t - x_i + t_i + \delta)- \tanh_p^{\circ H}(x-t - x_i+t_i-\delta)\big), \] 
where $\tanh_p:=\tanh(p\,\cdot)$, and $\bX_i = (x_i, t_i)$. 
Then, as soon as $\delta \leqslant \frac{1}{2}\min_{i\neq j} |x_i-x_j + t_j-t_i|$, we have that $\lim_{p \to \infty}\mathscr{R}_n (u_{\delta, p}) = 0$.  
Thus, as long as $D \geq 2n$, $\inf_{u\in \mathrm{NN}_H(D)}\mathscr R_n(u) = 0$. Therefore, Theorem \ref{thm:approximation} shows that $\lim_{D \to \infty} \lim_{n_e, n_r \to \infty} \lim_{p\to \infty}\mathscr{R}_n(u_{\hat{\theta}^{(\mathrm{ridge})}(p, n_e, n_r, D)}) =  0.$ It is then easy to check that this implies that $u_{\hat{\theta}^{(\mathrm{ridge})}(p, n_e, n_r, D)}$ converges in $L^2(\Omega)$ to $1$, independently of $n$ and the function $u^{\star}$. This shows that the ridge PINNs fails to learn $u^\star$ whenever the model is inexact.
\end{ex}

In the PDE solver setting, one can consider the a priori favorable case where the PDE system admits a unique (strong) solution $u^\star$ in $C^K(\bar{\Omega}, \mathbb{R}^{d_2})$ (where $K$ is the maximum order of the differential operators $\mathscr{F}_1$, $\hdots$, $\mathscr{F}_{M}$). 
Note that $u^\star$ is the unique minimizer of $\mathscr{R}$ over $C^K(\bar{\Omega}, \mathbb{R}^{d_2})$, with $\mathscr{R}(u^\star) = 0$ (and $\mathscr{R}(u) = 0$ if and only if $u$ satisfies the initial conditions, the boundary conditions, and the system of differential equations). However, we describe below a situation where a minimizing sequence of $\mathscr{R}$ does not converge to the unique strong solution $u^\star$ of the PDE in question. 
\begin{ex}[Divergence in the PDE solver setting]
\label{ex:degeneratePINN}
Suppose $M=1$, $d_1 = d_2 = 1$, $\Omega = ]-1,1[$, $h(1) = 1$, $\lambda_e > 0$, and let the polynomial operator be $\mathscr{F}(u, \bx) = \bx u'(\bx)$. 
Clearly, $u^\star(\bx) = 1$ is the only strong solution of the PDE $\bx u'(\bx) = 0$ with $u(1) = 1$. 
Let the sequence $(u_p)_{p\in \mathbb{N}}\in \mathrm{NN}_H(D)^{\mathbb{N}}$ be defined by $u_p = \tanh_p\circ \tanh^{\circ (H-1)}$. According to  \supplementSecTwo, $\lim_{p \to \infty} \mathscr{R}(u_p) = \mathscr{R}(u^{\star}) = 0$.
However, the minimizing sequence $(u_p)_{p\in\mathbb{N}}$ does not converge to $u^\star$, since $ u_{\infty}(\bx) := \lim_{p \to \infty} u_p(\bx) = \mathbf{1}_{\bx > 0} -  \mathbf{1}_{\bx < 0}$. 
\end{ex}
We have therefore exhibited a sequence $(u_p)_{p \in \mathbb{N}}$ of neural networks that minimizes $\mathscr{R}$ and such that $(u_p)_{p\in\mathbb{N}}$ converges pointwise. However, its limit $u_\infty$ is not the unique strong solution of the PDE. In fact, $u_\infty$ is not differentiable at $0$, which is incompatible with the differential operators $\mathscr{F}$ used in $\mathscr{R}(u_\infty)$. Interestingly, the Cauchy-Schwarz inequality states that the pathological sequence $(u_p)_{p\in\mathbb{N}}$ satisfies $\lim_{p\to\infty}\|u_p'\|_{L^2(\Omega)}^2=\infty$, as in Example \ref{ex:dataIncorpo}.

\subsection{Sobolev regularization}
The two examples above illustrate how the convergence of the theoretical risk $\mathscr{R}_n(u_{\hat{\theta}^{(\mathrm{ridge})}(p, n_e, n_r, D)})$ to $\inf_{u\in C^\infty(\bar{\Omega}, \mathbb{R}^{d_2})}\mathscr{R}_n(u)$ (for any $n$) is not sufficient to guarantee the strong convergence to a PDE or hybrid modeling solution.
To ensure such a convergence, a different analysis is needed, mobilizing tools from functional analysis.
In the sequel, we build upon the regression estimation penalized by PDEs of \citet{azzimonti2015blood}, \citet{sangalli2021spatial}, \citet{ arnone2022spatialRegression}, and \citet{ferraccioli2022some}, and make use of the calculus of variations \citep[e.g.,][Theorems 1-4, Chapter 8]{evans2010partial}.
In the former references, the minimizer of $\mathscr{R}_n$ does not satisfy the PDE system injected in the PINN penalty, but another PDE system, known as the Euler-Lagrange equations.
Although interesting, the mathematical framework is different from ours. First, the authors do not study the convergence of neural networks, but rather methods in which the boundary conditions are hard-coded, such as the finite element method. Second, these frameworks are limited to special cases of theoretical risks. Indeed, only second-order PDEs with $\lambda_e=\infty$ are considered in \citet{azzimonti2015blood}, 
while \citet{evans2010partial} deal with first-order PDEs, echoing the case of $\lambda_d = 0$ and $ \lambda_e = \infty$.

It is worthwhile mentioning that the results of \citet{azzimonti2015blood}  rely on an important property of the theoretical risk function $\mathscr{R}_n$, called coercivity.  This is a common assumption of the calculus of variations \citep{evans2010partial}. The operator $\mathscr{R}_n$ is said to be coercive if there exist $K \in \mathbb{N}$ and $\lambda_t>0$ such that, for all $u\in H^K(\Omega, \mathbb{R}^{d_2})$, $\mathscr{R}_n(u) \geqslant \lambda_t \|u\|_{H^K(\Omega)}^2$ (the notation $H^K(\Omega, \mathbb{R}^{d_2}$) stands for the usual Sobolev space of order $K$---see \appendixlink. 
It turns out that the failures of Examples  \ref{ex:dataIncorpo} and \ref{ex:degeneratePINN} are due to a lack of coercivity, since, in both cases,  $\lim_{p \to\infty} \|u_p\|_{H^1(\Omega)}= \infty$ but $\lim_{p \to\infty} \mathscr{R}_n(u_p) \leqslant \mathscr{R}_n(u^\star)$. There are two ways to correct this problem: either one can restrict the study to coercive operators only, or one can resort to an explicit regularization of the risk to enforce its coercivity. We choose the latter, since most PDEs used in the practice of PINNs are not coercive. Note however that our results could be easily adapted to the coercive case. 

In the following, we restrict ourselves to affine operators, which exactly correspond to linear PDE systems, including  the advection, heat, wave, and Maxwell equations. 
\begin{defi}[Affine operator]
\label{defi:affine}
    The operator $\mathscr{F}$ is affine of order $K$ if there exists $A_{\alpha} \in C^\infty(\bar{\Omega}, \mathbb{R}^{d_2})$ and $B \in C^\infty(\bar{\Omega}, \mathbb{R})$  such that, for all $\bx \in \Omega$ and all $u \in H^K(\Omega, \mathbb{R}^{d_2})$, 
    \[\mathscr{F}(u, \bx) = \mathscr{F}^{(\mathrm{lin})}(u, \bx) + B(\bx),\]
    where  $\mathscr{F}^{(\mathrm{lin})}(u, \bx) = \sum_{|\alpha|\leqslant K} \langle A_{\alpha}(\bx), \partial^\alpha u(\bx)\rangle$ is linear.
\end{defi}

The source term $B$ is important, as it makes it possible to model a large variety of applied physical problems, as illustrated in \citet{song2021solving}. Note also that affine operators of order $K$ are in fact polynomial operators of degree $K+1$ (Definitions \ref{defi:polyop} and \ref{defi:deg}) that are extended from smooth functions to the whole Sobolev space $H^K(\Omega, \mathbb{R}^{d_2})$.

\begin{defi}[Regularized PINNs]
    The regularized theoretical risk function is
\begin{equation}
    \label{eq:regThRisk} \mathscr{R}^{(\mathrm{reg})}_n(u) = \mathscr{R}_n(u) + \lambda_t \|u\|_{H^{m+1}(\Omega)}^2,
\end{equation} 
where $\mathscr{R}_n$ is the original theoretical risk as defined in \eqref{lossTheorical}, and $m \in \mathbb{N}$. 
The corresponding regularized empirical risk function is
\begin{equation*}
    R_{n, n_e,n_r}^{(\mathrm{reg)}}(u_\theta) = R_{n, n_e,n_r}(u_\theta) + \lambda_{(\mathrm{ridge})} \|\theta\|_2^2 + \frac{\lambda_t}{n_\ell} \sum_{\ell=1}^{n_\ell} \sum_{|\alpha|\leqslant m+1}\|\partial^\alpha u_\theta(\bX_{\ell}^{(r)})\|_2^2.
\end{equation*} 
\end{defi}
It is noteworthy that $R_{n, n_e,n_r}^{(\mathrm{reg)}}$ can be straightforwardly implemented in the usual PINN framework and benefit from the computational scalability of the backpropagation algorithm, by encoding the regularization as supplementary PDE-type constraints $\mathscr{F}_\alpha(u, \bx) = \partial^\alpha u(\bx) = 0$. Since this discretized Sobolev penalty can be seen as additional physical priors $\mathscr{F}_\alpha$, the overfitting behavior observed for the unregularized PINNs can be transferred to Sobolev-regularized PINNs trained without ridge regularization. This is why the ridge penalty is still included in the risk.
Note also that the Sobolev regularization has been shown to avoid overfitting in machine learning, yet in different contexts \citep[e.g.,][]{fischer2020sobolev}.

The following proposition shows that the unique minimizer of \eqref{eq:regThRisk} can be interpreted as the unique minimizer of an optimization problem involving a weak formulation of the differential terms included in the risk. Its proof is based on the Lax-Milgram theorem \citep[e.g.,][Corollary 5.8]{brezis2010functional}.

\begin{prop}[Characterization of the unique minimizer of  $\mathscr{R}^{(\mathrm{reg})}_n$]
\label{prop:laxMLin}
    Assume that $\mathscr{F}_1, \hdots, \mathscr{F}_M$ are affine operators of order $K$. 
    Assume, in addition, that $\lambda_t >0$ and  $m \geqslant \max(\lfloor d_1/2\rfloor, K)$. Then the regularized theoretical risk $\mathscr{R}^{(\mathrm{reg})}_n$ has a unique minimizer $\hat u_n$ over $H^{m+1}(\Omega, \mathbb{R}^{d_2})$. This minimizer
    $\hat u_n$ is the unique element of $H^{m +1}(\Omega, \mathbb{R}^{d_2})$ that satisfies 
    \[\forall v \in H^{m+1}(\Omega, \mathbb{R}^{d_2}),\quad  \mathcal{A}_n(\hat u_n,v) = \mathcal{B}_n(v),\] 
     where
     \begin{align*}
        \mathcal{A}_n(\hat u_n,v) &= \frac{\lambda_d}{n} \sum_{i=1}^n \langle \tilde \Pi(\hat u_n)(\bX_i), \tilde \Pi(v)(\bX_i)\rangle +\lambda_e \mathbb{E}\langle\tilde \Pi(\hat u_n)(\bX^{(e)}),\tilde \Pi(v)(\bX^{(e)})\rangle\\
        & \quad +\frac{1}{|\Omega|}\sum_{k=1}^M\int_{\Omega} \mathscr{F}_k^{(\mathrm{lin})}(\hat u_n,\bx)\mathscr{F}_k^{(\mathrm{lin})}(v,\bx)d\bx\\
        & \quad + \frac{\lambda_t }{|\Omega|} \sum_{|\alpha|\leqslant m+1}\int_{\Omega}\langle \partial^\alpha \hat u_n(\bx) , \partial^\alpha v(\bx)\rangle d\bx,\\
        \mathcal{B}_n(v) &= \frac{\lambda_d}{n} \sum_{i=1}^n \langle Y_i, \tilde \Pi(v)(\bX_i)\rangle+ \lambda_e \mathbb{E}\langle\tilde \Pi(v)(\bX^{(e)}),h(\bX^{(e)})\rangle\\
        &\quad -\frac{1}{|\Omega|}\sum_{k=1}^M\int_{\Omega} B_k(\bx) \mathscr{F}_k^{(\mathrm{lin})}(v,\bx)d\bx,
    \end{align*}
    and where $\tilde{\Pi} : H^{m+1}(\Omega, \mathbb{R}^{d_2}) \to C^0(\Omega, \mathbb{R}^{d_2})$ is the so-called Sobolev embedding, such that $\tilde \Pi(u)$ is the unique continuous function that coincides with $u$ almost everywhere.
\end{prop}
The Sobolev embedding $\tilde{\Pi}$ is essential in order to give a precise meaning to the pointwise evaluation at the points $\bX_i$ of a function $u\in H^{m+1}(\Omega, \mathbb{R}^{d_2}) \subseteq L^2(\Omega, \mathbb{R}^{d_2})$, which is defined only almost everywhere.
The rationale behind Proposition \ref{prop:laxMLin} is that 
\[
\mathscr{R}_n^{(\mathrm{reg})}(u) = \mathcal{A}_n(u,u) -2\mathcal{B}_n(u) + \frac{\lambda_d}{n} \sum_{i=1}^n \|Y_i\|^2 +\lambda_e \mathbb{E}\|h(\bX^{(e)})\|_2^2 + \frac{1}{|\Omega|}\sum_{k=1}^M\int_{\Omega} B_k(\bx)^2d\bx.
\]
Therefore, minimizing $\mathscr{R}_n^{(\mathrm{reg})}$ amounts to minimizing $\mathcal{A}_n - 2 \mathcal{B}_n$.
It is also interesting to note that the  weak formulation $\mathcal{A}_n(\hat u,v) = \mathcal{B}_n(v)$ can be interpreted as a weak PDE on $H^{m+1}(\Omega, \mathbb{R}^{d_2})$.
In particular, if $\hat u_n \in H^{2(m+1)}(\Omega, \mathbb{R}^{d_2})$, then one has, almost everywhere,
\[
\sum_{k=1}^M (\mathscr{F}_k^{(\mathrm{lin})})^* \mathscr{F}_k(\hat u_n,\bx)+ \lambda_t  \sum_{|\alpha|\leqslant m+1} (-1)^{|\alpha|} (\partial^\alpha)^2 \hat u_n(\bx)  = 0.
\] 
$(\mathscr{F}_k^{(\mathrm{lin})})^*$ is the adjoint operator of $\mathscr{F}_k^{(\mathrm{lin})}$ such that, for all $ u, v \in C^\infty(\Omega, \mathbb{R})$ with $v|_{\partial \Omega} = 0$, 
\[
\int_\Omega u \mathscr{F}^{(\mathrm{lin})}(v, \bx) d\bx = \int_\Omega (\mathscr{F}_k^{(\mathrm{lin})})^*(u, \bx)vd\bx.
\]
Thus, even in the regime $\lambda_t \to 0$  (i.e., when the regularization becomes negligible), the solution of the PINN problem does not satisfy the constraints $\mathscr{F}_k(u,\bx) = 0$, but the following constraint $\sum_{k=1}^M (\mathscr{F}_k^{(\mathrm{lin})})^* \mathscr{F}_k(u,\bx) =0$. 
(Notice that, in the PDE solver setting, since $u^\star$ satisfies all the constraints, it satisfies in particular the constraint $\sum_{k=1}^M (\mathscr{F}_k^{(\mathrm{lin})})^* \mathscr{F}_k(u^\star,\bx) =0$.)
For instance, the advection equation constraint $\mathscr{F}(u, \bx) = (\partial_x + \partial_t) u(\bx)$ of Example \ref{ex:dataIncorpo} becomes $\mathscr{F}^* \mathscr{F}(u, \bx) = - (\partial_x + \partial_t)^2 u(\bx)$, and the constraint $\mathscr{F}(u, \bx) = \bx u'(\bx)$ of Example \ref{ex:degeneratePINN} becomes $\mathscr{F}^* \mathscr{F}(u, \bx) = -2\bx  u'(\bx)-\bx^2  u''(\bx)$.

Proposition \ref{prop:laxMLin} shows that the regularization in $\lambda_t$ is sufficient to make the PINN problem well-posed, i.e., to ensure that the theoretical risk function \eqref{eq:regThRisk} admits a unique minimizer.
The next natural requirement is that the regularized PINN estimator obtained by minimizing the regularized empirical risk function converges to this unique minimizer $\hat u_n$. 
Proposition \ref{prop:sequenceCvLin} and Theorem \ref{thm:functionalCv} show that this is true for linear PDE systems.
\begin{prop}[From risk-consistency to strong convergence]
    \label{prop:sequenceCvLin}
    Assume that $\lambda_t >0$ and  $m \geqslant \max(\lfloor d_1/2\rfloor, K)$. 
    Let $(u_p)_{p\in\mathbb{N}} \in C^\infty(\bar{\Omega}, \mathbb{R}^{d_2})$ be a sequence of smooth functions satisfying that $\lim_{p \to \infty}\mathscr{R}^{\mathrm{(reg)}}_n(u_p) = \inf_{u\in C^\infty(\bar \Omega, \mathbb{R}^{d_2})}\mathscr{R}^{\mathrm{(reg)}}_n$. Then $\lim_{p\to \infty} \|u_p - \hat u_n\|_{H^{m}(\Omega)}=0$, where $\hat u_n$ is the unique minimizer of $\mathscr{R}^{(\mathrm{reg})}_n$ over $H^{m+1}(\Omega, \mathbb{R}^{d_2})$. 
\end{prop}
The next theorem follows from Theorem \ref{thm:approximation} and Proposition \ref{prop:sequenceCvLin}, by simply observing that the Sobolev regularization is just an ordinary PINN regularization, taking the form of a polynomial operator of degree $(m+2)$.
\begin{thm}[Strong convergence of regularized PINNs]
\label{thm:functionalCv}
Assume that $\mathscr{F}_1, \hdots, \mathscr{F}_M$ are affine operators of order $K$. 
Assume, in addition, that $\lambda_t >0$, $m \geqslant \max(\lfloor d_1/2\rfloor, K)$, and the condition function $h$ is Lipschitz.
Let $(\hat{\theta}^{(\mathrm{reg})}(p, n_e, n_r, D))_{p \in \mathbb{N}}$ be a minimizing sequence of the regularized empirical risk function
    \begin{equation*}
        R_{n, n_e,n_r}^{(\mathrm{reg)}}(u_\theta) = R_{n, n_e,n_r}(u_\theta) + \lambda_{(\mathrm{ridge})} \|\theta\|_2^2 + \frac{\lambda_t}{n_\ell} \sum_{\ell=1}^{n_\ell} \sum_{|\alpha|\leqslant m+1}\|\partial^\alpha u_\theta(\bX_{\ell}^{(r)})\|_2^2
    \end{equation*}
    over the class $\mathrm{NN}_H(D)=\{u_\theta, \theta\in\Theta_{H,D}\}$, where $H \geqslant 2$.
    Then, with the choice
\[\lambda_{(\mathrm{ridge})} = \min(n_e, n_r)^{-\kappa}, \quad \text{where} \quad  \kappa=\frac{1}{12+4H(1+(2+H)(m+2))},\]
one has, almost surely, 
\[\lim_{D \to \infty} \lim_{n_e, n_r \to \infty} \lim_{p\to \infty} \|u_{\hat{\theta}^{(\mathrm{reg})}(p, n_e, n_r, D)} - \hat u_n\|_{H^m(\Omega)} = 0,\]
where $\hat u_n$ is the unique minimizer of $\mathscr{R}^{\mathrm{(reg)}}_n$ over $H^{m+1}(\Omega, \mathbb{R}^{d_2})$.
\end{thm}

Theorem \ref{thm:functionalCv} ensures 
that the sequence $u_{\hat{\theta}^{(\mathrm{reg})}(p, n_e, n_r, D)}$ of PINNs converges to the unique minimizer $\hat u_n$ of the regularized theoretical risk function \eqref{eq:regThRisk}, provided that the ridge hyperparameter $\lambda_{(\mathrm{ridge})}$ vanishes slowly enough.
However, it does not provide any information about the proximity between $u_{\hat{\theta}^{(\mathrm{reg})}(p, n_e, n_r, D)}$ and $u^\star$. 
On the other hand, since the regularized theoretical risk function is a small perturbation of the theoretical risk function \eqref{lossTheorical}, it is reasonable to think that its minimizer $\hat u_n$ should in some way converge to $u^\star$ as $\lambda_t \to 0$. This is encapsulated in Theorem \ref{prop:pdeSolverFunctional} for the PDE  solver setting and in Theorem \ref{cor:sPINNsConsistency} for the more general hybrid modeling setting.

\subsection{The PDE solver case}
\begin{thm}[Strong convergence of linear PDE solvers]
    \label{prop:pdeSolverFunctional}
    Assume that $\mathscr{F}_1, \hdots, \mathscr{F}_M$ are affine operators of order $K$.
    Consider the PDE solver setting (i.e., $\lambda_e >0$ and $\lambda_d = 0$) and assume that the condition function $h$ is Lipschitz. Assume, in addition, that the PDE system admits a unique solution $u^\star$ in $H^{m+1}(\Omega, \mathbb{R}^{d_2})$ for some $m \geqslant \max(\lfloor d_1/2\rfloor, K)$. Let $(\hat{\theta}^{(\mathrm{reg})}(p, n_e, n_r, D, \lambda_t))_{p \in \mathbb{N}}$ be a minimizing sequence of the regularized empirical risk function
    \begin{equation*}
        R_{n_e,n_r}^{(\mathrm{reg)}}(u_\theta) = R_{n_e,n_r}(u_\theta) + \lambda_{(\mathrm{ridge})} \|\theta\|_2^2 + \frac{\lambda_t}{n_\ell} \sum_{\ell=1}^{n_\ell} \sum_{|\alpha|\leqslant m+1}\|\partial^\alpha u_\theta(\bX_{\ell}^{(r)})\|_2^2
    \end{equation*}
    over the class $\mathrm{NN}_H(D)=\{u_\theta, \theta\in\Theta_{H,D}\}$, where $H \geqslant 2$. 
      Then, with the choice
    \[\lambda_{(\mathrm{ridge})} = \min(n_e, n_r)^{-\kappa}, \quad \text{where} \quad  \kappa=\frac{1}{12+4H(1+(2+H)(m+2))},\]
    one has, almost surely, 
    \[\lim_{\lambda_t \to 0}\lim_{D \to \infty} \lim_{n_e, n_r \to \infty} \lim_{p\to \infty} \|u_{\hat{\theta}^{(\mathrm{reg})}(p, n_e, n_r, D, \lambda_t)} -  u^\star\|_{H^m(\Omega)} = 0.\]
\end{thm} 
Back to Example \ref{ex:degeneratePINN}, one has $m = 1$.  
Recall that, in this setting, the unique minimizer of $\mathscr{R}$ over $C^0([-1,1], \mathbb{R})$ is $u^\star(\bx) = 1$, satisfying $u^\star \in H^2(]-1,1[, \mathbb{R})$. Therefore, by letting $\lambda_t \to 0$, this theorem shows that any sequence minimizing the regularized empirical risk function converges, with respect to the $H^2(\Omega)$ norm, to the unique strong solution $u^\star$ of the PDE $\bx u'(\bx) = 0$ and $u(1)=1$.  
\begin{remark}[Dimensionless hyperparameters and lower regularity assumptions on $u^\star$]
    The condition  $m \geqslant \lfloor d_1 / 2 \rfloor$ in Theorem \ref{thm:functionalCv} is necessary to make the pointwise evaluations $\tilde \Pi(u)(\bX_i)$ continuous.
    This condition does have an impact on $\lambda_{(\mathrm{ridge})}$, which grows exponentially fast with the dimension $d_1$. However, in the PDE solver setting, it is possible to get rid of this dimension problem, taking $m = \max_k\deg(\mathscr{F}_k)$. To see this, just note that there is no $\bX_i$, and so there is no need to resort to the $\tilde \Pi(u)(\bX_i)$.
    Indeed, the proof of 
    Theorem \ref{prop:pdeSolverFunctional} can be adapted by replacing the Sobolev inequalities in the proofs of Theorem \ref{thm:functionalCv} by the trace theorem for Lipschitz domains \citep[e.g.,][Theorem 1.5.1.10]{grisvard1985elliptic}. 
    In this case, it is enough to assume that $u^\star \in H^{K+1}(\Omega, \mathbb{R}^{d_2})$, which is less restrictive than $u^\star \in H^{\max(\lfloor d_1/2\rfloor, K)+1}(\Omega, \mathbb{R}^{d_2})$. However, this comes at the price of assuming that $\mu_E$ admits a density with respect to the hypersurface measure on $\partial \Omega$ (as it is often the case in practice).
\end{remark}

\subsection{The hybrid modeling case}
To apply Theorem \ref{thm:functionalCv} to the general  hybrid modeling setting, it is necessary to measure the gap between $u^\star$ and the model specified by the constraints $\mathscr{F}_1, \hdots, \mathscr{F}_M$ and the condition function $h$. This is encapsulated in the next definition.
\begin{defi}[Physics inconsistency]
    For any $u \in H^{m+1}(\Omega, \mathbb{R}^{d_2})$, the physics inconsistency of $u$ is defined by
    \[\mathrm{PI}(u) = \lambda_e\mathbb{E}\|\tilde \Pi(u)(\bX^{(e)})-h(\bX^{(e)})\|_2^2 + \frac{1}{|\Omega|}\sum_{k=1}^{M}\int_\Omega \mathscr{F}_k(u,\bx)^2 d\bx.\] 
\end{defi}
Observe that $\mathscr{R}_n(u) = \frac{\lambda_d}{n} \sum_{i=1}^n \|\tilde \Pi(u)(\bX_i) - Y_i\|_2^2 + \mathrm{PI}(u)$.
In short, the quantity $\mathrm{PI}(u)$ measures how well the boundary/initial conditions, encoded by $h$, and the PDE system, encoded by the $\mathscr{F}_k$, describe the function $u$ \citep[see also][]{willard2023integrating}. In particular, $\mathrm{PI}(u^\star)$ measures the modeling error---the better the model, the lower $\mathrm{PI}(u^\star)$.

\begin{prop}[Strong convergence of hybrid modeling]
    \label{prop:consistencey} 
    Assume that the conditions of Theorem \ref{thm:functionalCv} are satisfied. 
    Then $\hat u_n\equiv \hat u_n(\bX_1, \hdots, \bX_n, \varepsilon_1, \hdots, \varepsilon_n)$ is a random variable such that $\mathbb{E}\|\hat u_n\|^2_{H^{m+1}(\Omega)} < \infty$.
    
   Suppose, in addition, that $u^\star \in H^{m+1}(\Omega, \mathbb{R}^{d_2})$,  that the noise $\varepsilon$ is independent from $\bX$, and that $\varepsilon$ has the same distribution as $-\varepsilon$. 
    Then there exists a constant $C_\Omega > 0$, depending only on $\Omega$, such that \begin{align*}
        \mathbb{E}\int_\Omega \|\tilde{\Pi} (\hat u_n - u^\star)\|_2^2d\mu_\bX &\leqslant \frac{1}{\lambda_d}\big(\mathrm{PI}(u^\star) + \lambda_t\|u^\star\|_{H^{m+1}(\Omega)}^2\big)\\
        &\quad + \frac{C_\Omega d_2^{1/2}}{n^{1/2}} \Big(2\|u^\star\|_{H^{m+1}(\Omega)}^2 + \frac{\mathrm{PI}(u^\star)}{\lambda_t}\Big)\\
        &\quad +\frac{8\mathbb{E}\|\varepsilon\|_2^2}{n}\Big(1+C_\Omega d_2^{3/2}\Big(\frac{\lambda_d}{\lambda_t} +\frac{\lambda_d^2}{\lambda_t^2n^{1/2}}\Big)\Big).
    \end{align*}
    In particular, with the choice $\lambda_e = 1$, $\lambda_t = (\log n)^{-1}$, and $\lambda_d = n^{1/2}/(\log n)$, one has
    \[\mathbb{E}\int_\Omega \|\tilde{\Pi} (\hat u_n - u^\star)\|_2^2d\mu_\bX \leqslant 
     \frac{\Lambda \log^2(n)}{n^{1/2} },\]
    where $\Lambda = 24d_2^{3/2}C_\Omega(\mathrm{PI}(u^\star) + \|u^\star\|_{H^{m+1}(\Omega)} + \mathbb{E}\|\varepsilon\|_2^2)$.
\end{prop}
This (nonasymptotic) proposition provides an insight into the scaling of the PINN hyperparameters. Indeed, the term $\frac{1}{\lambda_d}(\mathrm{PI}(u^\star) + \lambda_t\|u^\star\|_{H^{m+1}(\Omega)}) $ encapsulates the modeling error, damped by the weight $\lambda_d$. However, $\lambda_d$ cannot be arbitrarily large because of the term $\frac{8\mathbb{E}\|\varepsilon\|_2^2}{n}\big(1+C_\Omega d_2^{3/2}\big(\frac{\lambda_d}{\lambda_t} +\frac{\lambda_d^2}{\lambda_t^2n^{1/2}}\big)\big)$.
So, there is a trade-off between the modeling error and the random variation in the data.
Note also the other trade-off in the regularization hyperparameter $\lambda_t$, which should not converge to $0$ too quickly because of the term $\frac{C_\Omega d_2^{1/2}}{n^{1/2}} \big(2\|u^\star\|_{H^{m+1}(\Omega)}^2 + \frac{\mathrm{PI}(u^\star)}{\lambda_t}\big)$.
\begin{prop}[Physics consistency of hybrid modeling]
    \label{thm:phyCst}
     Under the conditions of Proposition \ref{prop:consistencey}, if $\lim_{n \to \infty} \frac{\lambda_d^2}{n \lambda_t} = 0$ and $\lim_{n \to \infty}\lambda_t = 0$, one has
    \[\mathbb{E}(\mathrm{PI}(\hat u_n)) \leqslant \mathrm{PI}(u^\star) + \oequivalent_{n\to \infty}(1).\]
    (Note that the conditions are satisfied with $\lambda_e = 1$, $\lambda_t = (\log n)^{-1}$, and $\lambda_d = n^{1/2}/(\log n)$.)
\end{prop}
As usual, we let $(u^{(n)}_{\hat{\theta}^{(\mathrm{reg})}(p, n_e, n_r, D)})_{p\in \mathbb{N}} \in \mathrm{NN}_H(D)^\mathbb{N}$ be a minimizing sequence of $R_{n, n_e,n_r}^{(\mathrm{reg)}}$, where the exponent $n$ indicates that the sample size $n$ is kept fixed along the sequence. 
Since $u^{(n)}_{\hat{\theta}^{(\mathrm{reg})}(p, n_e, n_r, D)} \in C^\infty(\bar \Omega, \mathbb{R}^{d_2})$, one has $\tilde \Pi(u^{(n)}_{\hat{\theta}^{(\mathrm{reg})}(p, n_e, n_r, D)}) = u^{(n)}_{\hat{\theta}^{(\mathrm{reg})}(p, n_e, n_r, D)}$. Thus, by combining Theorem \ref{thm:functionalCv} with Propositions \ref{prop:consistencey} and \ref{thm:phyCst}, we obtain the following important theorem.
\begin{thm}[Strong convergence of regularized PINNs]
\label{cor:sPINNsConsistency}
    Under the same assumptions as in 
    Theorem \ref{thm:functionalCv} and Proposition \ref{prop:consistencey}, with the choice $\lambda_e = 1$, $\lambda_t = (\log n)^{-1}$, and $\lambda_d = n^{1/2}/(\log n)$, one has 
    \[\lim_{D \to \infty} \lim_{n_e, n_r \to \infty} \lim_{p\to \infty} \mathbb{E}\int_\Omega \| u^{(n)}_{\hat{\theta}^{(\mathrm{reg})}(p, n_e, n_r, D)} -  u^\star\|_2^2d\mu_\bX \leqslant 
     \frac{\Lambda \log^2(n)}{n^{1/2} }\]
     and 
     \[\lim_{D \to \infty} \lim_{n_e, n_r \to \infty} \lim_{p\to \infty} \mathbb{E}(\mathrm{PI}(u^{(n)}_{\hat{\theta}^{(\mathrm{reg})}(p, n_e, n_r, D)})) \leqslant \mathrm{PI}(u^\star) + \oequivalent_{n\to \infty}(1).\]
\end{thm}
The minimax regression rate over any bounded class of functions in $C^{(m+1)}(\Omega, \mathbb{R}^{d_2})$ is known to be $n^{-2(m+1)/(2(m+1) +d_1)}$ \citep[][Theorem 1]{stone1982optimal}. Theorem \ref{cor:sPINNsConsistency} shows that the regularized PINN estimator achieves the rate $\log(n)/n^{1/2}$ over any \textit{larger} class bounded in $H^{(m+1)}(\Omega, \mathbb{R}^{d_2})$. Thus, the regularized PINN estimator has the nearly optimal rate, up to a log term, in the regime $d_1 \to \infty$ and $m = \lfloor d_1/2\rfloor$.

Theorem \ref{cor:sPINNsConsistency} shows that a properly regularized PINN estimator 
is both statistically \textit{and} physics consistent, in the sense that the error $\mathbb{E}\int_\Omega \| u^{(n)}_{\hat{\theta}^{(\mathrm{reg})}(p, n_e, n_r, D)} -  u^\star\|_2^2d\mu_\bX $ converges to zero with a physics inconsistency $ \mathbb{E}(\mathrm{PI}(u^{(n)}_{\hat{\theta}^{(\mathrm{reg})}(p, n_e, n_r, D)}))$ that is asymptotically no larger than $ \mathrm{PI}( u^\star)$.  It is also worth mentioning that in some applications, the physical measures $\bX_1, \hdots, \bX_n$ are forced to be sampled in certain subset of $\Omega$. 
An important application is when $\Omega$ is spatio-temporal and one wishes to extrapolate/transfer a model from a training dataset collected on $\mathrm{supp}(\mu_\bX) = \Omega_1 \times ]0,T_{\mathrm{train}}[$ to a test $\Omega_1 \times ]T_{\mathrm{train}},T_{\mathrm{test}}[$, using a temporal evolution PDE system to extrapolate \citep[e.g., ][]{cai2021physics}. On the other hand, the physical restriction on the data measurement can be also strictly spatial. 
This  is for example the case in some blood modeling problems, where the blood flow measures can only be taken in a specific region of a blood vessel, as illustrated in \citet{arzani2021uncovering}. Thus, in both these contexts, the support $\mathrm{supp}(\mu_\bX)$ of the distribution $\mu_\bX$ is strictly contained in $\Omega$.
Of course, this is compatible with Theorem \ref{cor:sPINNsConsistency}, which shows that the regularized PINN estimator consistently interpolates the function $u^\star$ on $\mathrm{supp}(\mu_\bX)$. Furthermore, Theorem \ref{cor:sPINNsConsistency} shows that the estimator uses the physical model to extrapolate on $\Omega \backslash \mathrm{supp}(\mu_\bX)$. 
In summary, the better the model, the lower the modeling error $\mathrm{PI}(u^\star)$, and the better the domain adaptation capabilities. 
This provides an interesting mathematical insight into the relevance of combining data-driven statistical models with the interpretability and extrapolation capabilities of physical modeling.

\smallskip
\paragraph{Numerical illustration of imperfect modeling}  In the following experiments, we  
illustrate with a toy example the results of Theorem \ref{cor:sPINNsConsistency} and show  how the Sobolev regularization can be implemented directly in the PINN framework, taking advantage of the automatic differentiation and backpropagation. 
Let $\Omega = ]0,1[^2$ and assume that $Y =  u^\star(\bX) + \mathcal{N}(0,10^{-2})$, where 
$u^\star(x, t) = \exp(t-x) + 0.1 \cos(2\pi x)$. 
In this hybrid modeling setting, the goal is to reconstruct $u^\star$.
We consider  an advection model of the form $\mathscr{F}(u, \bx) = \partial_x u(\bx) + \partial_t u(\bx)$, with $h(x, 0) = \exp(-x)$ and $h(0, t) = \exp(t)$. The unique solution of this PDE is $u_{\mathrm{model}}(x, t) = \exp(t-x)$ (Figure \ref{fig:hybrid_modelling_nn}, left). Note that the function $u_{\mathrm{model}}$ is different from $u^\star$ (Figure \ref{fig:hybrid_modelling_nn}, middle), which casts our problem in the imperfect modeling setting. This PDE prior is relevant because
$\|u_{\mathrm{model}}-u^\star\|_{L^2(\Omega)}^2 \simeq \exp(-5.3)$ and $\mathrm{PI}(u^\star) \simeq \exp(-1.6)$, two quantities that are negligible with respect to $\|u^\star\|_{L^2(\Omega)}^2\simeq \exp(0.3)$. 
We randomly sample $n$ observations $\bX_1, \hdots, \bX_n$ uniformly on the rectangle $\mathrm{supp}(\mu_\bX) = ]0,0.5[\times ]0,1[ \subsetneq \Omega$ (note that this is a strict inclusion), and let $n$ vary from $n_{\min} = 10$ to $n_{\max} = 10^3$ (linearly in a log scale).  

The architecture of the neural networks is set to $H= 2$ hidden layers with width $D = 100$, so that the total number of parameters is $10\,600 \gg n_{\max}$.
We fix $n_e, n_r = 10^4 \gg n_{\max}$ and $\lambda_{(\mathrm{ridge})} = \min(n_e, n_r)^{-1/2}$.  
Figure \ref{fig:monitoring} shows  the evolution of the regularized risk $R_{n, n_e,n_r}^{(\mathrm{reg)}}(u^{(n)}_{\hat \theta^{(\mathrm{reg})}(p, n_r, n_e, D)})$ in blue, with respect to the number $p$ of epochs in the gradient descent (for $n=10$). 
For a fixed number $n$ of observations,  the number $p_{\max}$ of epochs to stop training is determined by monitoring the evolution of the risk $R_{n, n_e,n_r}^{(\mathrm{reg})}(u^{(n)}_{\hat \theta^{(\mathrm{reg})}(p_{\max}, n_r, n_e, D)})$ (blue curve) and the overfitting gap $\mathrm{OG}_{n, n_e, n_r} = |R_{n, n_e, n_r}^{(\mathrm{reg})}- \mathscr{R}_n^{(\mathrm{reg})}|$ (orange curve). Both are required to be stable around a minimal value, so that the minimum of the risk is approximately reached, i.e., $R_{n, n_e,n_r}^{(\mathrm{reg})}(u^{(n)}_{\hat \theta^{(\mathrm{reg})}(p_{\max}, n_r, n_e, D)}) \simeq \inf_{u \in \mathrm{NN}_H(D)} R_{n, n_e,n_r}^{(\mathrm{reg})}(u)$ and $\mathscr{R}_n^{(\mathrm{reg})}(u^{(n)}_{\hat \theta^{(\mathrm{reg})}(p_{\max}, n_r, n_e, D)}) \simeq \inf_{u \in \mathrm{NN}_H(D)} \mathscr{R}_n^{(\mathrm{reg})}(u)$.
\begin{figure}
    \centering
    \vspace{-0.4cm}
    \includegraphics[width = 0.45\textwidth]{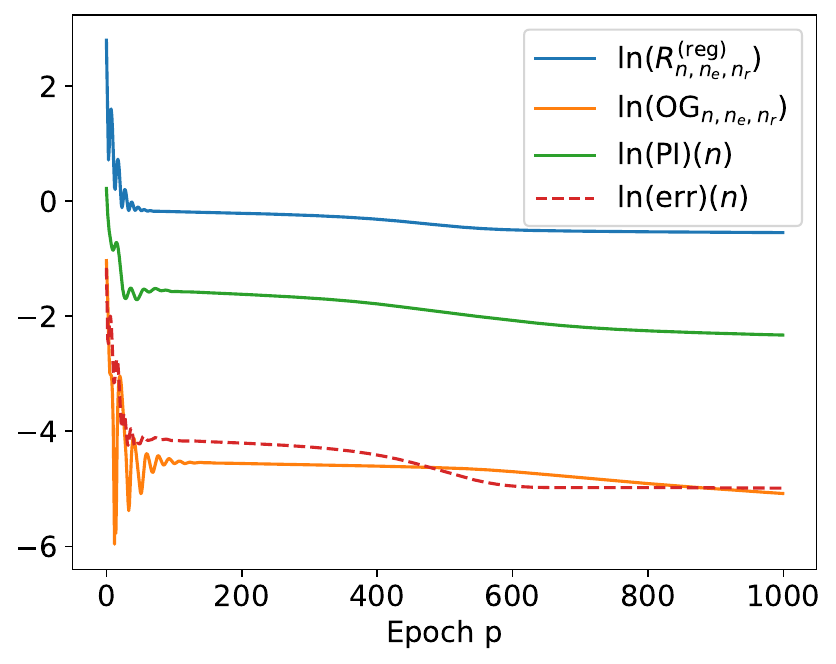}    \caption{Regularized empirical risk (blue) and overfitting gap $\mathrm{OG}$ (orange) with respect to the number $p$ of epochs for $n=10$. The physics inconsistency $\mathrm{PI}(n)$ (green) and the $L^2$ error $\mathrm{err}(n)$ (red) are also depicted.}
    \label{fig:monitoring}
\end{figure}
In this overparameterized regime ($D$ is large), one can consider that $\mathscr{R}_n^{(\mathrm{reg})}(u^{(n)}_{\hat \theta^{(\mathrm{reg})}(p_{\max}, n_r, n_e, D)}) \simeq \inf_{u \in C^\infty(\bar{\Omega}, \mathbb{R}^{d_2})} \mathscr{R}_n^{(\mathrm{reg})}(u)$ (Theorem \ref{thm:approximation}).
Keeping $n_e$, $n_r$, and $\lambda_{\text{ridge}}$ fixed, the proximity between the PINN  and $u^\star$ is measured by
\[\mathrm{err}(n)  = 2\int_0^{0.5}\int_0^1 \|u^{(n)}_{\hat \theta^{(\mathrm{reg})}(p_{\max}, n_r, n_e, D)}(x, t)-u^\star(x, t)\|_2^2 dxdt.\]
According to Theorem \ref{cor:sPINNsConsistency}, there exists some constant $\Lambda > 0$ such that, approximately, 
\[\ln\big(\mathbb{E}(\mathrm{err}(n))\big) \lesssim \ln(\Lambda) - \frac{\ln(n)}{2}.\]
This bound is validated numerically in Figure \ref{fig:linReg}, attesting a linear rate in log-log scale between $\mathrm{err}(n)$ and $n$ of $-0.69 \leqslant -0.5$. 
Furthermore, the second statement of Theorem \ref{cor:sPINNsConsistency} suggests that $\ln \mathrm{PI}(n) = \ln \mathrm{PI}(u^{(n)}_{\hat \theta^{(\mathrm{reg})}(p_{\max}, n_r, n_e, D)}) \leqslant \ln \mathrm{PI}(u^\star) = -1.6$, which is also verified in Figure \ref{fig:linReg}.
\begin{figure}
    \centering
    \begin{tabular}{cc}
    \includegraphics[width = 0.45\textwidth]{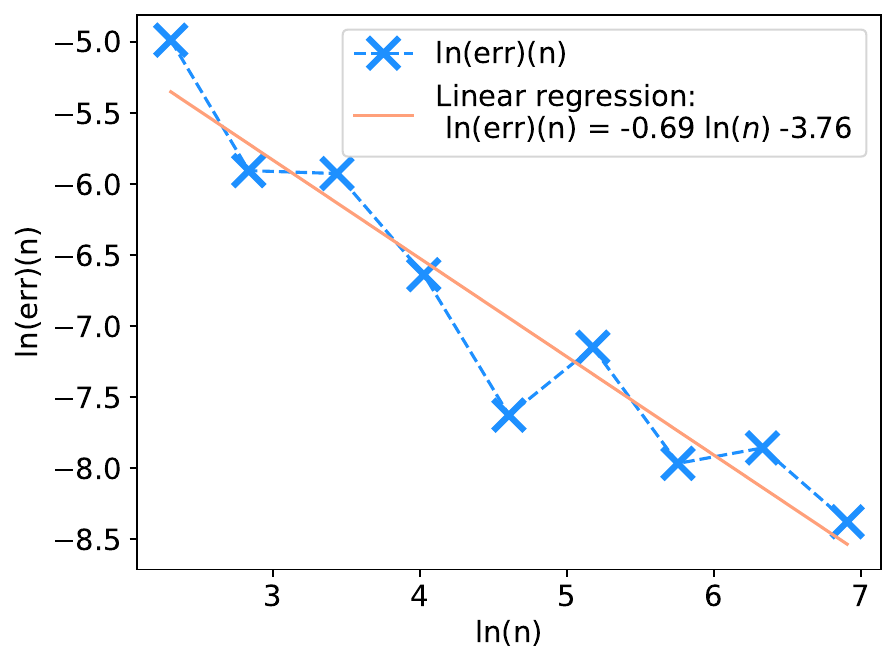} &
    \includegraphics[width = 0.45\textwidth]{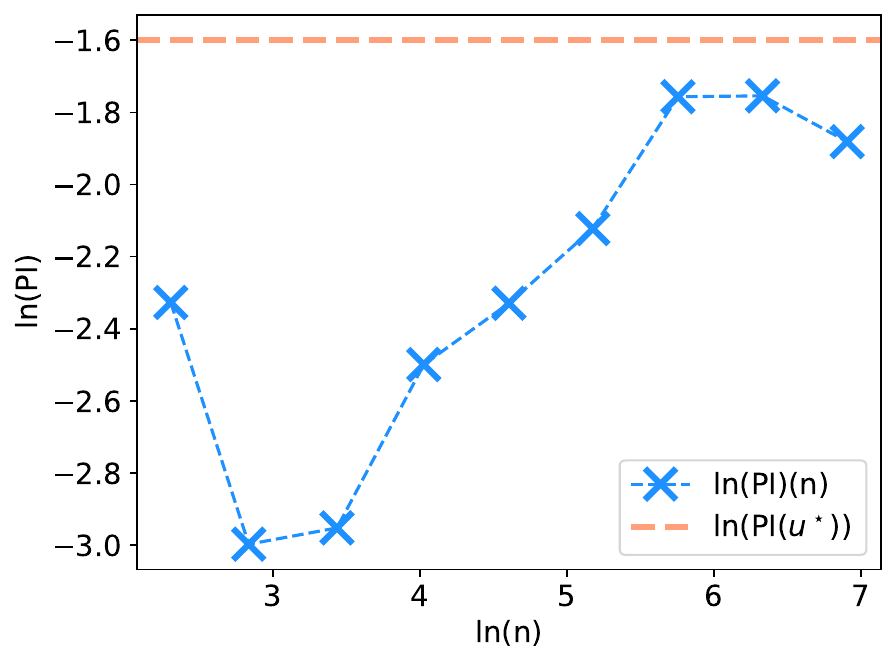} \\
    \end{tabular}
    \caption{Distance $\mathrm{err}(n)$ to  $u^\star$ (left) and physics inconsistency $\mathrm{PI}$ (right) of the regularized PINN estimator with respect to the number $n$ of observations in $\log$-$\log$ scale.}
    \label{fig:linReg}
\end{figure}
Interestingly, the regularized PINN estimator quickly becomes more accurate than the initial model, since $\mathrm{err}(n)$ is less than $ \int_\Omega\| u_{\mathrm{model}}-u^\star\|_2^2d\mu_\bX \simeq \exp(-5.3)$ as soon as $\ln(n) > 2.8$, i.e., $n\geqslant 17$.

\begin{figure}
    \centering
    \includegraphics[width = 0.32\textwidth]{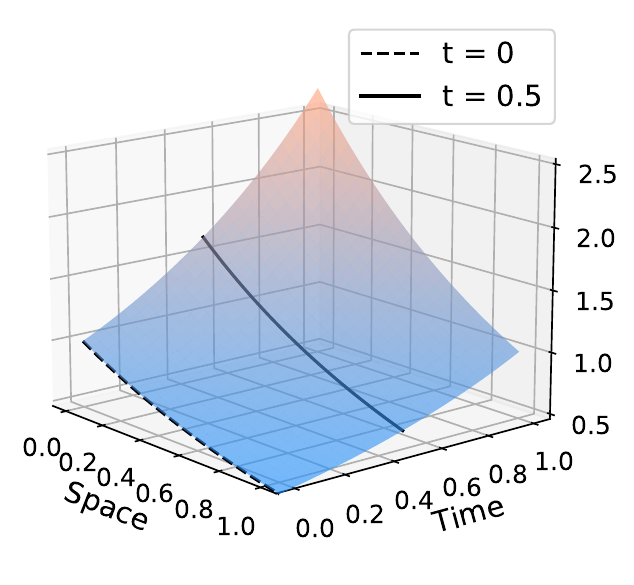}
    \includegraphics[width = 0.32\textwidth]{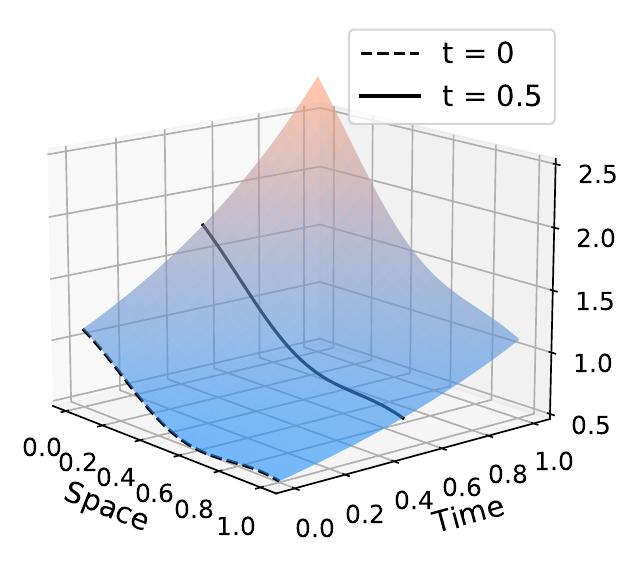}
    \includegraphics[width = 0.32\textwidth]{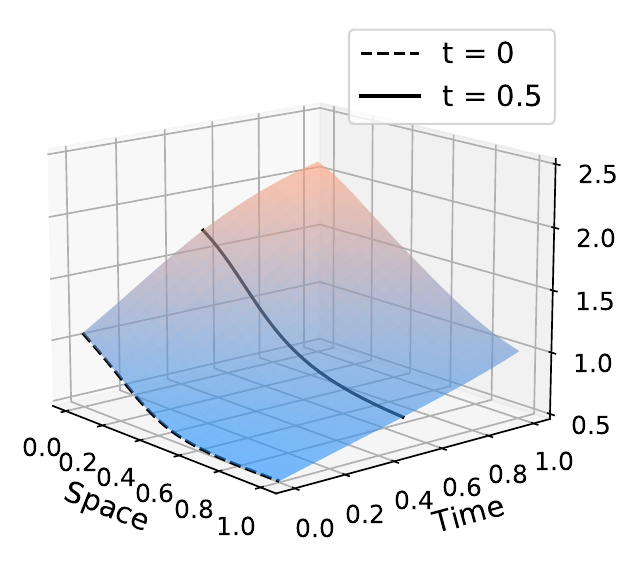}
    \caption{Functions $u_{\textrm{model}}$ (left), $u^\star$ (middle), and regularized PINN estimator with $n=10^3$ (right).}
    \label{fig:hybrid_modelling_nn}
\end{figure}
The obtained regularized PINN estimator for $n=10^3$ is shown in Figure \ref{fig:hybrid_modelling_nn} (right). This estimator looks globally similar to the model $u_{\mathrm{model}}$ (Figure \ref{fig:hybrid_modelling_nn}, left) while managing to reconstruct the variation typical of the cosine perturbation of $u^\star$ (Figure \ref{fig:hybrid_modelling_nn}, middle) at $t = 0$. 
Of course, for $t\geqslant 0.5$, the estimator cannot approximate $u^\star$  
with an infinite precision, since the measurements $\bX_i$ are only sampled for $t < 0.5$. 
However, the regularized PINN estimator succeeds to follow the advection equation dynamics, as it does not vary much along the lines $x-t=\mathrm{cst}$--- despite some flattening effect of the Sobolev regularization for $t \geqslant 0.5$. 

\section{Conclusion}
We have shown that unregularized PINNs can overfit. To remedy this problem, we have proposed to add a ridge penalty to the empirical risk. This regularization ensures the consistency of the PINNs for both linear and nonlinear PDE systems. However, to enforce strong convergence to the target function, another layer of regularization is needed. For linear PDEs, we have proved that the addition of a Sobolev-type penalty is sufficient to ensure the strong convergence of the PINNs. Regarding future research, the next step would be to derive tighter bounds to better quantify the impact of the physical penalty on the convergence speed.
\bibliographystyle{plainnat}
\bibliography{biblio}       

\begin{thebibliography}{63}
\providecommand{\natexlab}[1]{#1}
\providecommand{\url}[1]{\texttt{#1}}
\expandafter\ifx\csname urlstyle\endcsname\relax
  \providecommand{\doi}[1]{doi: #1}\else
  \providecommand{\doi}{doi: \begingroup \urlstyle{rm}\Url}\fi

\bibitem[Agranovich(2015)]{agranovich2015lispchitz}
M.S. Agranovich.
\newblock \emph{Sobolev Spaces, Their Generalizations and Elliptic Problems in
  Smooth and Lipschitz Domains}.
\newblock Springer, Cham, 2015.
\newblock \doi{10.1007/978-3-319-14648-5_2}.

\bibitem[Arnone et~al.(2022)Arnone, Kneip, Nobile, and
  Sangalli]{arnone2022spatialRegression}
E.~Arnone, A.~Kneip, F.~Nobile, and L.M. Sangalli.
\newblock Some first results on the consistency of spatial regression with
  partial differential equation regularization.
\newblock \emph{Stat. Sinica}, 32:\penalty0 209--238, 2022.
\newblock \doi{10.5705/ss.202019.0346}.

\bibitem[Arzani et~al.(2021)Arzani, Wang, and D'Souza]{arzani2021uncovering}
A.~Arzani, J.-X. Wang, and R.M. D'Souza.
\newblock Uncovering near-wall blood flow from sparse data with
  physics-informed neural networks.
\newblock \emph{Phys. Fluids}, 33:\penalty0 071905, 2021.
\newblock \doi{10.1063/5.0055600}.

\bibitem[Azzimonti et~al.(2015)Azzimonti, Sangalli, Secchi, Domanin, and
  Nobile]{azzimonti2015blood}
L.~Azzimonti, L.M. Sangalli, P.~Secchi, M.~Domanin, and F.~Nobile.
\newblock Blood flow velocity field estimation via spatial regression with
  {PDE} penalization.
\newblock \emph{J. Amer. Statist. Assoc.}, 110:\penalty0 1057--1071, 2015.
\newblock \doi{10.1080/01621459.2014.946036}.

\bibitem[Brezis(2010)]{brezis2010functional}
H.~Brezis.
\newblock \emph{Functional Analysis, Sobolev Spaces and Partial Differential
  Equations}.
\newblock Springer, New York, 2010.
\newblock \doi{10.1007/978-0-387-70914-7}.

\bibitem[Cai et~al.(2021)Cai, Wang, Wang, Perdikaris, and
  Karniadakis]{cai2021physics}
S.~Cai, Z.~Wang, S.~Wang, P.~Perdikaris, and G.E. Karniadakis.
\newblock Physics-informed neural networks for heat transfer problems.
\newblock \emph{J. Heat. Transf.}, 143:\penalty0 060801, 2021.
\newblock \doi{10.1115/1.4050542}.

\bibitem[Chandrajit et~al.(2023)Chandrajit, McLennan, Andeen, and
  Roy]{chandrajit2023recipes}
B.~Chandrajit, L.~McLennan, T.~Andeen, and A.~Roy.
\newblock Recipes for when physics fails: {R}ecovering robust learning of
  physics informed neural networks.
\newblock \emph{Mach. Learn.: Sci. Technol.}, 4:\penalty0 015013, 2023.
\newblock \doi{10.1088/2632-2153/acb416}.

\bibitem[Comtet(1974)]{comtet1974advanced}
L.~Comtet.
\newblock \emph{Advanced Combinatorics : The Art of Finite and Infinite
  Expansions}.
\newblock Springer, Dordrecht, 1974.
\newblock \doi{10.1007/978-94-010-2196-8}.

\bibitem[Costabal et~al.(2020)Costabal, Yang, Perdikaris, Hurtado, and
  Kuhl]{costabal2020physics}
F.S. Costabal, Y.~Yang, P.~Perdikaris, D.E. Hurtado, and E.~Kuhl.
\newblock Physics-informed neural networks for cardiac activation mapping.
\newblock \emph{AIP Conf. Proc.}, 8:\penalty0 42, 2020.
\newblock \doi{10.3389/fphy.2020.00042}.

\bibitem[Cunha et~al.(2023)Cunha, Droz, Zine, Foulard, and
  Ichchou]{cunha2022review}
B.~Cunha, C.~Droz, A.~Zine, S.~Foulard, and M.~Ichchou.
\newblock A review of machine learning methods applied to structural dynamics
  and vibroacoustic.
\newblock \emph{Mech. Syst. Signal. Pr.}, page 110535, 2023.
\newblock \doi{10.1016/j.ymssp.2023.110535}.

\bibitem[Cuomo et~al.(2022)Cuomo, Cola, Giampaolo, Rozza, Raissi, and
  Piccialli]{cuomo2022scientific}
S.~Cuomo, V.S.~Di Cola, F.~Giampaolo, G.~Rozza, M.~Raissi, and F.~Piccialli.
\newblock Scientific machine learning through physics-informed neural networks:
  {W}here we are and what's next.
\newblock \emph{J. Sci. Comput.}, 92:\penalty0 88, 2022.
\newblock \doi{10.1007/s10915-022-01939-z}.

\bibitem[Davini et~al.(2021)Davini, Samineni, Thomas, Tran, Zhu, Ha, Dasika,
  and White]{davini2021using}
D.~Davini, B.~Samineni, B.~Thomas, A.H. Tran, C.~Zhu, K.~Ha, G.~Dasika, and
  L.~White.
\newblock Using physics-informed regularization to improve extrapolation
  capabilities of neural networks.
\newblock In \emph{Fourth Workshop on Machine Learning and the Physical
  Sciences (NeurIPS 2021)}, 2021.

\bibitem[Daw et~al.(2022)Daw, Karpatne, Watkins, Read, and Kumar]{daw2022lake}
A.~Daw, A.~Karpatne, W.D. Watkins, J.S. Read, and V.~Kumar.
\newblock Physics-guided neural networks ({PGNN}): {A}n application in lake
  temperature modeling.
\newblock In A.~Karpatne, R.~Kannan, and V.~Kumar, editors, \emph{Knowledge
  guided machine learning: {A}ccelerating discovery using scientific knowledge
  and data}, pages 352--372, New York, 2022. Chapman and Hall/CRC.
\newblock \doi{10.1201/9781003143376-15}.

\bibitem[de~B{\'e}zenac et~al.(2019)de~B{\'e}zenac, Pajot, and
  Gallinari]{bezenac2017processes}
E.~de~B{\'e}zenac, A.~Pajot, and P.~Gallinari.
\newblock Deep learning for physical processes: {I}ncorporating prior
  scientific knowledge.
\newblock \emph{J. Stat. Mech.-Theory E.}, page 124009, 2019.
\newblock \doi{10.1088/1742-5468/ab3195}.

\bibitem[{De Ryck} and Mishra(2022)]{ryck2022kolmogorov}
T.~{De Ryck} and S.~Mishra.
\newblock Error analysis for physics informed neural networks ({PINN}s)
  approximating {K}olmogorov {PDE}s.
\newblock \emph{Adv. Comput. Math.}, 48:\penalty0 79, 2022.
\newblock \doi{10.1007/s10444-022-09985-9}.

\bibitem[{De Ryck} et~al.(2021){De Ryck}, Lanthaler, and
  Mishra]{ryck2021approximation}
T.~{De Ryck}, S.~Lanthaler, and S.~Mishra.
\newblock On the approximation of functions by tanh neural networks.
\newblock \emph{Neural Netw.}, 143:\penalty0 732--750, 2021.
\newblock \doi{10.1016/j.neunet.2021.08.015}.

\bibitem[de~Wolff et~al.(2021)de~Wolff, Carrillo, Mart{\'\i}, and
  Sanchez-Pi]{wolff2021ocean}
T.~de~Wolff, H.~Carrillo, L.~Mart{\'\i}, and N.~Sanchez-Pi.
\newblock Towards optimally weighted physics-informed neural networks in ocean
  modelling.
\newblock \emph{arXiv:2106.08747}, 2021.
\newblock \doi{10.48550/arXiv.2106.08747}.

\bibitem[Doumèche et~al.(2024{\natexlab{a}})Doumèche, Biau, and
  Boyer]{supplement2023}
Nathan Doumèche, Gérard Biau, and Claire Boyer.
\newblock Supplement to "on the convergences of pinns".
\newblock 2024{\natexlab{a}}.

\bibitem[Doumèche et~al.(2024{\natexlab{b}})Doumèche, Biau, and
  Boyer]{supplement2023code}
Nathan Doumèche, Gérard Biau, and Claire Boyer.
\newblock Code of "on the convergences of pinns".
\newblock 2024{\natexlab{b}}.

\bibitem[Esfahani(2023)]{esfahani2023adatadriven}
I.C. Esfahani.
\newblock A data-driven physics-informed neural network for predicting the
  viscosity of nanofluids.
\newblock \emph{AIP Adv.}, 13:\penalty0 025206, 2023.
\newblock \doi{10.1063/5.0132846}.

\bibitem[Evans(2010)]{evans2010partial}
L.C. Evans.
\newblock \emph{Partial Differential Equations}, volume~19 of \emph{Graduate
  Studies in Mathematics}.
\newblock American Mathematical Society, Providence, 2nd edition, 2010.
\newblock \doi{10.1090/gsm/019}.

\bibitem[Ferraccioli et~al.(2022)Ferraccioli, Sangalli, and
  Finos]{ferraccioli2022some}
F.~Ferraccioli, L.M. Sangalli, and L.~Finos.
\newblock Some first inferential tools for spatial regression with differential
  regularization.
\newblock \emph{J. Multivariate Anal.}, 189:\penalty0 104866, 2022.
\newblock \doi{10.1016/j.jmva.2021.104866}.

\bibitem[Fischer and Steinwart(2020)]{fischer2020sobolev}
S.~Fischer and I.~Steinwart.
\newblock Sobolev norm learning rates for regularized least-squares algorithm.
\newblock \emph{J. Mach. Learn. Res.}, 21:\penalty0 8464--8501, 2020.
\newblock \doi{10.48550/arXiv.1702.07254}.

\bibitem[Gokhale et~al.(2022)Gokhale, Claessens, and
  Develder]{gokhale2022thermal}
G.~Gokhale, B.~Claessens, and C.~Develder.
\newblock Physics informed neural networks for control oriented thermal
  modeling of buildings.
\newblock \emph{Appl. Energ.}, 314:\penalty0 118852, 2022.
\newblock \doi{10.1016/j.apenergy.2022.118852}.

\bibitem[Grisvard(2011)]{grisvard1985elliptic}
P.~Grisvard.
\newblock \emph{Elliptic Problems in Nonsmooth Domains}, volume~69 of
  \emph{Classics in Applied Mathematics}.
\newblock SIAM, Philadelphia, 2011.
\newblock \doi{10.1137/1.9781611972030}.

\bibitem[Guo et~al.(2017)Guo, Pleiss, Sun, and Weinberger]{guo2017on}
C.~Guo, G.~Pleiss, Y.~Sun, and K.Q. Weinberger.
\newblock On calibration of modern neural networks.
\newblock In D.~Precup and Y.W. Teh, editors, \emph{Proceedings of the 34th
  International Conference on Machine Learning}, volume~70 of \emph{Proceedings
  of Machine Learning Research}, pages 1321--1330. PMLR, 2017.
\newblock \doi{10.48550/arXiv.1706.04599}.

\bibitem[Hao et~al.(2022)Hao, Liu, Zhang, Ying, Feng, Su, and
  Zhu]{Hao2022review}
Z.~Hao, S.~Liu, Y.~Zhang, C.~Ying, Y.~Feng, H.~Su, and J.~Zhu.
\newblock Physics-informed machine learning: {A} survey on problems, methods
  and applications.
\newblock \emph{arXiv:2211.08064}, 2022.
\newblock \doi{10.48550/arXiv.2211.08064}.

\bibitem[Hardy(2006)]{hardy2006combinatorics}
M.~Hardy.
\newblock Combinatorics of partial derivatives.
\newblock \emph{Electron. J. Comb.}, 13:\penalty0 R1, 2006.
\newblock \doi{10.48550/arXiv.math/0601149}.

\bibitem[He et~al.(2020)He, Barajas-Solano, Tartakovsky, and
  Tartakovsky]{he2020subsurface}
Q.~He, D.~Barajas-Solano, G.~Tartakovsky, and A.M. Tartakovsky.
\newblock Physics-informed neural networks for multiphysics data assimilation
  with application to subsurface transport.
\newblock \emph{Adv. Water. Resourc.}, 141:\penalty0 103610, 2020.
\newblock \doi{10.1016/j.advwatres.2020.103610}.

\bibitem[Jagtap et~al.(2020)Jagtap, Kawaguchi, and
  Karniadakis]{jagtap2020adaptive}
A.D. Jagtap, K.~Kawaguchi, and G.E. Karniadakis.
\newblock Adaptive activation functions accelerate convergence in deep and
  physics-informed neural networks.
\newblock \emph{J. Comput. Phys.}, 404:\penalty0 109136, 2020.
\newblock \doi{10.1016/j.jcp.2019.109136}.

\bibitem[Kapusuzoglu and Mahadevan(2020)]{kapusuzoglu2020manufacturing}
B.~Kapusuzoglu and S.~Mahadevan.
\newblock Physics-informed and hybrid machine learning in additive
  manufacturing: {Application} to fused filament fabrication.
\newblock \emph{JOM-US}, 72:\penalty0 4695--4705, 2020.
\newblock \doi{10.1007/s11837-020-04438-4}.

\bibitem[Karniadakis et~al.(2021)Karniadakis, Kevrekidis, Lu, Perdikaris, Wang,
  and Yang]{karniadakis2021piml}
G.E. Karniadakis, I.G. Kevrekidis, L.~Lu, P.~Perdikaris, S.~Wang, and L.~Yang.
\newblock Physics-informed machine learning.
\newblock \emph{Nat. Rev. Phys.}, 3:\penalty0 422--440, 2021.
\newblock \doi{10.1038/s42254-021-00314-5}.

\bibitem[Krishnapriyan et~al.(2021)Krishnapriyan, Gholami, Zhe, Kirby, and
  Mahoney]{krishnapriyan2021characterizing}
A.~Krishnapriyan, A.~Gholami, S.~Zhe, R.~Kirby, and M.W. Mahoney.
\newblock Characterizing possible failure modes in physics-informed neural
  networks.
\newblock In M.~Ranzato, A.~Beygelzimer, Y.~Dauphin, P.S. Liang, and J.~Wortman
  Vaughan, editors, \emph{Advances in Neural Information Processing Systems},
  volume~34, pages 26548--26560. Curran Associates, Inc., 2021.
\newblock \doi{10.48550/arXiv.2109.01050}.

\bibitem[Krogh and Hertz(1991)]{krogh1991a}
A.~Krogh and J.~Hertz.
\newblock A simple weight decay can improve generalization.
\newblock In J.~Moody, S.~Hanson, and R.P. Lippmann, editors, \emph{Advances in
  Neural Information Processing Systems}, volume~4, pages 950--957.
  Morgan-Kaufmann, 1991.

\bibitem[Li et~al.(2023)Li, Wang, Di, Wang, Wang, and Zhou]{li2023aphysics}
S.~Li, G.~Wang, Y.~Di, L.~Wang, H.~Wang, and Q.~Zhou.
\newblock A physics-informed neural network framework to predict {3D}
  temperature field without labeled data in process of laser metal deposition.
\newblock \emph{Eng. Appl. Artif. Intel.}, 120:\penalty0 105908, 2023.
\newblock \doi{10.1016/j.engappai.2023.105908}.

\bibitem[Linardatos et~al.(2021)Linardatos, Papastefanopoulos, and
  Kotsiantis]{linardatos2021explainability}
P.~Linardatos, V.~Papastefanopoulos, and S.~Kotsiantis.
\newblock Explainable {AI}: {A} review of machine learning interpretability
  methods.
\newblock \emph{Entropy}, 23:\penalty0 18, 2021.
\newblock \doi{10.3390/e23010018}.

\bibitem[Loshchilov and Hutter(2019)]{loshchilov2019decoupled}
I.~Loshchilov and F.~Hutter.
\newblock Decoupled weight decay regularization.
\newblock In \emph{7th International Conference on Learning Representations},
  2019.
\newblock \doi{10.48550/arXiv.1711.05101}.

\bibitem[Mishra and Molinaro(2023)]{mishra2022generalization}
S.~Mishra and R.~Molinaro.
\newblock Estimates on the generalization error of physics-informed neural
  networks for approximating {PDE}s.
\newblock \emph{IMA J. Numer. Anal.}, 43:\penalty0 1--43, 2023.
\newblock \doi{10.1093/imanum/drab093}.

\bibitem[Nabian and Meidani(2020)]{nabian2019engineering}
M.A. Nabian and H.~Meidani.
\newblock Physics-driven regularization of deep neural networks for enhanced
  engineering design and analysis.
\newblock \emph{J. Comput. Inf. Sci. Eng.}, 20:\penalty0 011006, 2020.
\newblock \doi{10.1115/1.4044507}.

\bibitem[Nickl and P{\"o}tscher(2007)]{nickl2007bracketing}
R.~Nickl and B.M. P{\"o}tscher.
\newblock Bracketing metric entropy rates and empirical central limit theorems
  for function classes of {B}esov- and {S}obolev-type.
\newblock \emph{J. Theor. Probab.}, 20:\penalty0 177--199, 2007.
\newblock \doi{10.1007/s10959-007-0058-1}.

\bibitem[Pannell et~al.(2022)Pannell, Rigby, and Panoutsos]{pannell2022blast}
J.J. Pannell, S.E. Rigby, and G.~Panoutsos.
\newblock Physics-informed regularisation procedure in neural networks: {An}
  application in blast protection engineering.
\newblock \emph{Int. J. Prot. Struct.}, 13:\penalty0 555--578, 2022.
\newblock \doi{10.1177/20414196211073501}.

\bibitem[Qian et~al.(2023)Qian, Zhang, Huang, and Dong]{qian2023error}
Y.~Qian, Y.~Zhang, Y.~Huang, and S.~Dong.
\newblock Physics-informed neural networks for approximating dynamic
  (hyperbolic) {PDE}s of second order in time: {E}rror analysis and algorithms.
\newblock \emph{J. Comput. Phys.}, 495:\penalty0 112527, 2023.
\newblock \doi{10.1016/j.jcp.2023.112527}.

\bibitem[Rai and Sahu(2020)]{rai2020review}
R.~Rai and C.K. Sahu.
\newblock Driven by data or derived through physics? {A} review of hybrid
  physics guided machine learning techniques with cyber-physical system ({CPS})
  focus.
\newblock \emph{IEEE Access}, 8:\penalty0 71050--71073, 2020.
\newblock \doi{10.1109/ACCESS.2020.2987324}.

\bibitem[Raissi et~al.(2019)Raissi, Perdikaris, and
  Karniadakis]{raissi2019PINN}
M.~Raissi, P.~Perdikaris, and G.E. Karniadakis.
\newblock Physics-informed neural networks: {A} deep learning framework for
  solving forward and inverse problems involving nonlinear partial differential
  equations.
\newblock \emph{J. Comput. Phys.}, 378:\penalty0 686--707, 2019.
\newblock \doi{10.1016/j.jcp.2018.10.045}.

\bibitem[Ramezankhani et~al.(2022)Ramezankhani, Nazemi, Narayan, Voggenreiter,
  Harandi, Seethaler, and Milani]{ramezankhani2022multifidelity}
M.~Ramezankhani, A.~Nazemi, A.~Narayan, H.~Voggenreiter, M.~Harandi,
  R.~Seethaler, and A.S. Milani.
\newblock A data-driven multi-fidelity physics-informed learning framework for
  smart manufacturing: {A} composites processing case study.
\newblock In \emph{2022 IEEE 5th International Conference on Industrial
  Cyber-Physical Systems (ICPS)}, pages 01--07. IEEE, 2022.
\newblock \doi{10.1109/ICPS51978.2022.9816983}.

\bibitem[Riel et~al.(2021)Riel, Minchew, and Bischoff]{riel2021glacier}
B.~Riel, B.~Minchew, and T.~Bischoff.
\newblock Data-driven inference of the mechanics of slip along glacier beds
  using physics-informed neural networks: {Case} study on {R}utford {I}ce
  {S}tream, {A}ntarctica.
\newblock \emph{J. Adv. Model. Earth Syst.}, 13:\penalty0 e2021MS002621, 2021.
\newblock \doi{10.1029/2021MS002621}.

\bibitem[Rogers and Williams(2000)]{rogers1979diffusions}
L.C.G. Rogers and D.~Williams.
\newblock \emph{Diffusions, Markov processes and Martingales}, volume 1,
  Foundations.
\newblock Cambridge University Press, Cambridge, 2nd edition, 2000.
\newblock \doi{10.1017/CBO9780511805141}.

\bibitem[Sangalli(2021)]{sangalli2021spatial}
L.M. Sangalli.
\newblock Spatial regression with partial differential equation regularisation.
\newblock \emph{Int. Stat. Rev.}, 89:\penalty0 505--531, 2021.
\newblock \doi{10.1111/insr.12444}.

\bibitem[Shin(2020)]{shin2020convergence}
Y.~Shin.
\newblock On the convergence of physics informed neural networks for linear
  second-order elliptic and parabolic type {PDEs}.
\newblock \emph{Commun. Comput. Phys.}, 28:\penalty0 2042--2074, 2020.
\newblock \doi{10.4208/cicp.OA-2020-0193}.

\bibitem[Shin et~al.(2023)Shin, Zhang, and Karniadakis]{shin2020errorEstimates}
Y.~Shin, Z.~Zhang, and G.E. Karniadakis.
\newblock Error estimates of residual minimization using neural networks for
  linear {PDE}s.
\newblock \emph{JMLMC}, 4\penalty0 (4):\penalty0 73--101, 2023.
\newblock \doi{10.1615/JMachLearnModelComput.2023050411}.

\bibitem[Shvartzman(2010)]{shvartzman2010sobolev}
P.~Shvartzman.
\newblock On {S}obolev extension domains in {$\mathbb{R}^n$}.
\newblock \emph{J. Funct. Anal.}, 258:\penalty0 2205--2245, 2010.
\newblock \doi{10.1016/j.jfa.2010.01.002}.

\bibitem[Song et~al.(2021)Song, Alkhalifah, and Waheed]{song2021solving}
C.~Song, T.~Alkhalifah, and U.B. Waheed.
\newblock Solving the frequency-domain acoustic vti wave equation using
  physics-informed neural networks.
\newblock \emph{Geophys. J. Int.}, 225:\penalty0 846--859, 2021.
\newblock \doi{10.1093/gji/ggab010}.

\bibitem[Stein(1970)]{stein1970lipschitz}
E.M. Stein.
\newblock \emph{Singular Integrals and Differentiability Properties of
  Functions}, volume~30 of \emph{Princeton Mathematical Series}.
\newblock Princeton University Press, Princeton, 1970.
\newblock \doi{10.1515/9781400883882}.

\bibitem[Stone(1982)]{stone1982optimal}
C.J. Stone.
\newblock Optimal global rates of convergence for nonparametric regression.
\newblock \emph{Ann. Stat.}, 10:\penalty0 1040--1053, 1982.
\newblock \doi{10.1214/aos/1176345969}.

\bibitem[van Handel(2016)]{vanhandel2016lectures}
R.~van Handel.
\newblock \emph{Probability in High Dimension}.
\newblock APC 550 Lecture Notes, Princeton University, 2016.
\newblock \doi{10.21236/ADA623999}.

\bibitem[von Rueden et~al.(2023)von Rueden, Mayer, Beckh, Georgiev,
  Giesselbach, Heese, Kirsch, Walczak, Pfrommer, Pick, Ramamurthy, Garcke,
  Bauckhage, and Schuecker]{vonRueden2021informed}
L.~von Rueden, S.~Mayer, K.~Beckh, B.~Georgiev, S.~Giesselbach, R.~Heese,
  B.~Kirsch, M.~Walczak, J.~Pfrommer, A.~Pick, R.~Ramamurthy, J.~Garcke,
  C.~Bauckhage, and J.~Schuecker.
\newblock Informed machine learning -- {A} taxonomy and survey of integrating
  prior knowledge into learning systems.
\newblock \emph{IEEE T. Knowl. Data. En.}, 35:\penalty0 614--633, 2023.
\newblock \doi{10.1109/TKDE.2021.3079836}.

\bibitem[Wang et~al.(2020{\natexlab{a}})Wang, Bentivegna, Zhou, Klein, and
  Elmegreen]{wang2020superResolution}
C.~Wang, E.~Bentivegna, W.~Zhou, L.~Klein, and B.~Elmegreen.
\newblock Physics-informed neural network super resolution for
  advection-diffusion models.
\newblock In \emph{Third Workshop on Machine Learning and the Physical Sciences
  (NeurIPS 2020)}, 2020{\natexlab{a}}.

\bibitem[Wang et~al.(2020{\natexlab{b}})Wang, Kashinath, Mustafa, Albert, and
  Yu]{wang2020turbulence}
R.~Wang, K.~Kashinath, M.~Mustafa, A.~Albert, and R.~Yu.
\newblock Towards physics-informed deep learning for turbulent flow prediction.
\newblock In \emph{Proceedings of the 26th ACM SIGKDD International Conference
  on Knowledge Discovery \& Data Mining}, pages 1457--1466. Association for
  Computing Machinery, 2020{\natexlab{b}}.
\newblock \doi{10.48550/arXiv.1911.08655}.

\bibitem[Wang et~al.(2022)Wang, Yu, and Perdikaris]{wang2022ntk}
S.~Wang, X.~Yu, and P.~Perdikaris.
\newblock When and why {PINNs} fail to train: {A} neural tangent kernel
  perspective.
\newblock \emph{J. Comput. Phys.}, 449:\penalty0 110768, 2022.
\newblock \doi{10.1016/j.jcp.2021.110768}.

\bibitem[Willard et~al.(2023)Willard, Jia, Xu, Steinbach, and
  Kumar]{willard2023integrating}
J.~Willard, X.~Jia, S.~Xu, M.~Steinbach, and V.~Kumar.
\newblock Integrating scientific knowledge with machine learning for
  engineering and environmental systems.
\newblock \emph{ACM Comput. Surv.}, 55:\penalty0 66, 2023.
\newblock \doi{10.1145/3514228}.

\bibitem[Wu et~al.(2023)Wu, Zhu, Tang, and Lu]{wu2022convergence}
S.~Wu, A.~Zhu, Y.~Tang, and B.~Lu.
\newblock Convergence of physics-informed neural networks applied to linear
  second-order elliptic interface problems.
\newblock \emph{Commun. Comput. Phys.}, 33\penalty0 (2):\penalty0 596--627,
  2023.
\newblock \doi{https://doi.org/10.4208/cicp.OA-2022-0218}.

\bibitem[Xu et~al.(2021)Xu, Zhang, Li, Du, Kawarabayashi, and
  Jegelka]{xu2021extrapolation}
K.~Xu, M.~Zhang, J.~Li, S.S. Du, K.-I. Kawarabayashi, and S.~Jegelka.
\newblock How neural networks extrapolate: {From} feedforward to graph neural
  networks.
\newblock In \emph{International Conference on Learning Representations}, 2021.
\newblock \doi{10.48550/arXiv.2009.11848}.

\bibitem[Zhang et~al.(2020)Zhang, Liu, and Sun]{zhang2020seismic}
R.~Zhang, Y.~Liu, and H.~Sun.
\newblock Physics-guided convolutional neural network ({PhyCNN}) for
  data-driven seismic response modeling.
\newblock \emph{Eng. Struct.}, 215:\penalty0 110704, 2020.
\newblock \doi{10.1016/j.engstruct.2020.110704}.

\end{thebibliography}

\begin{appendix}
\section{Notation and definitions}\label{appn}

\paragraph{Composition of functions} Given two functions $u,v : \mathbb{R} \rightarrow \mathbb{R}$, we denote by $u\circ v$ the function $u\circ v(x) = u(v(x))$. For all $k \in \mathbb{N}$, the function $u^{\circ k}$ is defined by induction as $u^{\circ 0}(x) = x$ and $u^{\circ (k+1)} = u^{\circ k}\circ u = u \circ u^{\circ k}$. The composition symbol is placed before the derivative, so that the $k$th derivative of $u^{\circ H}$ is denoted by $(u^{\circ H})^{(k)}$.

\smallskip
\paragraph{Norms} The $p$-norm $\|x\|_p$ of a vector $x = (x_1,\hdots, x_{d}) \in \mathbb R^d$ is defined by $ \|x\|_p = (\frac{1}{d}\sum_{i=1}^{d} |x_i|^p)^{1/p}$. In addition, $\|x\|_\infty = \max_{1\leqslant i \leqslant d} |x_i|$.
For a function $u : \Omega \rightarrow \mathbb{R}^{d}$, we let $ \|u\|_{L^p(\Omega)} = (\frac{1}{|\Omega|}\int_\Omega \|u\|_p^p)^{1/p}$. Similarly, $\|u\|_{\infty, \Omega} = \sup_{x \in \Omega} \|u(x)\|_\infty$. For simplicity, we sometimes write $\|u\|_{\infty}$ instead of $\|u\|_{\infty, \Omega}$.

\smallskip
\paragraph{Multi-indices and partial derivatives} For a multi-index $\alpha = (\alpha_1, \hdots, \alpha_{d_1}) \in \mathbb{N}^{d}$ and a differentiable function $u:\mathbb R^{d_1}\to \mathbb R^{d_2}$, the $\alpha$ partial derivative of $u$ is defined by $\partial^\alpha u = (\partial_{1})^{\alpha_1}\hdots (\partial_{d_1})^{\alpha_{d_1}} u$. The set of multi-indices of sum less than $k$ is defined by \[\{|\alpha|\leqslant k\} = \{(\alpha_1, \hdots, \alpha_{d_1}) \in \mathbb{N}^{d}, \alpha_1 + \cdots +\alpha_{d_1} \leqslant k\}.\] If $\alpha = 0$, $\partial^\alpha u = u$. Given two multi-indices $\alpha$ and $\beta$, we write $\alpha \leqslant \beta$ when $\alpha_i \leqslant \beta_i$ for all $1\leqslant i \leqslant d_1$. 
The set of multi-indices less than $\alpha$ is denoted by $\{\beta \leqslant \alpha\}$. For a multi-index $\alpha$ such that $|\alpha|\leqslant k$, both sets $\{|\beta|\leqslant k\}$ and $\{\beta \leqslant \alpha\}$ are contained in $\{0, \hdots, k\}^{d_1}$ and are therefore finite.

\smallskip
\paragraph{Hölder norm} For $K \in \mathbb{N}$, the Hölder norm of order $K$ of a function $u \in C^K(\Omega, \mathbb{R}^{d})$, is defined by $\|u\|_{C^K(\Omega)} = \max_{|\alpha|\leqslant K} \|\partial^\alpha u\|_{\infty, \Omega}$. 
This norm allows to bound a function as well as its derivatives. 
The space $C^K(\Omega, \mathbb{R}^{d})$ endowed with the Hölder norm $\|\cdot\|_{C^K(\Omega)}$ is a Banach space. 
$C^\infty(\bar{\Omega}, \mathbb{R}^{d_2})$ is the space of continuous functions $u:\bar{\Omega} \to \mathbb{R}^{d_2}$ satisfying $u|_\Omega \in C^\infty(\Omega, \mathbb{R}^{d_2})$ and, for all $K\in \mathbb{N}$, $\|u\|_{C^K(\Omega)} < \infty$.

\smallskip
\paragraph{Lipschitz function} Given a normed space $(V, \|\cdot\|)$, the Lipschitz norm of a function $u : V \rightarrow \mathbb{R}^{d_1}$ is defined by 
$\|u\|_{\text{Lip}} = \sup_{x,y \in V}\|u(x)-u(y)\|_2/\|x-y\|$. A function $u$ is Lipschitz if $\|u\|_{\mathrm{Lip}}<\infty$. For all $u \in C^1(V, \mathbb{R})$, $\|u\|_{\text{Lip}} \leqslant \|u\|_{C^1(V)}$.

\smallskip
\paragraph{Lipschitz surface  and domain} A surface $\Gamma \subseteq \mathbb{R}^{d_1}$ is said to be Lipschitz if locally, in a neighborhood $U(x)$ of any point $x \in \Gamma$, an appropriate rotation $r_x$ of the coordinate system transforms $\Gamma$ into the graph of a Lipschitz function $\phi_{x}$, i.e., 
\[r_x(\Gamma \cap U(x)) = \{(x_1, \hdots, x_{d - 1}, \phi_x(x_1, \hdots, x_{d - 1})), \forall (x_1, \hdots, x_d)\in r_x(\Gamma \cap U_x)\}.\]
A domain $\Omega \subseteq \mathbb{R}^{d_1}$ is said to be Lipschitz if its has Lipschitz boundary and lies on one side of it, i.e., $\phi_x < 0$ or $\phi_x > 0$ on all intersections $\Omega \cap U_x$. All manifolds with $C^1$ boundary and all convex domains are Lipschitz domains \citep[e.g.,][]{agranovich2015lispchitz}.

\smallskip
\paragraph{Sobolev spaces} Let $\Omega \subseteq \mathbb{R}^{d_1}$ be an open set. A function $v \in L^2(\Omega, \mathbb{R}^{d_2})$ is said to be the $\alpha$th weak derivative of  $u \in L^2(\Omega, \mathbb{R}^{d_2})$ if, for any $\phi \in C^\infty(\bar{\Omega}, \mathbb{R}^{d_2})$ with compact support in $\Omega$, one has
$\int_\Omega \langle v, \phi\rangle = (-1)^{|\alpha|} \int_\Omega \langle u, \partial^\alpha \phi\rangle$. This is denoted by $v = \partial^\alpha u$. For $m \in\mathbb{N}$, the Sobolev space $H^m(\Omega, \mathbb{R}^{d_2})$ is the space of all functions $u \in L^2(\Omega, \mathbb{R}^{d_2})$ such that $\partial^\alpha u$ exists for all $|\alpha|\leqslant m$. This space is naturally endowed with the norm $\|u\|_{H^m(\Omega)} = (\sum_{|\alpha|\leqslant m} |\Omega|^{-1}\|\partial^\alpha u\|_{L^2(\Omega)} ^2)^{1/2}$. For example, the function $u : \, ]-1, 1[ \to \mathbb{R}$ such that $u(x) = |x|$ is not derivable on $]-1, 1[$, but it admits $u'(x) = \mathbf{1}_{x > 0} -  \mathbf{1}_{x < 0}$ as weak derivative. Since $u' \in L^2([-1, 1], \mathbb{R})$, $u$ belongs to the Sobolev space $H^1(]-1, 1[, \mathbb{R})$. However, $u'$ has no weak derivative, and so $u \notin H^2(]-1, 1[, \mathbb{R})$. Of course, if a function $u$ belongs to the Hölder space $C^K(\bar \Omega, \mathbb{R}^{d_2})$, then it belongs to the Sobolev space $H^K(\Omega, \mathbb{R}^{d_2})$, and its weak derivatives are the usual derivatives. For more on Sobolev spaces, we refer the reader to \citet[Chapter 5]{evans2010partial}.

\section{Some reminders of functional analysis on Lipschitz domains}

\paragraph{Extension theorems} Let $\Omega \subseteq \mathbb{R}^{d_1}$ be an open set and let $K \in \mathbb{N}$ be an order of differentiation. It is not straightforward to extend a function $u \in H^K(\Omega, \mathbb{R}^{d_2})$ to a function $\tilde{u} \in H^K(\mathbb{R}^{d_1}, \mathbb{R}^{d_2})$ such that 
\[\tilde{u}|_\Omega = u|_\Omega \quad \text{and} \quad \|\tilde{u}\|_{H^K(\mathbb{R}^{d_1})} \leqslant C_\Omega \|u\|_{H^K(\Omega)},\]
for some constant $C_\Omega$ independent of $u$. This result is known as the extension theorem in \citet[][Chapter 5.4]{evans2010partial} when $\Omega$ is a manifold with $C^1$ boundary. 
However, the simplest domains in PDEs take the form $]0,L[^3\times ]0,T[$, the boundary of which is not $C^1$. Fortunately, \citet[][Theorem 5 Chapter VI.3.3]{stein1970lipschitz} provides an extension theorem for bounded Lipschitz domains.
We refer the reader to \citet{shvartzman2010sobolev} for a survey on extension theorems.

\paragraph{Example of a non-extendable domain} Let the domain $\Omega = ]-1, 1[^2 \backslash (\{0\}\times [0,1[)$ be the square $]-1, 1[^2$ from which the segment $\{0\}\times [0,1[$ has been removed. Then the function 
\[u(x, y) = \left\{\begin{array}{cl}
    0 & \quad  \text{if } x < 0  \text{ or if } y \leqslant 0\\
     \exp(-\frac{1}{y})& \quad  \text {if } x,y > 0,
\end{array}\right.\] belongs to $C^\infty(\Omega, \mathbb{R})$ but cannot be extended to $\mathbb{R}^2$, since it cannot be continuously extended to the segment $\{0\}\times [0,1[$. Notice that $\Omega$ is not a Lipschitz domain because it lies on both sides of the segment  $\{0\}\times [0,1[$, which belongs to its boundary $\partial \Omega$. 

\begin{thm}[Sobolev inequalities]
    \label{thm:sobIneq}
    Let $\Omega \subseteq \mathbb{R}^{d_1}$ be a bounded Lipschitz domain and let $m \in \mathbb{N}$. If $m  \geqslant d_1/2$, then there exists an operator $\tilde \Pi : H^{m}(\Omega, \mathbb{R}^{d_2}) \to C^0(\Omega, \mathbb{R}^{d_2})$ such that, for any $u \in H^{m}(\Omega, \mathbb{R}^{d_2})$, $\tilde \Pi(u) = u$ almost everywhere. Moreover, there exists a constant $C_\Omega >0$, depending only on $\Omega$, such that, $\|\tilde \Pi(u)\|_{\infty, \Omega} \leqslant C_\Omega \|u\|_{H^{m}(\Omega)}.$
\end{thm}
\begin{proof}
    Since $\Omega$ is a bounded Lipschitz domain, there exists a radius $r > 0$ such that $\Omega \subseteq B(0, r)$. According to the extension theorem \citep[][Theorem 5, Chapter VI.3.3]{stein1970lipschitz}, there exists a constant $C_{\Omega}>0$, depending only on $\Omega$, such that any $u\in H^{m}(\Omega, \mathbb{R}^{d_2})$ can be extended to $\tilde u\in H^{m}(B(0,r), \mathbb{R}^{d_2})$, with
    $\|\tilde u \|_{H^{m}(B(0,r))} \leqslant C_{\Omega} \| u \|_{H^{m}(\Omega)}$.
    Since $m  \geqslant d_1/2$, the Sobolev inequalities \citep[e.g.,][Chapter 5.6, Theorem 6]{evans2010partial} state that there exists a constant $\tilde C_{\Omega} >0$, depending only on $\Omega$, and a linear embedding $\Pi : H^{m}(B(0,r), \mathbb{R}^{d_2}) \to C^0(B(0,r), \mathbb{R}^{d_2})$ such that  $\|\Pi (\tilde u)\|_{\infty} \leqslant \tilde C_{\Omega} \|\tilde u \|_{H^{m}(B(0,r))}$ and  $\Pi (\tilde u) = \tilde u$ in $H^{m}(B(0,r), \mathbb{R}^{d_2})$.
    Therefore, $\tilde \Pi (u) = \Pi (\tilde u)|_{\Omega}$ and $\|\tilde \Pi (u)\|_{\infty, \Omega} \leqslant  C_{\Omega} \tilde C_{\Omega} \| u \|_{H^{m}(\Omega)}$. 
\end{proof}

\begin{defi}[Weak convergence in $L^2(\Omega)$]
    A sequence $(u_p)_{p\in \mathbb{N}} \in L^2(\Omega)^{\mathbb{N}}$ weakly converges to $u_\infty \in L^2(\Omega)$ if, for any $\phi \in L^2(\Omega)$, $ \lim_{p\to\infty }\int_\Omega \phi u_p = \int_\Omega \phi u_\infty$. This convergence is denoted by $u_p \rightharpoonup u_\infty$.
\end{defi}

The Cauchy-Schwarz inequality shows that the convergence with respect to the $L^2(\Omega)$ norm implies the weak convergence. However, the converse is not true. For example, the sequence of functions $u_p(x) = \cos(px)$ weakly converges to $0$ in $L^2([-\pi, \pi])$, whereas $\|u_p\|_{L^2([-\pi,\pi])} = 1/2$.

\begin{defi}[Weak convergence in $H^m(\Omega)$]
    A sequence $(u_p)_{p\in \mathbb{N}} \in H^m(\Omega)^{\mathbb{N}}$ weakly converges to $u_\infty \in H^m(\Omega)$  in $H^m(\Omega)$ if, for all $|\alpha|\leqslant m$, $\partial^\alpha u_p \rightharpoonup \partial^\alpha u_\infty$.
\end{defi}

\begin{thm}[Rellich-Kondrachov]
    \label{thm:rellichK}
    Let $\Omega \subseteq \mathbb{R}^{d_1}$ be a bounded Lipschitz domain and let $m \in \mathbb{N}$. Let $(u_p)_{p\in \mathbb{N}}\in H^{m+1}(\Omega, \mathbb{R}^{d_2})$ be a sequence such that $(\|u_p\|_{H^{m+1}(\Omega)})_{p\in \mathbb{N}}$ is bounded.
    There exists a function $u_\infty \in H^{m+1}(\Omega, \mathbb{R}^{d_2})$ and a subsequence of  $(u_p)_{p\in \mathbb{N}}$ that converges to $u_\infty$ both weakly in $H^{m+1}(\Omega, \mathbb{R}^{d_2})$ and with respect to the $H^m(\Omega)$ norm.
\end{thm}
\begin{proof}
    Let $r>0$ be such that $\Omega \subseteq B(0,r)$.
    According to the extension theorem of \citet[][Theorem 5, Chapter VI.3.3]{stein1970lipschitz}, there exists a constant $C_r >0$ such that each $u_p$ can be extended to $\tilde u_p \in H^{m+1}(B(0,r),\mathbb{R}^{d_2})$, with $\|\tilde u_p\|_{H^{m+1}(B(0,r))} \leqslant C_r \|u_p\|_{H^{m+1}(\Omega)}$.
    Observing that, for all $|\alpha|\leqslant m$, $\partial^\alpha \tilde u_p$ belongs to $H^1(B(0,r),\mathbb{R}^{d_2})$, the Rellich-Kondrachov compactness theorem \citep[Theorem 1, Chapter 5.7]{evans2010partial} ensures that there exists a subsequence of $(\tilde u_p)_{p \in \mathbb{N}}$ that converges to an extension of $u_\infty$ with respect to the $H^m(B(0,r))$ norm. 
    Since the subsequence is also bounded, upon passing to another subsequence, it also weakly converges in $H^{m+1}(B(0,r), \mathbb{R}^{d_2})$ to $u_\infty \in H^{m+1}(B(0,r), \mathbb{R}^{d_2})$ \citep[e.g.,][Chapter D.4]{evans2010partial}.
    Therefore, by considering the restrictions of all the previous functions to $\Omega$, we deduce that there exists a subsequence of $(u_p)_{p\in\mathbb{N}}$ that converges to $u_\infty$ both weakly in $H^{m+1}(\Omega)$ and with respect to the $H^m(\Omega)$ norm.
\end{proof}

\section{Some useful lemmas}
\label{app:lemmas}
The $n$th Bell number $B_n$ \citep{hardy2006combinatorics} corresponds to the number of partitions of the set $\{1,\hdots, n\}$. Bell numbers satisfy the relationship $B_0 = 1$ and 
\begin{equation}
    B_{n+1} = \sum_{k=0}^n \begin{pmatrix}
        n\\k
    \end{pmatrix} B_k \label{eq:bell}.
\end{equation}
For $K \geqslant 1$ and $u\in C^K(\mathbb{R}^{d_1}, \mathbb{R}^{d_2})$, the $K$th derivative of $u$ is denoted by $u^{(K)}$.

\begin{lem}[Bounding the partial derivatives of a composition of functions]
    \label{lem:boundPartialDer}
    Let $d_1, d_2 \geqslant 1$, $K \geqslant 0$,  $f \in C^{K}(\mathbb{R}^{d_1}, \mathbb{R})$, and $g \in C^{K}(\mathbb{R}, \mathbb{R}^{d_2})$. Then
    \[\|g\circ f\|_{C^K(\mathbb{R}^{d_1})} \leqslant B_{K} \|g\|_{C^K(\mathbb{R})} (1+\|f\|_{C^K(\mathbb{R}^{d_1})})^{K}.\]
\end{lem}
\begin{proof}
    Let $K_1 \leqslant K$ and let $\Pi(K_1)$ be the set of all partitions of $\{1, \hdots, K_1\}$. According to \citet[Proposition 1]{hardy2006combinatorics}, one has, for all $h\in C^{K_1}(\mathbb{R}^{K_1+d_1}, \mathbb{R})$,
    \[\partial^{K_1}_{1,2,3,\hdots, K_1} (g\circ h) = \sum_{P \in \Pi(K_1)} g^{(|P|)}\circ h \times \prod_{S \in P} \Big[\big(\prod_{j \in S}\partial_j\big) h\Big].\]
    Let  $ \alpha = (\alpha_1, \hdots, \alpha_{d_1})$ be a multi-index such that $|\alpha|=K_1$. Setting $\alpha_0 = 0$, $y_j = x_{K_1+j} + (x_{\alpha_1 + \cdots + \alpha_{j-1}}+\cdots+x_{\alpha_1 + \cdots + \alpha_{j}-1})$, and letting $h(x_1, \hdots, x_{K_1+d_1}) = f(y_1, \hdots, y_{d_1})$, 
    we are led to
    \begin{equation}
        \partial^\alpha (g\circ f) = \sum_{P \in \Pi(K_1)} g^{(|P|)}\circ f \times \prod_{S \in P} \partial^{\alpha(S)}f, \label{eq:fdbFormula}
    \end{equation} 
    where $\alpha(S) = ( |\{b \in S,\quad \alpha_1+\cdots+\alpha_{\ell-1}\leqslant b  \leqslant \alpha_1+\cdots+\alpha_\ell \}|)_{1\leqslant \ell \leqslant d_1}$. Moreover, by definition of the Bell number, $|\Pi(K_1)| = B_{K_1}$, and, by definition of a partition, $|P|\leqslant K_1$. So,
    \begin{align*}
        \|\partial^\alpha (g\circ f)\|_\infty & \leqslant B_{K_1}  \|g\|_{C^{K_1}(\mathbb{R}^{d_1})} \max_{i_1+2i_2+\cdots+ K_1 i_{K_1}=K_1}\prod_{j=1}^{K_1} \|f\|_{C^j(\mathbb{R}^{d_1})}^{i_j} \\
        &\leqslant B_{K_1} \|g\|_{C^{K_1}(\mathbb{R}^{d_1})} (1+\|f\|_{C^{K_1}(\mathbb{R}^{d_1})})^{K_1}.\nonumber
    \end{align*} 
    Since this inequality is true for all $K_1\leqslant K$ and for all $|\alpha| = K_1$, the lemma is proved.
\end{proof}

\begin{lem}[Bounding the partial derivatives of a changing of coordinates $f$]
\label{lem:boundPartialDer2}
    Let $d_1, d_2 \geqslant 1$, $K \geqslant 0$,  $f \in C^{K}(\mathbb{R}, \mathbb{R})$, and $g \in C^{K}(\mathbb{R}^{d_1}, \mathbb{R}^{d_2})$. Let $v\in C^K(\mathbb{R}^{d_1}, \mathbb{R}^{d_1})$ be defined by $v(\bx) = (f(x_1),\hdots, f(x_{d_1}))$.
    Then
    \[\|g\circ v\|_{C^K(\mathbb{R}^{d_1})} \leqslant B_{K}\times \|g\|_{C^K(\mathbb{R}^{d_1})}\times (1+\|f\|_{C^K(\mathbb{R})})^{K}.\]
\end{lem}
\begin{proof} Let $\alpha = (\alpha_1, \hdots, \alpha_{d_1})$ be a multi-index such that $|\alpha|=K$. For $\bx = (x_1, \hdots, x_{d_1})$ and a fixed $i \in \{1, \hdots, d_1\}$, we let $h(t)= g(f(x_1), \hdots, f(x_{i-1}), t, f(x_{i+1}), \hdots, f(x_{d_1}))$. Clearly, $(h\circ f)^{(\alpha_i)}(x_i) = (\partial_i)^{\alpha_i} (g\circ v)(\bx)$. Thus, according to Lemma \ref{lem:boundPartialDer},
\[(h\circ f)^{(\alpha_i)} = \sum_{P_i \in \Pi(\alpha_i)} h^{(|P_i|)}\circ f \times \prod_{S_i \in P_i} f^{(|S_i|)}.\]
Therefore,
\[(\partial_i)^{\alpha_i}(g\circ v)(\bx) = \sum_{P_i \in \Pi(\alpha_i)} (\partial_i)^{|P_i|}g\circ v(\bx) \prod_{S_i \in P_i} f^{(|S_i|)}(x_i).\]
Letting $i=1$ and observing that $\partial_j f^{(|S_1|)}(x_1) =0$ for $j \neq 1$, we see that
    \[\partial^{\alpha}(g\circ v)(\bx) = \sum_{P_1 \in \Pi(\alpha_1)} \Big[\prod_{S_1 \in P_1} f^{(|S_1|)}(x_1)\Big] \times (\partial_2)^{\alpha_2}\hdots (\partial_{d_1})^{\alpha_{d_1}}[(\partial_1)^{|P_1|}g\circ v](\bx) .\]
Repeating the same procedure for $(\partial_1)^{|P_1|}g\circ v, \hdots, (\partial_1)^{|P_1|}\hdots (\partial_{d_1})^{|P_{d_1}|} g\circ v$, we obtain
\begin{align*}
    \partial^{\alpha}(g\circ v)(\bx) &= \sum_{P_1 \in \Pi(\alpha_1)} \Big[\prod_{S_1 \in P_1} f^{(|S_1|)}(x_1)]\Big] \times \cdots \\
     & \quad\cdots \times \sum_{P_{d_1} \in \Pi(\alpha_{d_1})} \Big[\prod_{S_{d_1} \in P_{d_1}} f^{(|S_{d_1}|)}(x_{d_1})]\Big]\times  (\partial_1)^{|P_1|}\hdots (\partial_{d_1})^{|P_{d_1}|}g\circ v(\bx) .
\end{align*}
Since $\sum_{S_i \in P_i} |S_i| = \alpha_i$ and $ \sum_{i=1}^{d_1} \alpha_i = K$, we conclude that 
\[\|\partial^{\alpha}(g\circ v)\|_\infty \leqslant B_{\alpha_1}\times\cdots\times B_{\alpha_{d_1}}\times \|\partial^{\alpha}g\|_\infty  (1+\|f\|_{C^K(\mathbb{R})})^K. \] 
Using the injective map $\mathcal{M} : \Pi(\alpha_1)\times\cdots\times \Pi(\alpha_{d_1}) \to \Pi(K)$ such that $\mathcal{M}(P_1,\hdots,P_{d_1}) = \cup_{i=1}^{d_1}P_i$, we have $B_{\alpha_1}\times\cdots\times B_{\alpha_{d_1}} \leqslant B_K$. This concludes the proof.
\end{proof}

\begin{lem}[Bounding hyperbolic tangent and its derivatives]
\label{lem:derTanh}
    For all $K\in \mathbb{N}$, one has 
 \[\|\tanh^{(K)}\|_{\infty} \leqslant 2^{K-1} (K+2)!\]
\end{lem}
\begin{proof}
The $\tanh$ function is a solution of the equation $y' = 1 - y^2$. An elementary induction shows that there exists a sequence of polynomials $(P_K)_{K \in \mathbb{N}}$ such that $\tanh^{(K)} = P_K(\tanh)$, with $P_0(X) = X$ and $P_{K+1}(X) = (1-X^2)\times P_K'(X)$. Clearly, $P_K$ is a real polynomial of degree $K+1$, of the form
$P_K(X) = a^{(K)}_0 + a^{(K)}_1 X + \cdots + a^{(K)}_{K+1} X^{K+1}$. 
One verifies that $a^{(K+1)}_i = (i+1)a^{(K)}_{i+1} - (i-1)a_{i-1}^{(K)}$, with $a^{(K)}_{-1} = a^{(K)}_{K+2}=0$. The largest coefficient $M(P_K) = \max_{0 \leqslant i \leqslant K+1} |a^{(K)}_i|$ of $P_K$ satisfies $M(P_{K+1}) \leqslant 2(K+1) \times M(P_K)$.  Thus, since $M(P_1) = 1$, we see that $M(P_K) \leqslant 2^{K-1} K!$\ . Recalling that $0 \leqslant \tanh \leqslant 1$, we conclude that
\[\|\tanh^{(K)}\|_{\infty} = \|P_K(\tanh)\|_\infty \leqslant (K+2)  M(P_K)  \leqslant 2^{K-1} (K+2)!\]
\end{proof}

In the sequel, for all $\theta \in \mathbb{R}$, we write $\tanh_\theta(x) = \tanh(\theta x)$. We define the sign function such that $\mathrm{sgn}(x) = \mathbf{1}_{x > 0} - \mathbf{1}_{x < 0}$.
\begin{lem}[Characterizing the limit of hyperbolic tangent in Hölder norm]
    \label{lem:compTanh}
    Let $K \in \mathbb{N}$ and $H \in \mathbb{N}^\star$. Then, for all $\varepsilon > 0$,  
    $\lim_{\theta\to\infty}\|\tanh^{\circ H}_\theta- \mathrm{sgn}\|_{C^K(\mathbb{R}\backslash ]-\varepsilon, \varepsilon[)}=0$.
\end{lem}
\begin{proof}  Fix $\varepsilon > 0$. We prove the stronger statement that, for all $m \in \mathbb{N}$, one has
$$\lim_{\theta \to \infty}\theta^m \|\tanh^{\circ H}_\theta- \mathrm{sgn}\|_{C^K(\mathbb{R}\backslash ]-\varepsilon, \varepsilon[)}= 0.$$ 
We start with the case $H=1$ and then prove the result by induction on $H$. 
Observe first, since $\tanh_\theta^{\circ H} - \mathrm{sgn}$ is an odd function, that
\[\|\tanh^{\circ H}_\theta- \mathrm{sgn}\|_{C^K(\mathbb{R}\backslash ]-\varepsilon, \varepsilon[)} = \|\tanh^{\circ H}_\theta- \mathrm{sgn}\|_{C^K([\varepsilon, \infty[)}.\]
\paragraph{The case $H=1$} Assume, to start with, that $K=0$. For all $x \geqslant \varepsilon$, one has
\begin{align*}
    \theta^m|\tanh_\theta(x)- 1| &= \frac{2\theta^m}{1+\exp(-2\theta x)}\leqslant \frac{2\theta^m}{1+\exp(-2\theta \varepsilon)}.
\end{align*}
Therefore, for all $ m \in \mathbb{N}$,  
\begin{align*}
    \theta^m\|\tanh_\theta-\mathrm{sgn}\|_{\infty, \mathbb{R}\backslash ]-\varepsilon, \varepsilon[} &= \theta^m\|\tanh_\theta-\mathrm{sgn}\|_{\infty, [\varepsilon, \infty[}\leqslant \frac{2\theta^m}{1+\exp(-2\theta \varepsilon)}\xrightarrow{\theta \rightarrow \infty} 0.
\end{align*}  
Next, to prove that the result if true for all $K \geqslant 1$, it is enough to show that, for all $m$, $$\theta^m\|\tanh_\theta^{(K)}\|_{\infty,\mathbb{R}\backslash ]-\varepsilon, \varepsilon[}\xrightarrow{\theta \rightarrow \infty} 0.$$ 
According to the proof of Lemma \autoref{lem:derTanh}, there exists a sequence of polynomials $(P_K)_{K\in \mathbb{N}}$ such that $\tanh^{(K)} = P_K(\tanh)$ and $P_{K+1}(X) = (1-X^2)\times P_K'(X)$. Since $\tanh_\theta(x) = \tanh(\theta x)$, one has
\begin{align*}
    \tanh_\theta^{(K)}(x) & =  \theta^K \tanh^{(K)}(\theta x)\\
    & = \theta^K (1-\tanh^2(\theta x))\times P'_{K-1}(\tanh(\theta x))\\
    & = \theta^K (1-\tanh(\theta x)) (1+\tanh(\theta x))\times P'_{K-1}(\tanh(\theta x)).
\end{align*}
Fix $x \geqslant \varepsilon$. Then, letting $M_K = \|P'_{K-1}\|_{\infty, [-1,1]}$, we are led to
\begin{align*}
    |\tanh_\theta^{(K)}(x)|\leqslant 2 M_K  \theta^K (1-\tanh(\theta x)) &\leqslant 4 M_K \times \frac{\theta^K}{1+\exp(2\theta x)}\\
    &\leqslant 4 M_K \times \frac{\theta^K}{1+\exp(2\theta \varepsilon)}.
\end{align*}
This shows that $\theta^m\|\tanh_\theta^{(K)}\|_{\infty, [\varepsilon, \infty[}\leqslant 4 M_K \times \frac{\theta^{K+m}}{1+\exp(2\theta \varepsilon)}$. One proves with similar arguments that the same result holds for all $x \leqslant - \varepsilon$. Thus, 
\begin{equation*}
    \theta^m\|\tanh_\theta^{(K)}\|_{\infty, \mathbb{R}\backslash  ]-\varepsilon, \varepsilon[}\leqslant 4 M_K \times \frac{\theta^{K+m}}{1+\exp(2\theta \varepsilon)} \xrightarrow{\theta \rightarrow \infty}0,
\end{equation*}
and the lemma is proved for $H=1$.
\paragraph{Induction} 
Assume that that, for all $K$ and all $m$, 
\begin{equation}
\label{Hfixe}
\theta^m \|\tanh^{\circ H}_\theta- \mathrm{sgn}\|_{C^K(\mathbb{R}\backslash ]-\varepsilon, \varepsilon[)}\xrightarrow{\theta \rightarrow \infty} 0.
\end{equation}
Our objective is to prove that, for all $K_2$ and all $m_2$,
\[\theta^{m_2}\|\tanh^{\circ (H+1)}_\theta- \mathrm{sgn}\|_{C^{K_2}(\mathbb{R}\backslash ]-\varepsilon, \varepsilon[)}\xrightarrow{\theta \rightarrow \infty} 0.\]
If $K_2 = 0$, since, for all $(x,y) \in \mathbb{R}^2$,
  $|\tanh_\theta(x) - \tanh_\theta(y)| \leqslant \theta |x-y|\times \|\tanh'\|_\infty 
    \leqslant \theta |x-y|$.
We deduce that \[\theta^{m_2}\|\tanh_\theta^{\circ (H+1)}- \tanh_\theta(\mathrm{sgn})\|_{\infty, \mathbb{R}\backslash ]-\varepsilon, \varepsilon[} \leqslant \theta^{m_2+1} \|\tanh^{\circ H}_\theta- \mathrm{sgn}\|_{\infty, \mathbb{R}\backslash ]-\varepsilon, \varepsilon[}.\]
Therefore, according to \eqref{Hfixe}, we have that 
$\lim_{\theta\to\infty}\theta^{m_2}\|\tanh_\theta^{\circ (H+1)}- \tanh_\theta(\mathrm{sgn})\|_{\infty, \mathbb{R}\backslash ]-\varepsilon, \varepsilon[} = 0$.
Since $\tanh_\theta (\mathrm{sgn}) - \mathrm{sgn}= (\tanh(\theta)-1) \mathbf{1}_{x >0 } - (\tanh(\theta)-1) \mathbf{1}_{x <0 }$, we see that, for all $m_2$, 
$$\lim_{\theta\to\infty}\theta^{m_2}\|\tanh_\theta (\mathrm{sgn}) - \mathrm{sgn}\|_{\infty, \mathbb{R}\backslash ]-\varepsilon, \varepsilon[} =0.$$ 
Using the triangle inequality, we conclude as desired that, for all $m_2$, 
\begin{equation}
\label{eq:characterizing_K2_1}
    \theta^{m_2}\|\tanh_\theta^{\circ (H+1)}- \mathrm{sgn}\|_{\infty, \mathbb{R}\backslash ]-\varepsilon, \varepsilon[} \xrightarrow{\theta \rightarrow \infty} 0.
\end{equation}
Assume now that $K_2\geqslant 1$. Since $\tanh_\theta^{\circ (H+1)} = \tanh^{\circ H}(\tanh)$, the Faà di Bruno formula \citep[e.g.,][Chapter 3.4]{comtet1974advanced} states that 
\begin{align}
    (\tanh_\theta^{\circ (H+1)})^{(K_2)} &= \sum_{m_1 + 2m_2 + \cdots + K_2 m_{K_2} = K_2} \frac{K_2!}{\prod_{i=1}^{K_2} m_i! \times i!^{m_i}}  \nonumber \\
    &\times (\tanh_\theta^{\circ H})^{(m_1+\cdots+m_{K_2})}(\tanh_\theta)\times \prod_{j=1}^{K_2}(\tanh_\theta^{(j)})^{m_j}. \nonumber
\end{align}
Notice that if $|x| \leq \mathrm{arctanh}(1/\sqrt{2})$, $|\tanh(x)|\geqslant \frac{|x|}{2}$ because by calling $f(x) = \tanh(x) - \frac{x}{2}$, $f(0) = 0$ and $f'(x) = (1-\tanh(x)^2) - \frac{1}{2} \geqslant 0$. Therefore, if $|x|\geq \varepsilon$, $|\tanh(\theta x)| \geqslant \min(\frac{1}{\sqrt{2}},\frac{\theta}{2}\varepsilon)\geqslant \varepsilon$ if $\theta \geqslant 2$ and $\varepsilon \geqslant \frac{1}{\sqrt{2}}$. This is why for $\theta \geqslant 2$ and $\varepsilon \leqslant 1$, 
\[ \| (\tanh_\theta^{\circ H})^{(m_1+\cdots+m_{K_2})}(\tanh_\theta)\|_{\infty, \mathbb{R}\backslash ]-\varepsilon, \varepsilon[} \leqslant \|(\tanh_\theta^{\circ H})^{(m_1+\cdots+m_{K_2})}\|_{\infty, \mathbb{R}\backslash ]-\varepsilon, \varepsilon[} .\]  
Therefore, from the triangular inequality on $\|\cdot\|_{\infty, \mathbb{R}\backslash ]-\varepsilon, \varepsilon[}$,
\begin{align*}
    \|(\tanh_\theta^{\circ (H+1)})^{(K_2)}\|_{\infty, \mathbb{R}\backslash ]-\varepsilon, \varepsilon[} &\leqslant \sum_{m_1 + 2m_2 + \cdots + K_2 m_{K_2} = K_2} \frac{K_2!}{\prod_{i=1}^{K_2} m_i! \times i!^{m_i}} \\
    &\! \times  \| (\tanh_\theta^{\circ H})^{(m_1+\cdots+m_{K_2})}\|_{\infty, \mathbb{R}\backslash ]-\varepsilon, \varepsilon[}  \prod_{j=1}^{K_2}\|\tanh_\theta^{(j)}\|_{\infty, \mathbb{R}\backslash ]-\varepsilon, \varepsilon[}^{m_j}.
\end{align*}
 According to the induction hypothesis \eqref{Hfixe}, one has, for all $K \geqslant 1$ and all $m \in \mathbb{N}$,  
 $$\lim_{\theta\to\infty}\theta^m \|(\tanh^{\circ H}_\theta)^{(K)}\|_{\infty, \mathbb{R}\backslash ]-\varepsilon, \varepsilon[}=0.$$ 
 We deduce from the above that for all $K_2 \geqslant 1$ and all $m_2$, 
\begin{equation}
\label{eq:characterizin_K2_geq1}
    \theta^{m_2} \|(\tanh^{\circ (H+1)}_\theta)^{(K_2)}\|_{\infty, \mathbb{R}\backslash ]-\varepsilon, \varepsilon[}\xrightarrow{\theta \rightarrow \infty} 0.
\end{equation}
Combining \eqref{eq:characterizing_K2_1} and \eqref{eq:characterizin_K2_geq1}, it comes that
$\lim_{\theta\to\infty}\theta^{m_2}\|\tanh^{\circ (H+1)}_\theta- \mathrm{sgn}\|_{C^{K_2}(\mathbb{R}\backslash ]-\varepsilon, \varepsilon[)}= 0$.
\end{proof}
\begin{cor}[Bounding hyperbolic tangent compositions and their derivatives]
\label{cor:bounding_tanh}
    Let $K\in \mathbb{N}$ and $H\in \mathbb{N}^\star$. Then, for or all $\theta \in \mathbb R$, 
   $\|(\tanh_\theta^{\circ H})^{(K)}\|_{\infty} < \infty$.
\end{cor}
\begin{proof} 
An induction as the one of Lemma \ref{lem:compTanh} shows that
$\|(\tanh_\theta^{\circ H})^{(K)}\|_{\infty, \mathbb{R}\backslash ]-\varepsilon, \varepsilon[} < \infty$. In addition, since $\tanh_\theta^{\circ H} \in C^\infty(\mathbb{R}, \mathbb{R})$, $\|(\tanh_\theta^{\circ H})^{(K)}\|_{\infty, [-\varepsilon, \varepsilon]} < \infty$.  
\end{proof}

When $d_1=d_2=1$, the observations $(\bX_1,Y_1), \hdots,$ 
$(\bX_n, Y_n) \in\mathbb{R}^{2}$ can be reordered as $(\bX_{(1)}, Y_{(1)}), \hdots,$ $(\bX_{(n)}, Y_{(n)})$ according to increasing values of the $\bX_i$, that is,
$\bX_{(1)} \leqslant \cdots \leqslant \bX_{(n)}$. Moreover, we let 
 $\mathcal{G}(n, n_r) = \{(\bX_i, Y_i), 1\leqslant i \leqslant n\}\cup \{\bX^{(r)}_j, 1\leqslant j \leqslant n_r\}$,
 and denote by $\delta(n, n_r)$ the minimum distance between two distinct points in $\mathcal{G}(n, n_r) $, i.e., 
 \begin{equation}
 \label{def:delta}
     \delta(n,n_r) = \underset{z_1\neq z_2}{\min_{z_1, z_2 \in \mathcal{G}(n, n_r)}} |z_1-z_2|.
 \end{equation}
\begin{lem}[Exact estimation with hyperbolic tangent]
\label{lem:estimationTanh}
    Assume that $d_1=d_2=1$, and let $H \geqslant 1$. Let the neural network $u_\theta \in \mathrm{NN}_H(n-1)$ be defined by
    \[u_\theta(x) =  Y_{(1)} + \sum_{i=1}^{n-1} \frac{Y_{(i+1)}-Y_{(i)}}{2} \bigg[\tanh_\theta^{\circ H}\Big(x-\bX_{(i)}-\frac{\delta(n, n_r)}{2}\Big)+1\bigg].\]
    Then, for all $1 \leqslant i \leqslant n$, \[ \lim_{\theta \to \infty}u_\theta(\bX_i) = Y_i.\]
    Moreover, for all order $K \in \mathbb{N}^\star$ of differentiation and all $1 \leqslant j \leqslant n_r$, \[\lim_{\theta \to \infty} u^{(K)}_\theta(\bX^{(r)}_j) = 0.\]
\end{lem}
\begin{proof} Applying Lemma \ref{lem:compTanh} with $\varepsilon = \nicefrac{\delta(n, n_r)}{4}$ and letting \[G = \mathbb{R}\backslash \cup_{i=1}^n ]\bX_{(i)} + \frac{1}{4}\delta(n, n_r), \bX_{(i)} + \frac{3}{4}\delta(n, n_r)[,\] one has, for all $K$,
$\lim_{\theta\to\infty }\|u_\theta - u_\infty\|_{C^K(G)} = 0$,
where \[u_\infty(x) = Y_{(1)}+\sum_{i=1}^{n-1} \big[Y_{(i+1)}-Y_{(i)}\big]\times \mathbf{1}_{x>\bX_{(i)}+\frac{\delta(n, n_r)}{2}} . \]
Clearly, for all $1\leqslant i \leqslant n$, $u_\infty(\bX_{i}) = Y_i$. Since $u'_\infty(x) = 0$ for all $x \in G$, and since $\bX^{(r)}_j \in G$ for all $1 \leqslant j \leqslant n_r$, we deduce that $u_\infty^{(K)}(\bX^{(r)}_j) = 0$. This concludes the proof.
\end{proof} 
\begin{defi}[Overfitting gap]
    \label{defi:OG}
    For any $n, n_e, n_r \in \mathbb{N}^\star$ and $\lambda_{(\mathrm{ridge})} \geqslant 0$, the overfitting gap operator $\mathrm{OG}_{n, n_e, n_r}$ is defined, for all $u \in C^\infty(\bar{\Omega}, \mathbb{R}^{d_2})$, by \[\mathrm{OG}_{n, n_e, n_r}(u) = |R_{n, n_e, n_r}^{(\mathrm{ridge})}(u)- \mathscr{R}_n(u)|.\]
\end{defi}
\begin{lem}[Monitoring the overfitting gap]
    \label{lem:overfitting_gap}
    Let $\varepsilon > 0$, $\lambda_{(\mathrm{ridge})} \geqslant 0$, $H \geqslant 2$, and $D\in \mathbb{N}^\star$. Let $n , n_e, n_r \in \mathbb{N}^\star$. Let $\hat \theta \in \Theta_{H, D}$ be a parameter such that $(i)$ $R^{(\mathrm{ridge})}_{n, n_e,n_r}(u_{\hat \theta}) \leqslant  \inf_{u \in \mathrm{NN}_H(D)} R_{n, n_e,n_r}^{(\mathrm{ridge})}(u) + \varepsilon$ and $(ii)$ $\mathrm{OG}_{n, n_e, n_r}(u_{\hat \theta})\leqslant \varepsilon$. 
    Then
    \[\mathscr{R}_n(u_{\hat \theta}) \leqslant \inf_{u \in \mathrm{NN}_H(D)} \mathscr{R}_n(u) + 2 \varepsilon + o_{n_e,n_r \to \infty}(1).\]
\end{lem}
\begin{proof}
    On the one hand, since $\mathscr{R}_n \leqslant R_{n, n_e,n_r}^{(\mathrm{ridge})} + \mathrm{OG}_{n, n_e, n_r}$, assumptions $(i)$ and $(ii)$ imply that $\mathscr{R}_n(u_{\hat \theta}) \leqslant \inf_{u \in \mathrm{NN}_H(D)} R_{n, n_e,n_r}^{(\mathrm{ridge})}(u) + 2\varepsilon$.
    On the other hand, $R_{n, n_e,n_r}^{(\mathrm{ridge})} - \mathrm{OG}_{n, n_e, n_r} \leqslant \mathscr{R}_n$. 
   The proof of Theorem \ref{thm:generalization_error} reveals that there exists a sequence $(\theta(n_e, n_r))_{n_e, n_r \in \mathbb{N}} \in \Theta_{H,D}^{\mathbb{N}}$ such that $\lim_{n_e, n_r \to \infty}\mathrm{OG}_{n, n_e, n_r}(u_{\theta(n_e, n_r)}) = 0$ and $\lim_{n_e, n_r \to \infty }\mathscr{R}_n(u_{\theta(n_e, n_r)}) =  \inf_{u \in \mathrm{NN}_H(D)} \mathscr{R}_n(u)$. 
   Thus, $\inf_{u \in \mathrm{NN}_H(D)} R_{n, n_e,n_r}^{(\mathrm{ridge})}(u) \leqslant \inf_{\mathrm{NN}_H(D)} \mathscr{R}_n(u) + o_{n_e,n_r \to \infty}(1)$.
   We deduce that $$\mathscr{R}_n(u_{\hat \theta}) \leqslant \inf_{u\in \mathrm{NN}_H(D)} \mathscr{R}_n(u) + 2\varepsilon +  o_{n_e,n_r \to \infty}(1).$$ 
\end{proof}
\begin{lem}[Minimizing sequence of the theoretical risk.]
\label{lem:min}
    Let $H,D \in \mathbb{N}^\star$. Define the sequence $(v_p)_{p\in \mathbb{N}}\in \mathrm{NN}_H(D)^\mathbb{N}$ of neural networks by $v_p(\bx) = \tanh_p\circ \tanh^{\circ (H-1)}(\bx)$. Then, for any $\lambda_e > 0$,
    \[\lim_{p \to \infty} \lambda_e (1-v_p(1))^2 + \frac{1}{2}\int_{-1}^1 \bx^2 (v_p')^2(\bx)d\bx = 0.\]
\end{lem}
\begin{proof}
$\tanh^{\circ (H-1)}$ is an increasing $C^\infty$ function such that $\tanh^{\circ (H-1)}(0) = 0$. Therefore, Lemma \ref{lem:compTanh} shows that $\lim_{p \to \infty}v_p(1) = 1$, so that $\lim_{p\to\infty}\lambda_e(1-v_p(1))^2=0$. This shows the convergence of the left-hand term of the lemma.
    
To bound the right-hand term, we have, according to the chain rule, \[|v_p'(\bx)| \leqslant p \|\tanh^{\circ (H-1)}\|_{C^1(\mathbb{R})} |\tanh'(p\tanh^{\circ (H-1)}(\bx))|,\] with  $\|\tanh^{\circ (H-1)}\|_{C^1(\mathbb{R})} < \infty$ by Corollary \ref{cor:bounding_tanh}.
Thus, \[\int_{-1}^1 \bx^2 (v_p')^2(\bx) d\bx \leqslant  \|\tanh^{\circ (H-1)}\|_{C^1(\mathbb{R})}^2 \int_{-1}^1 p^2 \bx^2 (\tanh'(p\tanh^{\circ (H-1)}(\bx)))^2d\bx.\]
Notice that $\bx^2 (\tanh'(p\tanh^{\circ (H-1)}(\bx)))^2$ is an even function, so that 
\[\int_{-1}^1 \bx^2 (v_p')^2(\bx) d\bx \leqslant 2 \|\tanh^{\circ (H-1)}\|_{C^1(\mathbb{R})}^2 \int_{0}^1 p^2\bx^2 (\tanh'(p\tanh^{\circ (H-1)}(\bx)))^2d\bx.\]
Remark that $(\tanh')^2(\bx) = (1-\tanh(\bx))^2(1+\tanh(\bx))^2 \leqslant 16 \exp(-2\bx)$, 
so that 
\[\int_{-1}^1 \bx^2 (v_p')^2(\bx) d\bx \leqslant 32 \|\tanh^{\circ (H-1)}\|_{C^1(\mathbb{R})}^2 \int_{0}^1 p^2\bx^2 \exp(-2p\tanh^{\circ (H-1)}(\bx))d\bx.\]

If $H=1$, then the change of variable $\bar{\bx} = p\bx$ states that $$\int_{0}^1 p^2\bx^2 \exp{(-2p\bx)}d\bx \leqslant p^{-1}\int_{0}^\infty \bar{\bx}^2 \exp{(-2\bar{\bx})}d\bar{\bx} \xrightarrow{p\to \infty}0$$ and the lemma is proved.

If $H\geqslant 2$, notice that $\tanh(\bx) \geqslant \bx\mathbf{1}_{\bx\leqslant 1}/2 + \mathbf{1}_{\bx\geqslant 1}/2$ for all $\bx \geqslant 0$, and therefore we have that $\tanh^{\circ (H-1)}(\bx) \geqslant \bx\mathbf{1}_{\bx\leqslant 2^{H-1}}/2^H + \mathbf{1}_{\bx\geqslant 2^{H-1}}/2^H$. 
Thus, using the change of variable $\bar{\bx} = p\bx$,
\begin{align*}
    \int_{0}^1 p^2\bx^2 \exp(-2p\tanh^{\circ (H-1)}(\bx))d\bx &\leqslant \int_{0}^{1} p^2\bx^2 \exp(-2^{H-1}p\bx)d\bx\\  
    & \leqslant p^{-1} \int_0^\infty \bar{\bx}^2 \exp(-2^{H-1}\bar{\bx})d\bar{\bx}.
\end{align*}
Since this upper bound vanishes as $p\to \infty$, this concludes the proof when $H\geqslant 2$.

\end{proof}

\begin{defi}[Weak lower semi-continuity]
    A fonction $I: H^m(\Omega)\to \mathbb{R}$ is weakly lower semi-continuous on $H^m(\Omega)$ if, for any sequence $(u_p)_{p\in \mathbb{N}} \in H^m(\Omega)^{\mathbb{N}}$ that weakly converges to $u_\infty \in H^m(\Omega)$ in $H^m(\Omega)$, one has $I(u_\infty) \leqslant \liminf_{p \to \infty} I(u_p).$ 
\end{defi}

The following technical lemma will be useful for the proof of Proposition \ref{prop:sequenceCvLin}.
\begin{lem}[Weak lower semi-continuity with convex Lagrangians]
\label{lem:lowerSemiC0Lin}
    Let the Lagrangian $L\in C^\infty(\mathbb{R}^{\binom{d_1+m}{m}d_2}\times \cdots \times \mathbb{R}^{d_2}\times \mathbb{R}^{d_1}, \mathbb{R})$ be such that, for any $x^{(m)}, \hdots, x^{(0)}$, and $z$, the function $x^{(m+1)} \mapsto L(x^{(m+1)}, \hdots, x^{(0)}, z)$
     is convex and nonnegative.
    
    Then the function $I : u \mapsto \int_{\Omega}L((\partial^{m+1}_{i_1,\hdots, i_{m+1}}u(\bx))_{1\leqslant i_1,\hdots, i_{m+1} \leqslant d_1},\hdots,u(\bx), \bx)d\bx$ is lower-semi continuous for the weak topology on $H^{m+1}(\Omega,\mathbb{R}^{d_2})$.
\end{lem}

\begin{proof}
    This results generalizes \citet[Theorem 1, Chapter 8.2]{evans2010partial}, which treats the case $m=0$. Let $(u_p)_{p\in \mathbb{N}} \in H^{m+1}(\Omega, \mathbb{R}^{d_2})^{\mathbb{N}}$ be a sequence that weakly converges  to $u_\infty \in H^{m+1}(\Omega, \mathbb{R}^{d_2})$ in $H^{m+1}(\Omega, \mathbb{R}^{d_2})$. Our goal is to prove that $I(u_\infty) \leqslant \liminf_{p \to \infty} I(u_p)$.  Upon passing to a subsequence, we can suppose that $\lim_{p\to\infty} I(u_p) =\liminf_{p\to\infty} I(u_p)$.

    As a first step, we strengthen the convergence of $(u_p)_{p \in\mathbb{N}}$ by showing that for any $\varepsilon > 0$, there exists a subset $E_\varepsilon$ of $\Omega$ such that $|\Omega\backslash E_\varepsilon|\leqslant \varepsilon$ (the notation $|\cdot|$ stands for the Lebesgue measure), and such that there exists a subsequence that uniformly converges on $E_\varepsilon$, as well as its derivatives.
    Recalling that a weakly convergent sequence is bounded \citep[e.g.,][Chapter D.4]{evans2010partial}, one has $\sup_{p\in\mathbb{N}} \|u_p\|_{H^{m+1}(\Omega)} < \infty$. 
    Theorem \ref{thm:rellichK} ensures that a subsequence of $( u_p)_{p \in \mathbb{N}}$ converges to, say, $u_\infty \in H^{m+1}(\Omega, \mathbb{R}^{d_2})$ with respect to the $H^{m}(\Omega)$ norm. 
    Upon passing again to another subsequence, we conclude that for all $|\alpha|\leqslant m$ and for almost every $x$ in $\Omega$, $\lim_{p \to \infty} \partial^\alpha u_p(x) = \partial^\alpha u_\infty(x)$ \citep[see, e.g.][Theorem 4.9]{brezis2010functional}.
    Finally, by Egorov's theorem \citep[Chapter E.2]{evans2010partial}, for any $\varepsilon >0$, there exists a measurable set $E_\varepsilon$ such that $|\Omega\backslash E_\varepsilon| \leqslant \varepsilon$ and such that, for all $|\alpha| \leqslant m$, $\lim_{p\to\infty} \|\partial^\alpha u_p-\partial^\alpha u_\infty\|_{L^\infty(E_\varepsilon)} =0$. 
    
    Our next goal is to bound the function $L$.
    Let $F_\varepsilon = \{x \in \Omega, \sum_{|\alpha|\leqslant m+1}|\partial^\alpha u_\infty(x)| \leqslant \varepsilon^{-1}\}$ and $G_\varepsilon = E_\varepsilon \cap F_\varepsilon$. 
    Observe that $\lim_{\varepsilon \to 0}|\Omega \backslash G_\varepsilon| =0$.
    Since, for all $|\alpha|\leqslant m+1$,  $\|\partial^\alpha u_\infty\|_{\infty, G_\varepsilon}< \infty$,  
    and since $\lim_{p \to\infty} \|\partial^\alpha u_p-\partial^\alpha u_\infty\|_{L^\infty(G_\varepsilon)} =0$, 
    then, for all $p$ large enough, $(\|\partial^\alpha u_p\|_{L^\infty(G_\varepsilon)})_{p\in\mathbb{N}}$ is bounded. 
    For now, for the ease of notation, we denote $((\partial^{m+1}_{i_1,\hdots, i_{m+1}}u(z))_{1\leqslant i_1,\hdots, i_{m+1} \leqslant d_1},\hdots, u(z), z)$ by $(D^{m+1}u(z),\hdots, u(z), z)$.
    Therefore, since the Lagrangian $L$ is smooth and $\Omega$ is bounded, for all $p$ large enough, $(\| L(D^{m+1}u_p(\cdot), \hdots, Du_p(\cdot), u_p(\cdot), \cdot)\|_{L^\infty(G_\varepsilon)})_{p\in\mathbb{N}}$ is bounded as well.
    
    To conclude the proof, we take advantage of the convexity of the Lagrangian $L$.
    Let $J_{m+1}$ be the Jacobian matrix of $L$ along the vector $x^{(m+1)}$.
    The convexity of $L$ implies 
    \begin{align*}
        &L(D^{m+1}u_p(z), \hdots, u_p(z), z)\\
        &\quad \geqslant L(D^{m+1}u_\infty(z), D^{m}u_p(z)\hdots, u_p(z), z) \\
        & \qquad+ J_{m+1}(D^{m+1}u_\infty(z), D^{m}u_p(z)\hdots, u_p(z), z) \times (D^{m+1}u_p(z)-D^{m+1}u_\infty(z)). 
    \end{align*}
    Using the fact that $L\geqslant0$ and that $I(u_p) \geqslant \int_{G_\varepsilon}L(D^{m+1}u_p(z), \hdots, u_p(z), z)dz$, 
    we obtain
    \begin{align*}
        I(u_p) &\geqslant \int_{G_\varepsilon } L(D^{m+1}u_\infty(z), D^m u_p(z), \hdots, u_p(z), z) \\
        &\quad + J_{m+1}(D^{m+1}u_\infty(z), D^m u_p(z), \hdots, u_p(z), z) \times (D^{m+1}u_p(z)-D^{m+1}u_\infty(z))dz.
    \end{align*}
    Since $(\| L(D^{m+1}u_p(\cdot), \hdots, Du_p(\cdot), u_p(\cdot), \cdot)\|_{L^\infty(G_\varepsilon)})_{p\in\mathbb{N}}$ is bounded for $p$ large enough, and since, for all $|\alpha| \leqslant m$, $\lim_{p\to\infty} \|\partial^\alpha u_p-\partial^\alpha u_\infty\|_{L^\infty(G_\varepsilon)} =0$, the dominated convergence theorem ensures that  
    \begin{align*}
        &\lim_{p\to \infty}\int_{G_\varepsilon }  \!\!\! L(D^{m+1}u_\infty(z), D^m u_p(z), \hdots, u_p(z), z)dz = \int_{G_\varepsilon }  \!\!\! L(D^{m+1}u_\infty(z), \hdots, u_\infty(z), z)dz. 
    \end{align*}
    Since $(i)$ $L$ is smooth (and therefore Lipschitz on bounded domains), 
   $(ii)$ for all $p$ large enough,  $(\|\partial^\alpha u_p\|_{L^\infty(G_\varepsilon)})_{p\in\mathbb{N}}$ is bounded, and $(iii)$ for all $|\alpha| \leqslant m$, $\lim_p \|\partial^\alpha u_p-\partial^\alpha u_\infty\|_{L^\infty(G_\varepsilon)} =0$, we have that $\lim_{p \to \infty }\|J_{m+1}(D^{m+1}u_\infty(\cdot), D^m u_p(\cdot), \hdots, u_p(\cdot), \cdot) - J_{m+1}(D^{m+1}u_\infty(\cdot), \hdots, u_\infty(\cdot), \cdot) \|_{L^\infty(G_\varepsilon)} = 0$. Therefore, since $D^{m+1}u_p \rightharpoonup D^{m+1}u_\infty$, 
    \[\lim_{p\to \infty}\int_{G_\varepsilon}\!\!\!J_{m+1}(D^{m+1}u_\infty(z), D^m u_p(z), \hdots, u_p(z), z) \times (D^{m+1}u_p(z)-D^{m+1}u_\infty(z))dz = 0.\]
    Hence,
    $\lim_{p\to\infty} I(u_p) \geqslant \int_{G_\varepsilon } L(D^{m+1}u_\infty(z), \hdots, u_\infty(z), z)dz$.
    Finally, applying the monotone convergence theorem with $\varepsilon \to 0$ shows that $\lim_{p\to\infty} I(u_p) \geqslant I(u_\infty)$, which is the desired result.
\end{proof}

\begin{lem}[Measurability of $\hat u_n$]
    \label{lem:measurability}
    Let $\hat u_n = \argmin_{u \in H^{m+1}(\Omega, \mathbb{R}^{d_2})} \mathscr{R}_n^{(\mathrm{reg})}(u)$, where, for all $u \in H^{m+1}(\Omega, \mathbb{R}^{d_2})$, 
    \begin{align*} 
        \mathscr{R}^{(\mathrm{reg})}_n(u) &= \frac{\lambda_d}{n} \sum_{i=1}^n \|\tilde \Pi(u)(\bX_i) - Y_i\|_2^2 +\lambda_e\mathbb{E}\|\tilde \Pi(u)(\bX^{(e)})-h(\bX^{(e)})\|_2^2\\
        & \quad + \frac{1}{|\Omega|}\sum_{k=1}^M \|\mathscr{F}_k(u, \cdot) \|_{L^2(\Omega)} + \lambda_t \|u\|_{H^{m+1}(\Omega)}^2.
    \end{align*}
    Then $\hat u_n$ is a random variable.
\end{lem}
\begin{proof}
Recall that \[
\mathscr{R}_n^{(\mathrm{reg})}(u) = \mathcal{A}_n(u,u) -2\mathcal{B}_n(u) + \frac{\lambda_d}{n} \sum_{i=1}^n \|Y_i\|^2 +\lambda_e \mathbb{E}\|h(\bX^{(e)})\|_2^2 + \frac{1}{|\Omega|}\sum_{k=1}^M\int_{\Omega} B_k(\bx)^2d\bx.
\]
Throughout we use the notation $\mathcal{A}_{(\bx, e)}(u,u)$ instead of $\mathcal{A}_n(u,u)$, to make the dependence of $\mathcal{A}_n$ in the random variables $\bx = (\bX_1, \hdots, \bX_n)$ and $e=(\varepsilon_1, \hdots, \varepsilon_n)$ more explicit. We do the same with $\mathcal{B}_n$. 
For a given a normed space $(F, \|\cdot\|)$,  we let $\mathscr{B}(F, \|\cdot\|)$ be the Borel $\sigma$-algebra on $F$ induced by the norm $\|\cdot\|$. 

Our goal is to prove that the function 
\begin{align*}
    \hat  u_n : (\Omega^n\! \times \!\mathbb{R}^{nd_2}, \mathscr{B}(\Omega^n\!\times\! \mathbb{R}^{nd_2}, \|\!\cdot\!\|_2)) &\to (H^{m+1}(\Omega, \mathbb{R}^{d_2}), \mathscr{B}(H^{m+1}(\Omega, \mathbb{R}^{d_2}), \|\!\cdot\!\|_{H^{m+1}(\Omega)}))\\
    (\bx, e) & \mapsto \argmin_{u \in H^{m+1}(\Omega, \mathbb{R}^{d_2})} \mathcal{A}_{(\bx, e)}(u,u)  - 2 \mathcal{B}_{(\bx, e)}(u)
\end{align*} is measurable. 
Recall that  $H^{m+1}(\Omega, \mathbb{R}^{d_2})$ is a Banach space separable with respect to its norm $\|\cdot\|_{H^{m+1}(\Omega)}$. 
Let $(v_q)_{q \in \mathbb{N}} \in H^{m+1}(\Omega, \mathbb{R}^{d_2})^\mathbb{N}$ be a sequence dense in $H^{m+1}(\Omega, \mathbb{R}^{d_2})$. 
Note that, for any $\bx\in \Omega^n$ and any $e \in \mathbb{R}^{nd_2}$, one has $\min_{u \in H^{m+1}(\Omega, \mathbb{R}^{d_2})} \mathcal{A}_{(\bx, e)}(u,u) - 2\mathcal{B}_{(\bx, e)}(u) = \inf_{q \in \mathbb{N}} \mathcal{A}_{(\bx, e)}(v_q,v_q) - 2\mathcal{B}_{(\bx, e)}(v_q)$. 
This identity is a consequence of the fact that the function $u \mapsto \mathcal{A}_{(\bx, e)}(u,u) - 2\mathcal{B}_{(\bx, e)}(u)$ is continuous for the $H^{m+1}(\Omega)$ norm, as shown in the proof of Proposition \ref{prop:laxMLin}).
Moreover, according to this proof, each function $F_q(\bx, e) := \mathcal{A}_{(\bx,  e)}(u_q, u_q) - 2 \mathcal{B}_{(\bx, e)}(u_q)$ is a composition of  continuous functions, and is therefore measurable. 
Thus, the function 
\begin{align*}
    G(\bx, e) :=  \min_{u \in H^{m+1}(\Omega, \mathbb{R}^{d_2})} \mathcal{A}_{(\bx, e)}(u, u) - 2 \mathcal{B}_{(\bx, e)}(u)= \inf_{q \in \mathbb{N}} \mathcal{A}_{(\bx, e)}(u_q, u_q) - 2 \mathcal{B}_{(\bx, e)}(u_q)
\end{align*} is measurable.

Next, since $\Omega$, $\mathbb{R}$, and $H^{m+1}(\Omega, \mathbb{R}^{d_2})$ are separable, we know that  the $\sigma$-algebras
$\mathscr{B}(\Omega^n\times\mathbb{R}^{nd_2}\times H^{m+1}(\Omega, \mathbb{R}^{d_2}), \|\cdot\|_{\otimes})$ and $ \mathscr{B}(\Omega^n  \times \mathbb{R}^{nd_2}, \|\cdot\|_2)\otimes \mathscr{B}(H^{m+1}(\Omega, \mathbb{R}^{d_2}), \|\cdot\|_{H^{m+1}(\Omega)})$ are identical, 
where $\|(\bx,e, u)\|_{\otimes} = \|(\bx,e)\|_2 +\|u\|_{H^{m+1}(\Omega)}$ \citep[see, e.g.][Chapter II.13, E13.11c]{rogers1979diffusions}. 
This implies that the coordinate projections $\Pi_{\bx,e}$ and $\Pi_u$---defined for $(\bx,e) \in \Omega^n\times\mathbb{R}^{nd_2}$ and $u \in H^{m+1}(\Omega, \mathbb{R}^{d_2})$ by $\Pi_{\bx,e}(\bx,e,u) = (\bx,e)$ and $\Pi_u(\bx,e,u) = u$---are $\|\cdot\|_\otimes$ measurable.
It is easy to check that, for any $(\bx,e) \in \Omega^n\times \mathbb{R}^{nd_2}$ and $u\in H^{m+1}(\Omega, \mathbb{R}^{d_2})$, if $\lim_{p \to \infty }\|(\bx_p, e_p, u_p)-(\bx, e, u)\|_\otimes =0$, then $\lim_{p\to \infty}\|\tilde \Pi(u_p)-\tilde \Pi(u)\|_{\infty, \Omega} = 0$ and, since $\tilde \Pi (u) \in C^0(\Omega, \mathbb{R}^{d_2})$, $\lim_{p\to \infty }\mathcal{A}_{\bx_p, e_p}(u_p,u_p) - 2\mathcal{B}_{\bx_p, e_p}(u_p) = \mathcal{A}_{\bx,e}(u,u) - 2\mathcal{B}_{\bx,e}(u)$.
This proves that $I : (\Omega^n\times \mathbb{R}^{nd_2} \times H^{m+1}(\Omega, \mathbb{R}^{d_2}), \mathscr{B}(\Omega^n\times \mathbb{R}^{nd_2}\times H^{m+1}(\Omega, \mathbb{R}^{d_2}), \|\cdot\|_{\otimes})) \to (\mathbb{R}, \mathscr{B}(\mathbb{R}))$ defined by 
\[I(\bx, e, u) = \mathcal{A}_{(\bx, e)}(u,u) - 2 \mathcal{B}_{(\bx, e)}(u)\] is continuous with respect to $\|\cdot\|_\otimes$ and therefore measurable. According to the above, the function \[\tilde I(\bx, e, u) = I(\bx, e, u) -  G\circ \Pi_{x,e}(\bx, e, u)\] 
is also measurable.
Observe that, by definition, $\hat u_n = J\circ (\bX_1, \hdots, \bX_n, \varepsilon_1, \hdots, \varepsilon_n)$, where $J(\bx, e) = \Pi_u(\tilde I^{-1}(\{0\})\cap (\{(\bx, e)\}\times H^{m+1}(\Omega, \mathbb{R}^{d_2})))$.
For any measurable set $S \in \mathscr{B}(H^{m+1}(\Omega, \mathbb{R}^{d_2}, \|\cdot\|_{H^{m+1}(\Omega)})$, $J^{-1}(S) = \Pi_{x,e}(\tilde{I}^{-1}(\{0\}) \cap (\Omega^n\times \mathbb{R}^{nd_2} \times S)) \in \mathscr{B}(\Omega^n\times \mathbb{R}^{nd_2})$. 
(Notice that $J^{-1}(S)$ is the collection of all pairs $(\bx, e)\in \Omega^n\times \mathbb{R}^{nd_2}$ satisfying $\argmin_{u \in H^{m+1}(\Omega, \mathbb{R}^{d_2})} \mathcal{A}_{(\bx, e)}(u,u)  - 2 \mathcal{B}_{(\bx, e)}(u) \in S$.)
To see this, jut note that for any set $\tilde S \in \mathscr{B}(\Omega^n \times \mathbb{R}^{nd_2}, \|\cdot\|_2)\otimes \mathscr{B}(H^{m+1}(\Omega, \mathbb{R}^{d_2}), \|\cdot\|_{H^{m+1}(\Omega, \mathbb{R}^{d_2})})$, one has $\Pi_{x,e}(\tilde S) \in \mathscr{B}(\Omega^n \times \mathbb{R}^{nd_2}, \|\cdot\|_2)$ \citep[see, e.g.][Lemma 11.4, Chapter II]{rogers1979diffusions}. We conclude that the function $J$ is measurable and so is $\hat u_n$.
\end{proof}

Let $B( 1, \|\cdot\|_{H^{m+1}(\Omega)}) = \{u \in H^{m+1}(\Omega, \mathbb{R}^{d_2}), \quad \|u\|_{H^{m+1}(\Omega)}\leqslant 1\}$ be the ball of radius $r$ centered at $0$. 
Let $N(B(1, \|\cdot\|_{H^{m+1}(\Omega)})), \|\cdot\|_{H^{m+1}(\Omega)}, r)$ be the minimum number of balls of radius $r$ according to the norm $\|\cdot\|_{H^{m+1}(\Omega)}$ needed to cover the space $B( 1, \|\cdot\|_{H^{m+1}(\Omega)})$.
\begin{lem}[Entropy of $H^{m+1}(\Omega, \mathbb{R}^{d_2})$]
Let $\Omega\subseteq \mathbb{R}^{d_1}$ be a Lipschitz domain. For $m \geqslant 1$, one has
    \label{lem:entropy}
    \[\log N(B( 1, \|\cdot\|_{H^{m+1}(\Omega)}), \|\cdot\|_{H^{m+1}(\Omega)}, r) = \Oequivalent_{r \to 0} (r^{-d_1/(m+1)}).\]
\end{lem}
\begin{proof}
According to the extension theorem \citep[][Theorem 5, Chapter VI.3.3]{stein1970lipschitz}, there exists a constant $C_{\Omega} >0$, depending only on $\Omega$, such that any $u\in H^{m+1}(\Omega, \mathbb{R}^{d_2})$ can be extended to $\tilde u\in H^{m+1}(\mathbb{R}^{d_1}, \mathbb{R}^{d_2})$, with
$\|\tilde u \|_{H^{m+1}(\mathbb{R}^{d_1})} \leqslant C_{\Omega} \| u \|_{H^{m+1}(\Omega)}$. 
Let $r > 0$ be such that $\Omega \subseteq B(r, \|\cdot\|_2)$ and let $\phi \in C^\infty(\mathbb{R}^{d_1}, \mathbb{R}^{d_2})$ be such that
\[\phi(\bx) = \left\{\begin{array}{ll}
    1  &\text{for } \bx \in \Omega\\
    0  &\text{for } \bx \in\mathbb{R}^{d_1}, |x|\geqslant r.
\end{array}\right. \]
Then, for any $ u \in H^{m+1}(\Omega, \mathbb{R}^{d_2})$, $(i)$ $\phi \tilde u \in H^{m+1}(\mathbb{R}^{d_1}, \mathbb{R}^{d_2})$, $(ii)$ $\phi \tilde u |_\Omega = u$, and $(iii)$ there exists a constant $\tilde C_\Omega > 0$ such that $\|\phi \tilde u\|_{H^{m+1}(\mathbb{R}^{d_1})} \leqslant \tilde C_\Omega \|u\|_{H^{m+1}(\Omega)}$.
The lemma follows from \citet[Corollary 4]{nickl2007bracketing}.     
\end{proof}

\begin{lem}[Empirical process $L^2$]
    \label{lem:empiricalL2}
    Let $\bX_1, \hdots, \bX_n$ be i.i.d.~random variables, with common distribution $\mu_\bX$ on $\Omega$. Then there exists a constant $C_\Omega >0$, depending only on $\Omega$, such that
    \[\mathbb{E}\Big(\sup_{\|u\|_{H^{m+1}(\Omega)}\leqslant 1} \mathbb{E}\|\tilde \Pi(u)(\bX)\|_2^2- \frac{1}{n}\sum_{i=1}^n \|\tilde \Pi(u)(\bX_i)\|_2^2\Big) \leqslant \frac{d_2^{1/2} C_\Omega}{n^{1/2}},\]
    and
    \[\mathbb{E}\Big(\Big(\sup_{\|u\|_{H^{m+1}(\Omega)}\leqslant 1} \mathbb{E}\|\tilde \Pi(u)(\bX)\|_2^2- \frac{1}{n}\sum_{i=1}^n \|\tilde \Pi(u)(\bX_i)\|_2^2\Big)^2\Big) \leqslant \frac{d_2 C_\Omega}{n},\]
    where $\tilde \Pi$ is the Sobolev embedding (see Theorem \ref{thm:sobIneq}).
\end{lem}
\begin{proof}
    For any $u \in H^{m+1}(\Omega, \mathbb{R}^{d_2})$, let 
\[Z_{n,u} =  \mathbb{E}\|\tilde \Pi(u)(\bX_i)\|_2^2 - \frac{1}{n} \sum_{j=1}^n \|\tilde \Pi(u)(\bX_i)\|_2^2 \quad \text{and} \quad Z_n = \sup_{\|u\|_{H^{m+1}(\Omega)}\leqslant 1} Z_{n,u}.\] 
For any $u,v \in  H^{m+1}(\Omega, \mathbb{R}^{d_2})$ such that $\|u\|_{H^{m+1}(\Omega)} \leqslant 1$ and $\|v\|_{H^{m+1}(\Omega)} \leqslant 1$, we have
\begin{align*}
    &\Big|\frac{1}{n}( \|\tilde \Pi(u)(\bX_i)\|_2^2 - \mathbb{E}\|\tilde \Pi(u)(\bX_i)\|_2^2)  - \frac{1}{n} (\|\tilde \Pi(v)(\bX_i)\|_2^2 - \mathbb{E}\|\tilde \Pi(v)(\bX_i)\|_2^2)\Big| \\
    &\quad  \leqslant  \frac{2}{n}(\| \tilde \Pi (u-v)(\bX_i)\|_2 + \mathbb{E}\|\tilde \Pi (u-v)(\bX_i)\|_2)\\
    &\quad \leqslant \frac{4 C_\Omega}{n}\sqrt{d_2} \|u-v\|_{H^{m+1}(\Omega)} \qquad\qquad\qquad \text{(by applying Theorem \ref{thm:sobIneq}).}
\end{align*}
Therefore, applying Hoeffding's, Azuma's and Dudley's theorem similarly as in the proof of Theorem \ref{thm:approx_integral} shows that
\[\mathbb{E}(Z_n) \leqslant 24 C_\Omega d_2^{1/2} n^{-1}  \int_0^\infty [\log N(B( 1, \|\cdot\|_{H^{m+1}(\Omega)}), \|\cdot\|_{H^{m+1}(\Omega)}, r)]^{1/2}dr.\]
Lemma \ref{lem:entropy} shows that there exists a constant $C_\Omega'$, depending only on $\Omega$, such that 
$\mathbb{E}(Z_n) \leqslant C'_\Omega d_2^{1/2} n^{-1/2}$. 
Applying McDiarmid's inequality  as in the proof of Theorem \ref{thm:approx_integral} shows that
$\mathrm{Var}(Z_n) \leqslant 16 C_\Omega^2 d_2 n^{-1}$.
Finally, since $\mathbb{E}(Z_n^2) \leqslant \mathrm{Var}(Z_n) + \mathbb{E}(Z_n)^2$, we deduce that
\[\mathbb{E}(Z_n^2)\leqslant \frac{d_2}{n}\big((C'_\Omega)^2+16 C_\Omega^2\big).\]
\end{proof}

\begin{lem}[Empirical process]
    \label{lem:empiricalProcess}
    Let $\bX_1, \hdots, \bX_n, \varepsilon_1, \hdots, \varepsilon_n$ be independent random variables, such that $\bX_i$ is distributed along $\mu_\bX$ and $\varepsilon_i$ is distributed along $\mu_\varepsilon$, such that $\mathbb{E}(\varepsilon)=0$. Then there exists a constant $C_\Omega >0$, depending only on $\Omega$, such that
    \[\mathbb{E}\Big(\Big(\sup_{\|u\|_{H^{m+1}(\Omega)}\leqslant 1} \frac{1}{n} \sum_{j=1}^n \langle \tilde \Pi (u)(\bX_j)- \mathbb{E}(\tilde \Pi (u)(\bX)),\varepsilon_j\rangle \Big)^2\Big) \leqslant \frac{d_2 \mathbb{E}\|\varepsilon\|_2^2}{n} C_\Omega,\]
    where $\tilde \Pi$ is the Sobolev embedding.
\end{lem}

\begin{proof}
First note, since $H^{m+1}(\Omega, \mathbb{R}^{d_2})$ is separable and since, for all $u \in H^{m+1}(\Omega, \mathbb{R}^{d_2})$, the function $(\bx_1, \hdots, \bx_n, e_1, \hdots, e_n) \mapsto \frac{1}{n} \sum_{j=1}^n \langle \tilde \Pi (u)(\bx_j)- \mathbb{E}(\tilde \Pi (u)(\bX)),e_j\rangle$ is continuous,  that the quantity $Z = \sup_{\|u\|_{H^{m+1}(\Omega)}\leqslant 1} \frac{1}{n} \sum_{j=1}^n \langle \tilde \Pi (u)(\bX_j) - \mathbb{E}(\tilde \Pi (u)(\bX)),\varepsilon_j\rangle$ is a random variable.
Moreover, $|Z|\leqslant 2 C_\Omega \sqrt{d_2} \sum_{j=1}^n \|\varepsilon_j\|_2/n$, where $C_\Omega$ is the constant of Theorem \ref{thm:sobIneq}. 
Thus, $\mathbb{E}(Z^2) < \infty$.

Define, for any $u \in H^{m+1}(\Omega, \mathbb{R}^{d_2})$, 
\[Z_{n,u} =  \frac{1}{n} \sum_{j=1}^n \langle \tilde \Pi (u)(\bX_j) - \mathbb{E}(\tilde \Pi (u)(\bX)),\varepsilon_j\rangle \quad \text{and} \quad Z_n = \sup_{\|u\|_{H^{m+1}(\Omega)}\leqslant 1} Z_{n,u}.\] 
For any $u,v \in  H^{m+1}(\Omega, \mathbb{R}^{d_2})$, we have
\begin{align*}
    &\Big|\frac{1}{n} \langle \tilde \Pi (u)(\bX_i) - \mathbb{E}(\tilde \Pi (u)(\bX)),\varepsilon_i\rangle  - \frac{1}{n} \langle \tilde \Pi (v)(\bX_i) - \mathbb{E}(\tilde \Pi (u)(\bX)),\varepsilon_i\rangle \Big| \\
    &\quad  =  \frac{1}{n}|\langle \tilde \Pi (u-v)(\bX_i) - \mathbb{E}(\tilde \Pi (u-v)(\bX)),\varepsilon_i\rangle|\\
    &\quad \leqslant \frac{2 C_\Omega}{n}\sqrt{d_2} \|u-v\|_{H^{m+1}(\Omega)}\|\varepsilon_i\|_2 \qquad\qquad\qquad \text{(by applying Theorem \ref{thm:sobIneq}).}
\end{align*}
Using that $\varepsilon$ is independent of $\bX$, so that the conditional expectation of $Z_n$ is indeed a real expectation with $\varepsilon_1, \hdots, \varepsilon_n$ fixed, we can apply Hoeffding's, Azuma's and Dudley's theorem similarly as in the proof of Theorem \ref{thm:approx_integral} to show that
\begin{align*}
    \mathbb{E}(Z_n\mid \varepsilon_1, \hdots, \varepsilon_n) &\leqslant \frac{24 C_\Omega}{n} \sqrt{d_2}\Big(\sum_{i=1}^n \|\varepsilon_i\|_2^2\Big)^{1/2} \\
    &\quad \times \int_0^\infty[\log N(B( 1, \|\cdot\|_{H^{m+1}(\Omega)}), \|\cdot\|_{H^{m+1}(\Omega)}, r)]^{1/2}dr.
\end{align*}
Hence, according to Lemma \ref{lem:entropy}, there exists a constant $C_\Omega'>0$, depending only on $\Omega$, such that $\mathbb{E}(Z_n \mid \varepsilon_1, \hdots, \varepsilon_n) \leqslant C'_\Omega n^{-1} \sqrt{d_2}\Big(\sum_{i=1}^n \|\varepsilon_i\|_2^2\Big)^{1/2}$.
We deduce that
\begin{align*}
    \mathbb{E}(Z_n) \leqslant C'_\Omega  \sqrt{d_2} \frac{(\mathbb E \|\varepsilon\|_2^2)^{1/2}}{n^{1/2}},
\end{align*}
and 
\begin{align*}
    \mathrm{Var}(\mathbb{E}(Z_n\mid \varepsilon_1, \hdots, \varepsilon_n)) \leqslant \mathbb E(\mathbb{E}(Z_n\mid \varepsilon_1, \hdots, \varepsilon_n)^2)  \leqslant (C'_\Omega)^2  d_2 \frac{\mathbb E \|\varepsilon\|_2^2}{n}.
\end{align*}
Applying McDiarmid's inequality as in the proof of Theorem \ref{thm:approx_integral} shows that
\[\mathrm{Var}(Z_n\mid \varepsilon_1, \hdots, \varepsilon_n) \leqslant 16 C_\Omega^2 d_2 \frac{1}{n^2}\sum_{i=1}^n \|\varepsilon_i\|_2^2. \]
The law of the total variance ensures that 
\begin{align*}
    \mathrm{Var}(Z_n) &= \mathrm{Var}(\mathbb{E}(Z_n\mid \varepsilon_1, \hdots, \varepsilon_n)) + \mathbb{E}(\mathrm{Var}(Z_n\mid \varepsilon_1, \hdots, \varepsilon_n))\\
    & \leqslant \frac{d_2 \mathbb E \|\varepsilon\|_2^2}{n}\big((C'_\Omega)^2+16 C_\Omega^2\big).
\end{align*}
Since $\mathbb{E}(Z_n^2) \leqslant \mathrm{Var}(Z_n) + \mathbb{E}(Z_n)^2$, we deduce that
\[\mathbb{E}(Z_n^2)\leqslant \frac{d_2 \mathbb E \|\varepsilon\|_2^2}{n}\big(2(C'_\Omega)^2+16 C_\Omega^2\big).\]
\end{proof}

\section{Proofs of Proposition \ref{prop:densite}} 
\label{proof:density}
\citet[][Theorem 5.1]{ryck2021approximation} ensures that $\text{NN}_2$ is dense in $(C^\infty([0,1]^{d_1}, \mathbb{R}), \|\cdot\|_{C^K([0,1]^{d_1})})$ for all $d_1 \geqslant 1$ and $K\in \mathbb{N}$. Note that the authors state the result for Hölder spaces $(W^{K+1,\infty}([0,1]^{d_1}), \|\cdot\|_{W^{K,\infty}(]0,1[^{d_1})})$ \citep[see][for a definition]{evans2010partial}. Clearly, $C^\infty([0,1]^{d_1}) \subseteq W^{K+1, \infty}([0,1]^{d_1})$ and the norms $\|\cdot\|_{C^{K}}$ and $\|\cdot\|_{W^{K,\infty}}$ coincide on $C^\infty([0,1]^{d_1})$.

Our proof generalizes this result to any bounded Lipschitz domain $\Omega$, to any number $H\geqslant 2$ of layers, and to any output dimension $d_2$. We stress that for any $U \subseteq \mathbb{R}^{d_1}$, the set $\mathrm{NN}_2 \subseteq C^\infty(\mathbb{R}^{d_1}, \mathbb{R}^{d_2})$ can of course be seen as a subset of $C^\infty(U, \mathbb{R}^{d_2})$.

\paragraph{Generalization to any bounded Lipschitz domain $\Omega$} 
In this and the next paragraph, $d_2=1$.
Our objective is to prove that $\text{NN}_2$ is dense in $(C^\infty(\bar{\Omega}, \mathbb{R}), \|\cdot\|_{C^K(\Omega)})$. 
Let $f \in C^\infty(\bar{\Omega}, \mathbb{R})$. Since $\Omega$ is bounded, there exists an affine transformation $\tau : x \mapsto A_\tau x +b_\tau$, with $A_\tau \in \mathbb{R}^\star$ and $b_\tau \in \mathbb{R}^{d_1}$, 
such that $\tau(\Omega) \subseteq [0,1]^d$. Set  $\hat{f} = f (\tau^{-1})$. 
According to the extension theorem for Lipschitz domains of \citet[][Theorem 5 Chapter VI.3.3]{stein1970lipschitz}, the function $\hat{f}$ can be extended to a function $\tilde{f} \in W^{K,\infty}([0,1]^{d_1})$ such that $\tilde{f}|_{\tau(\Omega)} = \hat{f}|_{\tau(\Omega)}$. 
Fix $\epsilon > 0$. According to \citet[][Theorem 5.1]{ryck2021approximation}, there exists $u_\theta \in \text{NN}_2$ such that 
$\|u_\theta - \hat{f}\|_{W^{K, \infty}([0,1]^d)} \leqslant \epsilon$. Since $\tilde{f}$ is an extension of 
$\hat{f}$,  $\tilde{f}|_{\tau(\Omega)} \in C^\infty(\bar{\Omega})$ and one also has $\|u_\theta - \hat{f}\|_{C^K(\tau(\Omega))} \leqslant \epsilon$. 

Now, let $m\in \mathbb{N}$ and let $\alpha$ be a multi-index such that $\sum_{i=1}^{d_1}\alpha_i = m$. Then, clearly,
$\partial^\alpha (\hat{f}(\tau))= A_\tau^m \times \partial^\alpha \hat{f}(\tau)$. 
Therefore,
$\|u_\theta(\tau) - \hat{f}(\tau)\|_{C^K(\Omega)} \leqslant \epsilon\times \max(1,A_\tau^K)$,
that is
\[\|u_\theta(\tau) - f\|_{C^K(\Omega)} \leqslant \epsilon\times \max(1,A_\tau^K).\] 
But, since $\tau$ is affine, $u_\theta(\tau)$ belongs to $\text{NN}_2$. This is the desires result.
\paragraph{Generalization to any number $H\geqslant 2$ of layers} We show in this paragraph that $\text{NN}_H$ is dense in $(C^\infty(\bar{\Omega}, \mathbb{R}), \|\cdot\|_{C^K(\Omega)})$ for all $H \geqslant 2$. The case $H=2$ has been treated above and it is therefore assumed that $H \geqslant 3$.

Let $f \in C^\infty(\bar{\Omega}, \mathbb{R})$. Introduce the function $v$ defined by
\[v(x_1, \hdots, x_{d_1}) = (\tanh^{\circ (H-2)}(x_1), \hdots, \tanh^{\circ (H-2)}(x_{d_1})),\] 
where $\tanh^{\circ (H-2)}$ stands for the $\tanh$ function composed $(H-2)$ times with itself.
For all $u_\theta \in \text{NN}_2$, $u_\theta(v) \in \text{NN}_H$ is a neural network such that the first weights matrices $(W_\ell)_{1 \leqslant \ell \leqslant H-2}$ are identity matrices and the first offsets $(b_\ell)_{1 \leqslant \ell \leqslant H-2}$ are equal to zero. 
Since $\tanh$ is an increasing $C^\infty$ function, $v$ is a $C^\infty$ diffeomorphism.
Therefore, $v(\Omega)$ is a bounded Lipschitz domain and $f(v^{-1}) \in C^\infty(v(\Omega), \mathbb{R})$. Lemma \ref{lem:boundPartialDer2} shows that $f(v^{-1}) \in C^\infty(\bar{v}(\Omega), \mathbb{R})$, where $\bar{v}(\Omega)$ is the closure of $v(\Omega)$. According to the previous paragraph, there exists a sequence $(\theta_m)_{m \in \mathbb{N}}$ of parameters such that $u_{\theta_m} \in \text{NN}_2$ and 
\begin{equation*}
    \lim_{m \to \infty} \|u_{\theta_m} - f(v^{-1})\|_{C^K(v(\Omega))} = 0.
\end{equation*}
Thus, $u_{\theta_m}$ approximates $f(v^{-1})$, and we would like $u_{\theta_m} (v)$ to approximate $f$.
From Lemma \ref{lem:boundPartialDer2}, \[ \|u_{\theta_m}(v)-f\|_{C^K(\Omega)} \leqslant B_K \times  \|u_{\theta_m}-f\circ v^{-1}\|_{C^K(\Omega)} \times (1+\|\tanh^{\circ H-2}\|_{C^K(\mathbb{R})})^K,\]
while Corollary \ref{cor:bounding_tanh} asserts that $\|\tanh^{\circ H-2}\|_{C^K(\mathbb{R})} < \infty$.

Therefore, we deduce that $\lim_{m \to \infty }\|u_{\theta_m}(v)-f\|_{C^K(\Omega)}  = 0$ with $u_{\theta_m}(v) \in \text{NN}_H$, which proves the lemma for  $H \geqslant 2$.

\paragraph{Generalization to all output dimension $d_2$} We have shown so far that for all $H \geqslant 2$, $\text{NN}_H$ is dense in $(C^\infty(\bar{\Omega}, \mathbb{R}), \|\cdot\|_{C^K(\Omega)})$. It remains to establish that $\text{NN}_H$ is dense in $(C^\infty(\bar{\Omega}, \mathbb{R}^{d_2}), \|\cdot\|_{C^K(\Omega)})$ for any output dimension $d_2$.

Let $f = (f_1, \hdots, f_{d_2}) \in C^\infty(\Omega, \mathbb{R}^{d_2})$. For all $1\leqslant i \leqslant d_2$, let $(\theta_m^{(i)})_{m \in \mathbb{N}} \in (\text{NN}_H)^{\mathbb{N}}$ be a sequence of neural networks such that $\lim_{m \to \infty}\|u_{\theta^{(i)}_m}-f_i\|_{C^K(\Omega)} = 0$. Denote by $u_{\theta_m} = (u_{\theta_m^{(1)}}, \hdots, u_{\theta_m^{(d_2)}})$ the stacking of these sequences. For all $m\in \mathbb{N}$, $u_{\theta_m} \in \text{NN}_H$ and $\lim_{m \to \infty }\|u_{\theta_m}-f\|_{C^K(\Omega)} = 0$. Therefore, $\text{NN}_H$ is dense in $(C^\infty(\bar{\Omega}, \mathbb{R}), \|\cdot\|_{C^K(\Omega)})$.

\section{Proofs of Section \ref{POF}}

\subsection{Proof of Proposition \ref{prop:friction}}
\label{proof:hybrid_modeling_failure}
Consider $u_{\hat{\theta}(p, n_r,D)} \in \text{NN}_H(D)$, the neural network defined by
    \[u_{\hat{\theta}(p, n_r,D)}(\bx) =  Y_{(1)} + \sum_{i=1}^{n-1} \frac{Y_{(i+1)}-Y_{(i)}}{2} \bigg[\tanh_p^{\circ H}\Big(\bx-\bX_{(i)}-\frac{\delta(n,n_r)}{2}\Big)+1\bigg],\]
where $\delta(n,n_r)$ is defined in \eqref{def:delta} and where the observations have been reordered according to increasing values of the $\bX_{(i)}$. 
According to Lemma \ref{lem:estimationTanh}, one has, for all $1 \leqslant i \leqslant n$, $\lim_{p \to \infty} u_{\hat{\theta}(p, n_r,D)}(\bX_i) = Y_i$. Moreover, for all order $K\geqslant 1$ of differentiation and all $1 \leqslant j \leqslant n_r$, $\lim_{p\to\infty}u^{(K)}_{\hat{\theta}(p, n_r,D)}(\bX^{(r)}_j) = 0$. Recalling that $\mathscr{F}(u, \bx) = mu''(\bx) + \gamma u'(\bx)$, we have
$\|\mathscr{F}(u, \bx)\|_2\leqslant m \|u''(\bx)\|_2 + \gamma \|u'(\bx)\|_2$. We therefore conclude that $\lim_{p\to\infty} R_{n, n_r}(u_{\hat{\theta}(p, n_r,D)}) = 0$, which is the first statement of the proposition.

Next, using the Cauchy-Schwarz inequality, we have that, for any function $f\in C^2(\mathbb{R})$ and any $\varepsilon >0$,
\[
2\varepsilon\int_{-\varepsilon}^\varepsilon (m f''+ \gamma f')^2  \geqslant \Big(\int_{-\varepsilon}^\varepsilon mf'' + \gamma f'\Big)^2= \big[m (f'(\varepsilon)-f'(-\varepsilon)) + \gamma (f(\varepsilon)-f(-\varepsilon))\big]^2. 
\]
Thus,
\begin{align*}
&\mathscr{R}_n(u_{\hat{\theta}(p, n_r,D)}) \\
&\quad \geqslant \frac{1}{T} \int_{[0,T]} \mathscr{F}(u_{\hat{\theta}(p, n_r,D)}, \bx)^2 d\bx\\
 &\quad \geqslant \frac{1}{T} \sum_{i=1}^{n}\int_{\bX_{(i)}+\delta(n, n_r)/2-\varepsilon}^{\bX_{(i)}+\delta(n, n_r)/2+\varepsilon } \mathscr{F}(u_{\hat{\theta}(p, n_r,D)}, \bx)^2 d\bx\\
 &\quad \geqslant  \frac{1}{T}\sum_{i=1}^{n} \frac{1}{2 \varepsilon}\big[m(u'_{\hat{\theta}(p, n_r,D)}(\bX_{(i)}+\delta(n, n_r)/2+\varepsilon)-u'_{\hat{\theta}(p, n_r,D)}(\bX_{(i)}+\delta(n, n_r)/2-\varepsilon))\\
 &\qquad \quad +\gamma(u_{\hat{\theta}(p, n_r,D)}(\bX_{(i)}+\delta(n, n_r)/2+\varepsilon)-u_{\hat{\theta}(p, n_r,D)}(\bX_{(i)}+\delta(n, n_r)/2-\varepsilon))\big]^2.
\end{align*}
Observe that, as soon as $\delta(n,n_r)/4 > \varepsilon$, one has, for all $1 \leqslant i\leqslant n-1$,
\[\lim_{p \to \infty}u_{\hat{\theta}(p, n_r,D)}(\bX_{(i)}+\delta(n, n_r)/2+\varepsilon ) - u_{\hat{\theta}(p, n_r,D)}(\bX_{(i)}+\delta(n, n_r)/2 -  \varepsilon ) = Y_{(i+1)}-Y_{(i)},\] and, for all $1 \leqslant i\leqslant n-1$,
\[\lim_{p\to\infty}u'_{\hat{\theta}(p, n_r,D)}(\bX_{(i)}+\delta(n, n_r)/2+\varepsilon ) - u'_{\hat{\theta}(p, n_r,D)}(\bX_{(i)}+\delta(n, n_r)/2 -  \varepsilon ) = 0.\] 
Hence, for any $ 0<\varepsilon<\delta(n,n_r)/4$,
\begin{align*}
    &\sum_{i=1}^{n} \frac{1}{2 \varepsilon}\big[m(u'_{\hat{\theta}(p, n_r,D)}(\bX_{(i)}+\delta(n, n_r)/2-\varepsilon)-u'_{\hat{\theta}(p, n_r,D)}(\bX_{(i)}+\delta(n, n_r)/2-\varepsilon))\\
    &\qquad +\gamma(u_{\hat{\theta}(p, n_r,D)}(\bX_{(i)}+\delta(n, n_r)/2-\varepsilon)-u_{\hat{\theta}(p, n_r,D)}(\bX_{(i)}+\delta(n, n_r)/2-\varepsilon))\big]
    ^2 \\
&\xrightarrow[p \rightarrow \infty]{} \gamma \times \frac{\sum_{i=1}^{n-1} (Y_{(i+1)}-Y_{(i)})^2}{2 \varepsilon}.
\end{align*}
We have just proved that, for any $ 0<\varepsilon<\delta(n,n_r)/4$, there exists $P \in \mathbb{N}$ such that, for all  $p \geqslant P$, 
\[\mathcal{R}_n(u_{\hat{\theta}(p, n_r,D)}) \geqslant  \gamma \times \frac{\sum_{i=1}^{n-1} (Y_{(i+1)}-Y_{(i)})^2}{2 \varepsilon T}.
\]
We conclude as desired that $\lim_{p\to\infty}\mathcal{R}_n(u_{\hat{\theta}(p, n_r,D)})  =\infty$, since we suppose that there exists two observations $Y_{(i)} \neq Y_{(j)}$.

\subsection{Proof of Proposition \ref{prop:wave}}
\label{proof:PDE_failure}
Let $u_{\hat{\theta}(p,n_e, n_r,D)} \in \text{NN}_H(4)$ be the neural network defined by
\begin{align*}
    u_{\hat{\theta}(p,n_e, n_r,D)}(x, t) =& \tanh^{\circ H}(x+0.5+p t)-\tanh^{\circ H}(x-0.5+pt) \\
    &+\tanh^{\circ H}(0.5+pt) - \tanh^{\circ H}(1.5+pt).
\end{align*}
Clearly, for any $p\in \mathbb{N}$, $u_{\hat{\theta}(p,n_e, n_r,D)}$ satisfies the initial condition
\[u_{\hat{\theta}(p,n_e, n_r,D)}(x, 0) = \tanh^{\circ H}(x+0.5)-\tanh^{\circ H}(x-0.5) + \tanh^{\circ H}(0.5) - \tanh^{\circ H}(1.5).\]
We are going to prove in the next paragraphs that the derivatives of $u_{\hat{\theta}(p,n_e, n_r,D)}$ vanish as $p\to \infty$, starting with the temporal derivative and continuing with the spatial ones.
According to Lemma \ref{lem:compTanh}, for all $\varepsilon > 0$ and all $x \in [-1, 1]$, 
$\lim_{p\to\infty}\|u_{\hat{\theta}(p,n_e, n_r,D)}(x, \cdot)\|_{C^2([\varepsilon, T])} = 0$.
Therefore, for any $\bX^{(e)}_i\in \{-1, 1\}\times [0,T]$,
$\lim_{p\to\infty}\|u_{\hat{\theta}(p,n_e, n_r,D)}(\bX^{(e)}_i)\|_2 = 0$ and, 
for any  $\bX^{(r)}_j\in \Omega$,
$\lim_{p\to\infty}\|\partial_{t}u_{\hat{\theta}(p,n_e, n_r,D)}(\bX^{(r)}_j)\|_2 = 0$ (since $\bX^{(r)}_j \notin \partial \Omega$).

Letting $v(x,t) = \tanh^{\circ H}(x+0.5+p t)-\tanh^{\circ H}(x-0.5+pt)$,
it comes that $\partial^2_{x,x} u_{\hat{\theta}(p,n_e, n_r,D)} = p^{-2}\partial^2_{t,t} v$.
Thus, invoking again Lemma \ref{lem:compTanh}, for all $\varepsilon > 0$, and all $x \in [-1, 1]$,
\[\lim_{p\to\infty}p^{-2}\|\partial^2_{t,t} v(x, \cdot)\|_{\infty, [\varepsilon, T]} =  \lim_{p\to\infty}\|\partial^2_{x,x} u_{\hat{\theta}(p,n_e, n_r,D)}(x, \cdot)\|_{\infty, [\varepsilon, T]}  = 0.\]
Therefore, for any $\bX^{(r)}_j\in \Omega$, one has $\lim_{p\to\infty}\|\partial^2_{x,x}u_{\hat{\theta}(p,n_e, n_r,D)}(\bX^{(r)}_j)\|_2 = 0$ and, in turn, one has
$\lim_{p\to\infty}\|\mathscr{F}(u_{\hat{\theta}(p,n_e, n_r,D)}, \bX^{(r)}_j)\|_2 = 0$. Thus, for all $n_e, n_r \geqslant 0$,
$\lim_{p\to\infty}R_{n_e,n_r}(u_{\hat{\theta}(p,n_e, n_r,D)}) = 0$.

Next, observe that $\mathscr{R}(u_{\hat{\theta}(p,n_e, n_r,D)}) \geqslant \int_{[-1,1]\times[0,T]} (\partial_{t}u_{\hat{\theta}(p,n_e, n_r,D)}-\partial^2_{x,x}u_{\hat{\theta}(p,n_e, n_r,D)})^2$. 
By the Cauchy-Schwarz inequality, for any $\delta >0$, 
\begin{align*}
   & \int_{[-1,1]\times[0,T]} (\partial_{t}u_{\hat{\theta}(p,n_e, n_r,D)}-\partial^2_{x,x}u_{\hat{\theta}(p,n_e, n_r,D)})^2\\
   & \geqslant \delta^{-1}\int_{x =-1}^{1}\Big(\int_{t=0}^\delta \partial_{t}u_{\hat{\theta}(p,n_e, n_r,D)}(x,t)-\partial^2_{x,x}u_{\hat{\theta}(p,n_e, n_r,D)}(x,t)\Big)^2dx\\
    & \geqslant \delta^{-1}\int_{x =-1}^{1}\Big(u_{\hat{\theta}(p,n_e, n_r,D)}(x,\delta) - u_{\hat{\theta}(p,n_e, n_r,D)}(x,0) -  \int_{t=0}^\delta \partial^2_{x, x}u_{\hat{\theta}(p,n_e, n_r,D)}(x,t)dt\Big)^2dx.
\end{align*}
Invoking again Lemma \ref{lem:compTanh}, we know that $\lim_{p\to \infty}\|u_{\hat{\theta}(p,n_e, n_r,D)}(\cdot,\delta)\|_{[-1, 1], \infty} = 0$. Moreover, for all $t >0$ and all $-1\leqslant x \leqslant 1$, $\lim_{p\to\infty}\partial^2_{x,x} u_{\hat{\theta}(p,n_e, n_r,D)}(x, t)  = 0$. Besides, by Corollary 
\ref{cor:bounding_tanh}, $\|\partial^2_{x,x} u_{\hat{\theta}(p,n_e, n_r,D)}\|_{\infty, [0,1]\times [-1,1]} \leqslant 2\|\tanh^{\circ H}\|_{C^2(\mathbb{R})} < \infty$.
Thus, by the dominated convergence theorem, for any $\delta > 0$ and all $p$ large enough, 
\[\mathscr{R}(u_{\hat{\theta}(p,n_e, n_r,D)}) \geqslant  \frac{1}{2\delta}\int_{x =-1}^{1}\big(u_{\hat{\theta}(p,n_e, n_r,D)}(x,0)\big)^2dx.\]
Noticing that $u_{\hat{\theta}(p,n_e, n_r,D)}(x,0)$ corresponds to the initial condition, that does not depends on $p$, we conclude that $\lim_{p\to\infty}\mathscr{R}(u_{\hat{\theta}(p,n_e, n_r,D)}) = \infty$.

\section{Proofs of Section \ref{sec:consistency}}
\label{app:sec5}
\subsection{Proof of Proposition \ref{prop:bounding}}
Recall that each neural network $u_\theta \in \mathrm{NN}_H(D)$ is written as $u_\theta = \mathcal{A}_{H+1} \circ (\tanh \circ \mathcal{A}_{H}) \circ \cdots \circ  (\tanh \circ \mathcal{A}_1)$,
where each $\mathcal{A}_k : \mathbb{R}^{L_{k-1}} \rightarrow \mathbb{R}^{L_{k}}$ is an affine function of the form $\mathcal{A}_k(x) = W_k x + b_k$, with $W_k$ a ($L_{k-1} \times L_k$)-matrix, $b_k \in \mathbb{R}^{L_k}$ a vector,  $L_0 = d_1$, $L_1=\cdots=L_H=D$, $L_{H+1} = d_2$, and $\theta = (W_1, b_1, \hdots, W_{H+1}, b_{H+1})\in \mathbb{R}^{\sum_{i=0}^H (L_i+1)\times L_i}$. For each $i\in\{1, \hdots, d_1\}$, we let $\pi_i$ be the projection operator on the $i$th coordinate, defined by $\pi_i(x_1, \hdots, x_{d_1}) = x_i$. Similarly, for a matrix $W = (W_{i,j})_{1\leqslant i \leqslant d_2, 1\leqslant j \leqslant d_1 }$, we let $\pi_{i,j}(W) = W_{i,j}$ and $\|W\|_\infty = \max_{1\leqslant i \leqslant d_2, 1\leqslant j \leqslant d_1 } |W_{i,j}|$. Note that $\|W_k\bx\|_\infty \leqslant L_{k-1} \|W_k\|_\infty \|\bx\|_\infty$.  Clearly,
$\max_{1\leqslant k \leqslant H+1}(\|W_k\|_\infty , \|b_k\|_\infty) \leqslant\|\theta\|_\infty \leqslant \|\theta\|_2$.
Finally, we recursively define  the constants $C_{K,H}$ for all $K\geqslant 0$ and all $H\geqslant 1$ by $C_{0,H} = 1$, $C_{K,1}= 2^{K-1}\times(K+2)!$, and 
\begin{equation}
    C_{K,H+1} = B_K 2^{K-1} (K+2)! \max_{\underset{i_1+2i_2+\cdots + Ki_K = K}{ i_1,\hdots, i_K \in \mathbb{N}}} \prod_{1\leqslant \ell \leqslant K}C_{\ell,H}, \label{eq:defC}
\end{equation}
where $B_K$ is the $K$th Bell number, defined in \eqref{eq:bell}.

We prove the proposition by induction  on $H$, starting with the case $H=1$. 
Clearly, for $H=1$, one has
\begin{equation}
    \label{eq:iniBounding}
    \|u_\theta\|_\infty \leqslant \|W_2 \times \tanh \circ \mathcal{A}_1\|_\infty + \|b_2\|_\infty \leqslant \|W_2\|_\infty D  + \|b_2\|_\infty \leqslant (D+1)  \|\theta\|_2.
\end{equation}
Next, for any multi-index $\alpha = (\alpha_1, \hdots, \alpha_{d_1})$ such that $|\alpha|\geqslant 1$,
\begin{equation}
    \label{eq:boundingNNIni}
    \partial^\alpha u_\theta(\bx) = W_2 \begin{pmatrix}
    \pi_{1,1}(W_1)^{\alpha_1} \times \cdots \times \pi_{1,d_1}(W_1)^{\alpha_{d_1}} \times \tanh^{(|\alpha|)}(\pi_1(\mathcal{A}_1(\bx)))\\
    \vdots\\
    \pi_{1,d_1}(W_1)^{\alpha_1} \times \cdots \times \pi_{d_1,d_1}(W_1)^{\alpha_{d_1}} \times \tanh^{(|\alpha|)}(\pi_{d_1}(\mathcal{A}_1(\bx)))
\end{pmatrix}.
\end{equation}
Upon noting that $|\pi_{1,d_1}(W_1)|\leqslant \|\theta\|_\infty$, we see that 
\begin{equation}
\label{eq:iniBoundingK}
    \|\partial^\alpha u_\theta\|_\infty \leqslant D  \|W_2\|_\infty  \|\theta\|_2^{|\alpha|}  \|\tanh^{(|\alpha|)}\|_\infty \leqslant D  \|\theta\|_2^{1+|\alpha|}  \|\tanh^{(|\alpha|)}\|_\infty.
\end{equation}
Therefore, combining \eqref{eq:iniBounding} and \eqref{eq:iniBoundingK},  for any $K \geqslant 1$, $\|u_\theta\|_{C^K(\mathbb{R}^{d_1})} \leqslant (D+1) \max_{k \leq K}\|\tanh^{(k)}\|_\infty  (1+\|\theta\|_2)^{K} \|\theta\|_2$.
Applying Lemma \ref{lem:derTanh}, we conclude that, for all $u \in \mathrm{NN}_1(D)$ and for all $K \geqslant 0$,
\[\|u_\theta\|_{C^K(\mathbb{R}^{d_1})} \leqslant  C_{K,1} (D+1)  (1+\|\theta\|_2)^{K} \|\theta\|_2.\]
\paragraph{Induction} Assume that for a given $H \geqslant 1$, one has, for any neural network $u_\theta \in \mathrm{NN}_H(D)$ and any $K \geqslant 0$,
\begin{equation}
\label{eq:recBoundingNN}
    \|u_\theta\|_{C^K(\mathbb{R}^{d_1})} \leqslant C_{K,H} (D+1)^{1+KH}  (1+\|\theta\|_2)^{KH}\|\theta\|_2.
\end{equation}
Our objective is to show that for any $u_\theta \in \mathrm{NN}_{H+1}(D)$ and any $K \geqslant 0$,
\[\|u_\theta\|_{C^K(\mathbb{R}^{d_1})} \leqslant C_{K,H+1} (D+1)^{1+K(H+1)}  (1+\|\theta\|_2)^{K(H+1)} \|\theta\|_2.\] 
For such a $u_\theta$, we have, by definition, 
$u_\theta = \mathcal{A}_{H+2}\circ \tanh \circ v_\theta$, where $v_\theta \in \mathrm{NN}_{H}(D)$ (by a slight abuse of notation, the parameter of $v_\theta$ is in fact $\theta' = (W_1, b_1, \hdots, W_{H+1}, b_{H+1})$ while $\theta = (W_1, b_1, \hdots,$ $W_{H+2}, b_{H+2})$, so $\|\theta'\|_2 \leqslant \|\theta\|_2$ and $\|\theta'\|_\infty \leqslant \|\theta\|_\infty$).
Consequently, 
\begin{equation}
    \|u_\theta\|_\infty \leqslant \|W_{H+2}\|_\infty  D + \|b_{H+2}\|_\infty \leqslant (D+1) \|\theta\|_2.\label{eq:boundNN}
\end{equation}
In addition, for any multi-index $\alpha = (\alpha_1, \hdots, \alpha_{d_1})$ such that $|\alpha|\geqslant 1$,
\[
\partial^\alpha u_\theta(\bx) = W_{H+2} \begin{pmatrix}
    \partial^\alpha ( \tanh\circ \pi_1 \circ v_\theta(\bx))\\
    \vdots\\
    \partial^\alpha ( \tanh\circ \pi_{D}\circ v_\theta(\bx))
\end{pmatrix}.
\]
Thus,
$\|\partial^\alpha u_\theta\|_\infty \leqslant D \|W_{H+2}\|_\infty  \max_{j \leqslant D}\|\tanh \circ \pi_j \circ v_\theta\|_{C^K(\mathbb{R}^{d_1})}$. Invoking identity \eqref{eq:fdbFormula}, one has
\[\|\tanh \circ \pi_j\circ v\|_{C^K(\mathbb{R}^{d_1})} \leqslant B_K \|\tanh \|_{C^K(\mathbb{R})} \max_{i_1+2i_2+\cdots + Ki_{K}=K}\prod_{1\leqslant \ell \leqslant K}\|\pi_j \circ v_\theta\|_{C^\ell(\mathbb{R}^{d_1})}^{i_\ell}.\] 
Observing that $\pi_{j}\circ v_\theta $ belongs to $\mathrm{NN}_H(D)$, Lemma \ref{lem:derTanh} and inequality \eqref{eq:recBoundingNN} show that
\[\|\tanh \circ \pi_j\circ v_\theta\|_{C^\ell(\mathbb{R}^{d_1})} \leqslant C_{\ell,H+1} (D+1)^{1+\ell H}  (1+\|\theta\|_2)^{1+\ell H} \|\theta\|_2.\] 
Therefore, $\|\partial^\alpha u_\theta\|_\infty \leqslant C_{K,H+1} (D+1)^{1+KH}  (1+\|\theta\|_2)^{K(H+1)}  \|\theta\|_2$, which concludes the induction.

To complete the proof, it remains to show that the exponent of $\|\theta\|_2$ is optimal. To this aim, we let $d_1 = d_2 = 1$, $D=1$. For each $H \geqslant 1$, we consider the sequence $(\theta^{(H)}_m)_{m \in \mathbb{N}}$ defined by  
$\theta^{(H)}_m = (W_1^{(m)}, b_1^{(m)}, \hdots, W^{(m)}_{H+1}, b^{(m)}_{H+1})$, with $W_i^{m} = m$ and $b_i^m = 0$. Then, for all $\theta = (W_1, b_1, \hdots,$ $ W_{H+1}, b_{H+1}) \in \Theta_{H, 1}$, the associated neural network's derivatives satisfy
\[\|u_\theta^{(k)}\|_\infty =
\|(\tanh^{\circ H})^{(K)}\|_{\infty}  |W_{H+1}| \prod_{i=1}^{H} |W_i|^K. \]
Next, since 
$\|\theta^{(H)}_m\|_2 = m\sqrt{H+1}$, we have
\begin{equation*}
\|u_{\theta^{(H)}_m}\|_{C^K(\mathbb{R}^{d_1})} 
\geqslant \big\|u_{\theta^{(H)}_m}^{(K)}\big\|_\infty \geqslant \big\|(\tanh^{\circ H})^{(K)}\big\|_\infty   m^{1+HK}
\geqslant \bar{C}(H, K)   \|\theta^{(H)}_m\|_2^{1+HK},
\end{equation*} 
where $\bar{C}(H, K) = (H+1)^{-(1+HK)/2} \|(\tanh^{\circ H})^{(K)}\|_\infty $. Since $\lim_{m\to\infty}\|\theta^{(H)}_m\|_2 = \infty $, we conclude that the bound of inequality \eqref{eq:recBoundingNN} is tight.
\subsection{Lipschitz dependence of the Hölder norm in the NN parameters}

\begin{prop}[Lipschitz dependence of the Hölder norm in the NN parameters]
\label{prop:lipschitzParam}
Consider the class $\mathrm{NN}_H(D)=\{u_\theta, \theta\in\Theta_{H,D}\}$.
Let $K \in \mathbb{N}$. Then there exists a constant $\tilde{C}_{K,H}>0$, depending only on $K$ and $H$, such that, for all $\theta, \theta' \in \Theta_{H, D}$, 
\[
\|u_\theta-u_{\theta'}\|_{C^K(\Omega)} \leqslant \tilde{C}_{K,H}  (1+d_1M(\Omega)) (D+1)^{H+KH^2}  (1+\|\theta\|_2)^{H+KH^2}  \|\theta-\theta'\|_2,
\]
where $M(\Omega) = \sup_{\bx\in\Omega} \|\bx\|_\infty$.
\end{prop}

\begin{proof}
    We recursively define the constants $\tilde{C}_{K,H}$ for all $K\geqslant 0$ and all $H\geqslant 1$ by $\tilde{C}_{K,1} = (K+2) 2^{2K-1}  (K+2)! (K+3)!$, and \[\tilde{C}_{K,H+1} = C_{K,H+1} [1 + (K+1) B_K 2^{2K-1} (K+3)!(K+2)! \tilde{C}_{K,H}].\]
    Recall that $\pi_i$ is the projection operator on the $i$th coordinate, defined by $\pi_i(x_1, \hdots, x_{d_1}) = x_i$.
Before embarking on the proof, observe that by identity \eqref{eq:fdbFormula}, we have, for all $u_1, u_2 \in C^K(\Omega, \mathbb{R}^{D})$, 
for all $1 \leqslant i \leqslant D,$
\begin{align*}
    \partial^\alpha (\tanh\circ \pi_i \circ u_1-\tanh\circ \pi_i \circ u_2) &= \sum_{P\in \Pi(K)}[\tanh^{(|P|)}\circ \pi_i \circ u_1]  \prod_{S \in P} \partial^{\alpha(S)}(\pi_i \circ u_1) \\
    &\quad - [\tanh^{(|P|)}\circ \pi_i \circ u_2]  \prod_{S \in P} \partial^{\alpha(S)} (\pi_i \circ u_2).
\end{align*}
In addition, for two sequences $(a_i)_{1\leqslant i \leqslant n}$ and $(b_i)_{1\leqslant i \leqslant n}$,
\begin{equation}
    \prod_{i=1}^n a_i - \prod_{i=1}^n b_i = \sum_{i=1}^n (a_i-b_i) \Big(\prod_{j=i+1}^n a_j\Big)\Big(\prod_{j=1}^{i-1}b_j\Big) \leqslant n \max_{1\leqslant i\leqslant n}\{|a_i-b_i|\} \prod_{i=1}^n \max(|a_i|, |b_i|). \label{eq:diffProd}
\end{equation}  
 Observe that for any $1 \leqslant i\leqslant d_2$ and $P \in \Pi(K)$, the term  $[\tanh^{(|P|)}\circ \pi_i \circ u_1]  \prod_{S \in P} \partial^{\alpha(S)}(\pi_i \circ u_1) - [\tanh^{(|P|)}\circ \pi_i \circ u_2]  \prod_{S \in P} \partial^{\alpha(S)} (\pi_i \circ u_2)$ is the difference of two products of $|P| +1$ terms to which we can apply \eqref{eq:diffProd}. So,
\begin{align}
    &\Big\|[\tanh^{(|\pi|)}\circ \pi_i \circ u_1] \prod_{S \in P} \partial^{\alpha(S)}(\pi_i \circ u_1) - [\tanh^{(|\pi|)}\circ \pi_i \circ u_2]  \prod_{S \in \pi} \partial^{\alpha(S)}(\pi_i \circ u_2)\Big\|_{\infty, \Omega}\nonumber\\
    &\quad \leqslant (|P|+1)  \big(\|\tanh^{(|P|)}\|_{\mathrm{Lip}} \|u_1-u_2\|_{\infty, \Omega} + \|u_1-u_2\|_{C^K(\Omega)}\big) \nonumber\\
    &\qquad \times \|\tanh^{(|P|)}\|_{\infty} \prod_{S \in P} \max(\|\partial^{\alpha(S)}u_1\|_{\infty, \Omega},\|\partial^{\alpha(S)}u_2\|_{\infty, \Omega}).\label{eq:fundRes}
\end{align}
Notice finally that $\|\tanh^{(|P|)}\|_{\mathrm{Lip}} = \|\tanh^{(|P|+1)}\|_{\infty}$.

With the preliminary results out of the way, we are now equipped to prove the statement of the proposition, by induction on $H$. Assume first that $H=1$. We start by examining the case $K= 0$ and then generalize to all $K \geqslant 1$. Let $u_\theta = \mathcal{A}_2 \circ \tanh \circ \mathcal{A}_1$ and $u_{\theta'} = \mathcal{A}'_2 \circ \tanh \circ \mathcal{A}'_1$. Notice that \[\|\mathcal{A}_1-\mathcal{A}'_1\|_{\infty, \Omega} \leqslant \|b_1-b'_1\|_\infty + d_1 M(\Omega) \|W_1-W_1'\|_\infty \leqslant \|\theta-\theta'\|_2 (1+d_1M(\Omega)),\]
where $M(\Omega)= \max_{\bx \in \Omega}\|\bx\|_\infty$. 
Since $\|\tanh\|_{\mathrm{Lip}}=1$, we deduce that $\|\tanh \circ \mathcal{A}_1-\tanh \circ \mathcal{A}'_1\|_\infty \leqslant \|\theta-\theta'\|_2 (1+d_1M(\Omega))$. Similarly, $\|\mathcal{A}_2-\mathcal{A}'_2\|_{\infty,  B(1, \|\cdot\|_\infty)} \leqslant \|\theta-\theta'\|_2 (1+D)$. Next, 
\begin{align*}
    \|u_\theta-u_{\theta'}\|_{\infty, \Omega} &\leqslant \|(\mathcal{A}_2-\mathcal{A}_2')\circ \tanh \circ \mathcal{A}_1\|_{\infty, \Omega} + \|\mathcal{A}_2' \circ \tanh \circ \mathcal{A}_1 - \mathcal{A}_2'\circ \tanh\circ \mathcal{A}_1')\|_{\infty, \Omega}\\
    & \leqslant \|\mathcal{A}_2-\mathcal{A}_2'\|_{\infty, B(1, \|\cdot\|_\infty)} + D \|W_2'\|_\infty \|\tanh \circ \mathcal{A}_1 - \tanh\circ \mathcal{A}_1'\|_{\infty, \Omega}\\
    & \leqslant \|\theta-\theta'\|_2  (1+D+ D  \|\theta'\|_2  (1+d_1M(\Omega)))\\
    &\leqslant \tilde{C}_{0,1}  (1+d_1M(\Omega)) (D+1) (1+\max(\|\theta\|_2,\|\theta'\|_2)) \|\theta-\theta'\|_2.
\end{align*}
This shows the result for $H=1$ and $K=0$. Assume now that $K\geqslant 1$, and let $\alpha$ be a multi-index such that $|\alpha| = K$. Observe that
\begin{align}
    \|\partial^\alpha(u_\theta-u_{\theta'})\|_{\infty, \Omega} &\leqslant \|(W_2-W_2') \partial^\alpha( \tanh \circ \mathcal{A}_1)\|_{\infty, \Omega} \nonumber\\
    &\quad + \|W_2' \partial^\alpha (\tanh \circ \mathcal{A}_1 - \tanh\circ \mathcal{A}_1')\|_{\infty, \Omega}.\label{eq:trig}
\end{align}
By Lemma \ref{lem:derTanh} and an argument similar to the inequality \eqref{eq:boundingNNIni}, we have
\begin{align}
    \|(W_2-W_2') \partial^\alpha( \tanh \circ \mathcal{A}_1)\|_{\infty, \Omega} &\leqslant (D+1)  \|\theta-\theta'\|_2  \|\theta\|_2^K  \|\tanh\|_{C^K(\mathbb{R})}\nonumber\\
    &\leqslant  2^{K-1}(K+2)!  (D+1)  \|\theta-\theta'\|_2  \|\theta\|_2^K\label{eq:trigPart1}.
\end{align}
 In order to bound the second term on the right-hand side of \eqref{eq:trig}, we use inequality \eqref{eq:fundRes} with $u_1 =\mathcal{A}_1$ and $u_2 =\mathcal{A}_1'$. In this case,  the only non-zero term on the right-hand side of \eqref{eq:fundRes} corresponds to the partition $\pi  = \{\{1\}, \{2\}, \hdots, \{K\}\}$. Recall that $\|\mathcal{A}_1-\mathcal{A}_1'\|_{\infty, \Omega} \leqslant \|\theta-\theta'\|_2 (1+d_1M(\Omega))$, and note that whenever $|\alpha| = 1$, $\|\partial^\alpha(\mathcal{A}_1-\mathcal{A}_1')\|_{\infty, \Omega} \leqslant \|\theta-\theta'\|_2$.  
Therefore, $\|\mathcal{A}_1-\mathcal{A}_1'\|_{C^K(\Omega)} = \|\mathcal{A}_1-\mathcal{A}_1'\|_{C^1(\Omega)} \leqslant  \|\theta-\theta'\|_2  (1+d_1M(\Omega))$. Observe that $\prod_{B \in \{\{1\}, \{2\}, \hdots, \{K\}\}} \max(\|\partial^{\alpha(B)}\mathcal{A}_1\|_{\infty, \Omega},\|\partial^{\alpha(B)}\mathcal{A}_1'\|_{\infty, \Omega}) \leqslant \max(\|\theta\|_2, \|\theta'\|_2)^{K}.$ Thus, putting all the pieces together, we are led to \begin{align*}
    &\|\partial^\alpha (\tanh\circ \mathcal{A}_1-\tanh\circ \mathcal{A}_1')\|_{\infty, \Omega}\\
    &\quad \leqslant (K+1) \|\tanh^{(K+1)}\|_{\infty} \|\theta-\theta'\|_2 (1+d_1 M(\Omega))  \|\tanh^{(K)}\|_{\infty}    \max(\|\theta\|_2, \|\theta'\|_2)^{K}.
\end{align*}
Now, by Lemma \ref{lem:derTanh}, $\|\tanh^{(K)}\|_{\infty} \leqslant 2^{K-1}(K+2)!$ So,
\begin{align}
    &\|\partial^\alpha (\tanh\circ \mathcal{A}_1-\tanh\circ \mathcal{A}_1')\|_{\infty, \Omega}\nonumber\\
    &\quad \leqslant (K+1)2^{2K-1}(K+2)!(K+3)!  \|\theta-\theta'\|_2 (1+d_1 M(\Omega)) \max(\|\theta\|_2, \|\theta'\|_2)^{K} \label{eq:trigPart2}.
\end{align}
Combining inequalities \eqref{eq:trig}, \eqref{eq:trigPart1}, and \eqref{eq:trigPart2}, we conclude that
\[\|\partial^\alpha(u_\theta-u_{\theta'})\|_{\infty, \Omega} \leqslant \tilde{C}_{K,1}  (1+d_1M(\Omega)) (D+1)(1+\max(\|\theta\|_2, \|\theta'\|_2))^{K+1} \|\theta-\theta'\|_2, \]
so that 
$\|u_\theta-u_{\theta'}\|_{C^K(\Omega)} \leqslant \tilde{C}_{K,1}  (1+d_1M(\Omega)) (D+1)(1+\max(\|\theta\|_2, \|\theta'\|_2))^{K+1} \|\theta-\theta'\|_2$.

\paragraph{Induction} Fix $H\geqslant 1$, and assume that for all $u_\theta, u_{\theta'} \in \mathrm{NN}_H(D)$ and all $K \geqslant 0$,
\begin{align}
    &\|u_\theta-u_{\theta'}\|_{C^K(\Omega)} \nonumber \\
    &\quad \leqslant \tilde{C}_{K,H}  (1+d_1M(\Omega)) (D+1)^{H+KH^2}(1+\max(\|\theta\|_2, \|\theta'\|_2))^{H+KH^2} \|\theta-\theta'\|_2. \label{eq:recLip}
\end{align}
Let $u_\theta, u_{\theta'} \in \mathrm{NN}_{H+1}(D)$. Observe that $u_\theta = \mathcal{A}_{H+2}\circ\tanh \circ v_\theta$ and $u_{\theta'} = \mathcal{A}_{H+2}'\circ\tanh \circ v_{\theta'}$, where  $v_\theta, v_{\theta'} \in \mathrm{NN}_H(D)$.
Moreover, 
\begin{align}
    &\|\partial^\alpha(u_\theta-u_{\theta'})\|_{\infty, \Omega} \nonumber \\
    &\quad \leqslant \|(W_{H+2}-W_{H+2}') \partial^\alpha( \tanh \circ v_\theta)\|_{\infty, \Omega} + \|W_{H+2}' \partial^\alpha (\tanh \circ v_\theta - \tanh\circ v_{\theta'})\|_{\infty, \Omega} \nonumber \\
    &\quad \leqslant D(\|\theta-\theta'\|_2 \times \|\partial^\alpha( \tanh \circ v_\theta)\|_{\infty, \Omega} + \|\theta'\|_2\times\|\partial^\alpha (\tanh \circ v_\theta - \tanh\circ v_{\theta'})\|_{\infty, \Omega}). \label{eq:lipTrig}
\end{align}
Since $\tanh\circ v_\theta \in \mathrm{NN}_{H+1}(D)$, we have, by Proposition \ref{prop:bounding},
\begin{equation}
    \|\partial^\alpha( \tanh \circ v_\theta)\|_{\infty, \Omega} \leqslant C_{K,H+1}(D+1)^{1+K(H+1)}(1+\|\theta\|_2)^{K(H+1)}\|\theta\|_2. \label{eq:boundTrig1}
\end{equation} 
Moreover, using \eqref{eq:fundRes}, Lemma \ref{lem:derTanh}, and the definition of $C_{K,H+1}$ in \eqref{eq:defC}, we have 
\begin{align}
      &\|\partial^\alpha (\tanh \circ v_\theta - \tanh\circ v_{\theta'})\|_{\infty, \Omega}
      \nonumber\\
      & \quad \leqslant B_K(K+1) \|\tanh^{(K+1)}\|_{\infty} \|v_\theta-v_{\theta'}\|_{C^K(\Omega)} \|\tanh^{(K)}\|_{\infty} \nonumber \\
      &\qquad \times C_{K,H+1} (D+1)^{KH}  (1+\max(\|\theta\|_2, \|\theta'\|_2))^{KH} \nonumber \\
      &\quad \leqslant 2^{2K-1} (K+3)!(K+2)! B_K(K+1) \|v_\theta-v_{\theta'}\|_{C^K(\Omega)} \nonumber \\
      &\qquad \times C_{K,H+1} (D+1)^{KH}  (1+\max(\|\theta\|_2, \|\theta'\|_2))^{KH}. \label{eq:boundingTrig2}
\end{align}
 The term $\|v_\theta-v_{\theta'}\|_{C^K(\Omega)}$ in \eqref{eq:boundingTrig2} can be upper bounded using the induction assumption \eqref{eq:recLip}. Thus, combining 
\eqref{eq:lipTrig}, \eqref{eq:boundTrig1}, and \eqref{eq:boundingTrig2}, we conclude as desired that for all $u_\theta, u_{\theta'} \in \mathrm{NN}_{H+1}(D)$ and all $K \in \mathbb{N}$,
\begin{align*}
\|u_\theta-u_{\theta'}\|_{C^K(\Omega)}&\leqslant \tilde{C}_{K,H+1}  (1+d_1M(\Omega)) (D+1)^{(H+1)+K(H+1)^2}\\
& \quad \times (1+\max(\|\theta\|_2, \|\theta'\|_2))^{(H+1)+K(H+1)^2} \|\theta-\theta'\|_2.
\end{align*}
\end{proof}

\subsection{Uniform approximation of integrals}
\label{app:glivenko_cantelli}
Throughout this section, the parameters $H,D\in \mathbb{N}^\star$ are held fixed, as well as the neural architecture $\mathrm{NN}_H(D)$  parameterized by $\Theta_{H,D}$. We let $d$ be a metric in $\Theta_{H,D}$, and denote by $B(r, d)$ the closed ball in $\Theta_{H,D}$ centered at $0$ and of radius $r$ according to the metric $d$, that is, $B(r, d)= \{\theta \in \Theta_{H,D},\ d(0,\theta) \leqslant r\}$.
\begin{thm}[Uniform approximation of integrals]
\label{thm:approx_integral}
    Let $\Omega \subseteq \mathbb{R}^{d_1}$ be a bounded Lipschitz domain, let $\alpha_1 >0$, and let $\bX_1, \hdots, \bX_n$ be a sequence of i.i.d.~random variables in $\bar{\Omega}$, with distribution $\mu_{X}$. Let $f:C^\infty(\bar{\Omega}, \mathbb{R}^{d_2})\times\bar{\Omega}\to\mathbb{R}^{d_2}$ be an operator, and assume that the following two requirements are satisfied: 
    \begin{itemize}
        \item[$(i)$] there exist $C_1 >0$ and $\beta_1 \in [0, 1/2[$ such that, for all $n\geqslant 1$ and all $\theta, \theta'\in  B(n^{\alpha_1}, \|.\|_2)$,
    \begin{equation}
        \|f(u_\theta, \cdot)-f(u_{\theta'},\cdot)\|_{\infty, \bar{\Omega}} \leqslant C_1  n^{\beta_1}  \|\theta-\theta'\|_2;\label{eq:cdn1}
    \end{equation}
\item[$(ii)$] there exist $C_2 >0$ and $\beta_2 \in [0, 1/2[$ satisfying $\beta_2 > \alpha_1 + \beta_1$ such that, for all $n \geqslant 1$ and all $\theta \in B(n^{\alpha_1}, \|.\|_2)$,
    \begin{equation}
        \|f(u_\theta, \cdot)\|_{\infty, \bar{\Omega}} \leqslant C_2  n^{\beta_2}. \label{eq:cdn2}
    \end{equation}
    \end{itemize}
Then, almost surely, there exists $N \in \mathbb{N}^\star$ such that, for all $n\geqslant N$,
    \[\sup_{\theta\in B(n^{\alpha_1}, \|.\|_2)} \Big\|\frac{1}{n}\sum_{i=1}^n f(u_\theta, \bX_i) - \int_{\bar{\Omega}} f(u_\theta,\cdot) d\mu_{X}\Big\|_2 \leqslant \log^2(n) n^{\beta_2 - 1/2}.\]
    (Notice that the rank $N$ is random.)
\end{thm}
\begin{proof} Let us start the proof by considering the case $d_2=1$.
For a given $\theta \in B(n^{\alpha_1}, \|\cdot\|_2)$, we let 
\[
Z_{n,\theta} = \frac{1}{n}\sum_{i=1}^n f(u_\theta, \bX_i) - \int_{\bar{\Omega}} f(u_\theta, \cdot) d\mu_X.
\] 
We are interested in bounding the random variable 
\[
Z_n = \sup_{\theta \in B(n^{\alpha_1}, \|\cdot\|_2)} |Z_{n,\theta}| = \sup_{\theta \in B(n^{\alpha_1}, \|\cdot\|_2)} Z_{n,\theta}.
\] 
Note that there is no need of absolute value in the rightmost term since, for any $ \theta=(W_1, b_1,\hdots,$ $ W_{H+1}, b_{H+1}) \in B(n^{\alpha_1}, \|\cdot\|_2)$, it is clear that $\theta' = (W_1, b_1,\hdots, W_H, b_H$, $ -W_{H+1}, -b_{H+1}) \in B(n^{\alpha_1}, \|\cdot\|_2)$ and $ u_{\theta'} = -u_\theta$.
Let $M(\Omega) = \max_{\bx\in\bar{\Omega}} \|x\|_2$.
Using inequality \eqref{eq:cdn1}, we have, for any $\theta, \theta' \in B(n^{\alpha_1}, \|\cdot\|_2)$, 
\[
\Big|\frac{1}{n}\Big(f(u_\theta, \bX_i)-\int_{\bar{\Omega}} f(u_\theta,\cdot)d\mu_X\Big) - \frac{1}{n}\Big(f(u_\theta', \bX_i)-\int_{\bar{\Omega}} f(u_\theta',\cdot)d\mu_X\Big) \Big| \leqslant 2C_1 n^{\beta_1 - 1}\|\theta-\theta'\|_2.
\] 
According to Hoeffding's theorem \citep[][Lemma 3.6]{vanhandel2016lectures},  the random variable $n^{-1}(f(u_\theta, \bX_i) $ $-\int_{\bar{\Omega}} f(u_\theta,\cdot)d\mu_X) - n^{-1}(f(u_\theta', \bX_i)-\int_{\bar{\Omega}} f(u_\theta',\cdot)d\mu_X)$ is subgaussian with parameter $4C_1^2 n^{2\beta_1-2} \|\theta-\theta'\|_2^2$. 
Invoking Azuma's theorem \citep[][Lemma 3.7]{vanhandel2016lectures}, we deduce that $Z_{n,\theta}-Z_{n,\theta'}$, is also subgaussian, with parameter $4C_1^2n^{2\beta_1-1}\|\theta-\theta'\|_2^2$. 
Since $\mathbb{E}(Z_{n,\theta}) = 0$, we conclude that for all $n\geqslant 1$, $(Z_{n,\theta})_{\theta\in B(n^{\alpha_1}, \|\cdot\|_2)}$ is a subgaussian process on $B(n^{\alpha_1}, \|\cdot\|_2)$ for the metric $d(\theta,\theta') = 2C_1n^{\beta_1-1/2}\|\theta-\theta'\|_2$.
Moreover, since $\theta \mapsto Z_{n,\theta}$ is continuous for the topology induced by the metric $d$, $(Z_{n,\theta})_{\theta \in B(n^{\alpha_1}, \|\cdot\|_2)}$ is separable \citep[][Remark 5.23]{vanhandel2016lectures}.
Thus, by Dudley's theorem \citep[][Corollary 5.25]{vanhandel2016lectures}
\[\mathbb{E}(Z_n) \leqslant 12 \int_0^\infty [\log N(B( n^{\alpha_1}, \|\cdot\|_2), d, r)]^{1/2}dr,\] where $N(B(n^{\alpha_1}, \|\cdot\|_2), d, r)$ is the minimum number of balls  of radius $r$ according to the metric $d$ needed to cover the space $B( n^{\alpha_1}, \|\cdot\|_2)$. 
Clearly,
$N(B( n^{\alpha_1}, \|\cdot\|_2),\ d,\ r) = N(B( n^{\alpha_1}, \|\cdot\|_2),\ \|\cdot\|_2,\ n^{1/2-\beta_1}r/(2C_1))$.
Thus,
\[\mathbb{E}(Z_n) \leqslant 24 C_1 n^{\beta_1-1/2} \int_0^\infty [\log N(B( n^{\alpha_1}, \|\cdot\|_2), \|\cdot\|_2, r)]^{1/2}dr\] 
and, in turn,
\[
\mathbb{E}(Z_n) \leqslant 24 C_1 n^{\alpha_1+\beta_1-1/2} \int_0^\infty [\log N(B(1, \|\cdot\|_2),\ \|\cdot\|_2, r)]^{1/2}dr.
\]
Upon noting that $N(B(1, \|\cdot\|_2),\ \|\cdot\|_2, r) = 1$ for $r\geqslant 1$, we are led to
\[
\mathbb{E}(Z_n) \leqslant 24 C_1 n^{\alpha_1+\beta_1-1/2} \int_0^1 [\log N(B(1, \|\cdot\|_2),\ \|\cdot\|_2, r)]^{1/2}dr.
\] 
 Since $\Theta_{H,D} = \mathbb{R}^{(d_1+1)D+(H-1)D(D+1)+(D+1)d_2}$, according to \citet[][Lemma 5.13]{vanhandel2016lectures}, one has
\[
\log N(B(1, \|\cdot\|_2),\ \|\cdot\|_2, r) \leqslant [(d_1+1)D+(H-1)D(D+1)+(D+1)d_2]\log(3/r).
\] 
Notice that $\int_0^1 \log(3/r)^{1/2}dr \leqslant 3/2$. Therefore,  
\begin{equation}
    \mathbb{E}(Z_n) \leqslant 36 C_1[(d_1+1)D+(H-1)D(D+1)+(D+1)d_2]^{1/2} n^{\alpha_1+\beta_1-1/2}. \label{eq:expectancyZn}
\end{equation}
Next, observe that, by definition of $Z_n=Z_n(\bX_1, \hdots, \bX_n)$, 
\begin{align*}
    &\sup_{\bx_i \in \mathbb{R}^{d_1}}Z_n(\bX_1, \hdots, \bX_{i-1}, \bx_i, \bX_{i+1}, \hdots, \bX_n) - \inf_{\bx_i \in \mathbb{R}^{d_1}}Z_n(\bX_1, \hdots, \bX_{i-1}, \bx_i, \bX_{i+1}, \hdots, \bX_n)\\ 
    & \quad \leqslant 2n^{-1}\sup_{\theta \in B(n^{\alpha_1}, \|\cdot\|_2)}\Big\|f(u_\theta, \bX_i) - \int_{\bar{\Omega}} f(u_\theta,\cdot) d\mu_X\Big\|_2\\
    & \quad \leqslant 4n^{-1}\sup_{\theta \in B(n^{\alpha_1}, \|\cdot\|_2)}\|f(u_\theta, \cdot)\|_\infty.
\end{align*} 
Using inequality \eqref{eq:cdn2}, McDiarmid's inequality  \citep[Theorem 3.11]{vanhandel2016lectures} ensures that $Z_n$ is subgaussian with parameter $4C_2^2n^{2\beta_2-1}$. In particular, for all $t_n \geqslant 0$, $\mathbb{P}(|Z_n - \mathbb{E}(Z_n)|\geqslant t_n) \leqslant 2 \exp(-n^{1-2\beta_2}t_n^2/(8C_2^2))$, which is summable with $t_n = C_3 n^{\beta_2-1/2}\log^2(n)$, where $C_3$ is any positive constant.
Thus,  recalling that $\beta_2 > \alpha_1+\beta_1$, the Borel-Cantelli lemma and \eqref{eq:expectancyZn} ensure that,
almost surely, for all $n$ large enough,
$0 \leqslant Z_n \leqslant 2 C_3 n^{\beta_2-1/2}\log^2(n)$. Taking $C_3 = 1/2$ yields the desired result.

The generalization to the case $d_2 \geqslant 2$ is easy. Just note, letting $f = (f_1, \hdots, f_{d_2})$, that 
\begin{align*}
    &\sup_{\theta \in B(n^{\alpha_1}, \|\cdot\|_2)} \Big\|\frac{1}{n}\sum_{i=1}^n f(u_\theta, \bX_i) - \int_{\bar{\Omega}} f(u_\theta, \cdot) d\mu_X\Big\|_2 \\
    &\qquad \leqslant \sqrt{d_2} \max_{1 \leqslant j\leqslant d_2}\sup_{\theta \in B(n^{\alpha_1}, \|\cdot\|_2)} \Big\|\frac{1}{n}\sum_{i=1}^n f_j(u_\theta, \bX_i) - \int_{\bar{\Omega}} f_j(u_\theta, \cdot) d\mu_X\Big\|_2 .
\end{align*}
Taking $C_3 = d_2^{-1/2}/2$ as above leads to the result.
\end{proof}
\begin{prop}[Condition function]
\label{prop:GKopH}
    Let $\Omega$ be a bounded Lipschitz domain, let $E$ be a closed subset of $\partial\Omega$, and let $h \in \mathrm{Lip}(E,\mathbb{R}^{d_2})$. Then the operator $\mathscr{H}(u, \bx) = \mathbf{1}_{\bx \in E} \|u(\bx) -h(\bx)\|^2$ satisfies inequalities \eqref{eq:cdn1} and \eqref{eq:cdn2} with $\alpha_1 < (3+H)^{-1}/2$, $\beta_1=(1+H)\alpha_1$, and  $1/2 > \beta_2 \geqslant (3+H)\alpha_1$.  
\end{prop}
\begin{proof} First note, since $\mathrm{Lip}(E,\mathbb{R}^{d_2}) \subseteq C^0(E, \mathbb{R}^{d_2})$, that $\|h\|_\infty < \infty$. Observe also that for any $v,w \in \mathbb{R}^{d_2}$, $|\|v\|_2^2-\|w\|_2^2|= |\langle v+w, v-w\rangle|\leqslant \|v+w\|_2\|v-w\|_2 \leqslant d_2 \|v+w\|_\infty\|v-w\|_\infty$, where $\langle\cdot, \cdot\rangle$ denotes the canonical scalar product. Thus, 
we obtain, for all $\theta, \theta' \in B(n^{\alpha_1}, \|\cdot\|_2)$ and all $\bx\in E$,
\begin{align*}
    |\mathscr{H}(u_\theta,\bx)-\mathscr{H}(u_{\theta'},\bx)| &\leqslant (\|u_\theta(\bx)\|_2 + \|u_{\theta'}(\bx)\|_2 + 2\|h(\bx)\|_2) \|u_\theta(\bx) - u_{\theta'}(\bx)\|_2\\
    &\leqslant d_2 (\|u_\theta\|_{\infty, \bar{\Omega}} + \|u_{\theta'}\|_{\infty, \bar{\Omega}} + 2\|h\|_\infty) \|u_\theta - u_{\theta'}\|_{\infty, \bar{\Omega}} \\
    &\leqslant d_2 (2(D+1)n^{\alpha_1} + 2\|h\|_\infty) \|u_\theta - u_{\theta'}\|_{\infty, \bar{\Omega}}  \;\; \mbox{(by inequality \eqref{eq:boundNN})} \\
    &\leqslant 2 d_2 ((D+1)n^{\alpha_1} + \|h\|_\infty)\tilde{C}_{0,H}  (1+d_1M(\Omega)) \\
    &\qquad \times (D+1)^{H}(1+n^{\alpha_1})^{H} \|\theta-\theta'\|_2 \quad \mbox{(by Proposition \ref{prop:lipschitzParam})}\\
    &\leqslant C_1 n^{\beta_1} \|\theta- \theta'\|_2,
\end{align*}
where $\beta_1 = (1+H)\alpha_1$ and 
$C_1 = 2^{H+1} d_2 (D+1+ \|h\|_\infty)\tilde{C}_{0,H}  (1+d_1M(\Omega)) (D+1)^{H}$. 

Next,  using \eqref{eq:boundNN} once again, for all $\theta \in B(n^{\alpha_1}, \|.\|_2)$, 
$\|\mathscr{H}(u_\theta, \cdot)\|_{\infty, \bar{\Omega}} \leqslant d_2(\|u_\theta\|_{\infty, \bar{\Omega}}+\|h\|_\infty)^2\leqslant d_2 ((D+1)n^{\alpha_1} + \|h\|_\infty)^2 \leqslant C_2 n^{2 \alpha_1}$. Recall that for inequality \eqref{eq:cdn2}, $\beta_2$ must satisfy $\alpha_1 + \beta_1 < \beta_2 < 1/2$. This is true for $\beta_2 = (3+H)\alpha_1$, which completes the proof.
\end{proof}

\begin{prop}[Polynomial operator]
\label{prop:GKopF}
    Let $\Omega$ be a bounded Lipschitz domain, and let $\mathscr{F} \in {\mathscr P}_{\mathrm{op}}$. Then the operator $\mathbf{1}_{\bx \in \Omega} \mathscr{F}(u_\theta, \bx)^2$ satisfies inequalities \eqref{eq:cdn1} and \eqref{eq:cdn2} with $\alpha_1 < [2+H(1+ (2+H)\deg(\mathscr{F}))]^{-1}/2$, $\beta_1 = H(1+ (2+H)\deg(\mathscr{F}))\alpha_1$, and $1/2 > \beta_2 \geqslant [2+H(1+ (2+H)\deg(\mathscr{F}))]\alpha_1$. 
\end{prop}
\begin{proof} 
Let $\mathscr{F} \in {\mathscr P}_{\mathrm{op}}$. 
By definition, there exist a degree $s\geqslant 1$, a polynomial $P \in C^\infty(\mathbb{R}^{d_1}, \mathbb{R})[Z_{1,1},$ $\hdots,$ $ Z_{d_2, s}]$, and a sequence $(\alpha_{i,j})_{1\leqslant i\leqslant d_2, 1\leqslant j\leqslant s}$ of multi-indexes such that, for any $u\in C^\infty(\bar{\Omega},\mathbb{R}^{d_2})$, $\mathscr{F}(u,\cdot) = P((\partial^{\alpha_{i,j}}u_i)_{1\leqslant i\leqslant d_2, 1\leqslant j\leqslant s})$. 
Namely, there exists $N(P) \in \mathbb{N}^\star$, exponents $ I(i,j,k) \in \mathbb{N}$, and functions $\phi_1, \hdots, \phi_{N(P)} \in C^\infty(\bar{\Omega}, \mathbb{R})$, such that $P(Z_{1,1}, \hdots, Z_{d_2,s}) = \sum_{k=1}^{N(P)} \phi_k \times \prod_{i=1}^{d_2}\prod_{j=1}^{s} Z_{i,j}^{I(i,j,k)}$. 
Recall, by Definition \ref{defi:deg}, that $\deg(\mathscr{F})= \max_k \sum_{i=1}^{d_2}\sum_{j=1}^{s} (1+|\alpha_{i,j}|)I(i,j,k)$.

Now, according to Proposition \ref{prop:bounding}, there exists a positive constant $C_{\mathrm{deg}(\mathscr{F}),H}$ such that 
\begin{align*}
    &\|\mathscr{F}(u_\theta,\cdot)^2\|_{\infty,\bar{\Omega}} \\
    &\quad \leqslant \bigg[\sum_{k=1}^{N(P)} \|\phi_k\|_{\infty, \bar{\Omega}}  \prod_{i=1}^{d_2}\prod_{j=1}^{s} \|\partial^{\alpha_{i,j}}u_\theta\|_{\infty, \bar{\Omega}}^{I(i,j,k)}\bigg]^2\\
    & \quad \leqslant N^2(P)  \big[\max_{1\leqslant k \leqslant N(P)}\|\phi_k\|_{\infty, \bar{\Omega}}\big]^2  C_{\mathrm{deg}(\mathscr{F}),H}^2(D+1)^{2 H \deg(\mathscr{F})} (1+\|\theta\|_2)^{2 H\deg(\mathscr{F})}.
\end{align*}
Thus, for any $\theta\in B(n^{\alpha_1}, \|\cdot\|_2)$, $\|\mathscr{F}(u_\theta,\cdot)^2\|_{\infty,\bar{\Omega}} \leqslant C_2 n^{\beta_2}$,  where 
\[C_2 = 2^{2H\deg(\mathscr{F})} N^2(P)  \big[\max_{1\leqslant k \leqslant N(P)}\|\phi_k\|_{\infty, \bar{\Omega}}\big]^2  C_{\mathrm{deg}(\mathscr{F}),H}^2(D+1)^{2H\deg(\mathscr{F})}, 
\]
and for any $\beta_2 \geqslant 2H\deg(\mathscr{F}) \alpha_1$.

Next, observe that, any $u$ and $v$, $||u|^2-|v|^2|= |(u+v)(u-v)|\leqslant |u+v||u-v|$. Therefore,
\begin{align*}
    |\mathscr{F}(u_\theta,\bx)^2 - \mathscr{F}(u_{\theta'},\bx)^2| &\leqslant \big(|\mathscr{F}(u_\theta,\bx)|+| \mathscr{F}(u_{\theta'},\bx)| \big) |\mathscr{F}(u_\theta,\bx) - \mathscr{F}(u_{\theta'},\bx)|\\
    &\leqslant 2C_2^{1/2} n^{H \deg(\mathscr{F})\alpha_1}|\mathscr{F}(u_\theta,\bx) - \mathscr{F}(u_{\theta'},\bx)|.
\end{align*}
Using inequality \eqref{eq:diffProd} (remark that the product $\prod_{i=1}^{d_2}\prod_{j=1}^{s} Z_{i,j}^{I(i,j,k)}$ has less than $\deg(\mathscr{F})$ terms different from $1$), it is easy to see that
\begin{align*}
    |\mathscr{F}(u_\theta,\bx) - \mathscr{F}(u_{\theta'},\bx)|&\leqslant N(P)  \big[\max_{1\leqslant k \leqslant N(P)}\|\phi_k\|_{\infty, \bar{\Omega}}\big] \deg(\mathscr{F}) \|u_\theta-u_{\theta'}\|_{C^{\deg(\mathscr{F})}(\Omega)}\\
    &\quad \times \max_{1\leqslant k \leqslant N(P)}\prod_{i,j} \max(\|u_\theta\|_{C^{|\alpha_{i,j}|}(\Omega)}, \|u_{\theta'}\|_{C^{|\alpha_{i,j}|}(\Omega)}) ^{I(i,j,k)}. 
\end{align*}
From Proposition \ref{prop:bounding}, we deduce that
\begin{align*}
    &\max_{1\leqslant k \leqslant N(P)}\prod_{i,j} \max(\|u_\theta\|_{C^{|\alpha_{i,j}|}(\Omega)}, \|u_{\theta'}\|_{C^{|\alpha_{i,j}|}(\Omega)}) ^{I(i,j,k)} \\
    &\quad\leqslant C_{\deg(\mathscr{F}),H} (D+1)^{H\deg(\mathscr{F})}(1+\max(\|\theta\|_2, \|\theta'\|_2))^{H\deg(\mathscr{F})}.
\end{align*}

Combining the last two inequalities with  Proposition \ref{prop:lipschitzParam} gives that
\begin{align*}
    &|\mathscr{F}(u_\theta,\bx) - \mathscr{F}(u_{\theta'},\bx)| \\
    &\quad \leqslant  N(P)  \big[\max_{1\leqslant k \leqslant N(P)}\|\phi_k\|_{\infty, \bar{\Omega}}\big] \deg(\mathscr{F})  \tilde{C}_{\deg(\mathscr{F}),H}(1+d_1M(\Omega))\|\theta-\theta'\|_2\\
    &\qquad \times  C_{\deg(\mathscr{F}),H} (D+1)^{H(1+ (1+H)\deg(\mathscr{F}))}(1+\max(\|\theta\|_2, \|\theta'\|_2))^{H(1+ (1+H)\deg(\mathscr{F}))}.
\end{align*}
Hence, for all $\theta, \theta'\in B(n^{\alpha_1}, \|\cdot\|_2)$, $ |\mathscr{F}(u_\theta,\bx)^2 - \mathscr{F}(u_{\theta'},\bx)^2| \leqslant C_1 n^{\beta_1} \|\theta-\theta'\|_2$, where 
\begin{align*}
C_1 &= 2C_2^{1/2}N(P)  \big[\max_{1\leqslant k \leqslant N(P)}\|\phi_k\|_{\infty, \bar{\Omega}}]\big] \deg(\mathscr{F}) \tilde{C}_{\deg(\mathscr{F}), H} (1+d_1M(\Omega)) \\
& \quad \times C_{\deg(\mathscr{F}), H}  (D+1)^{H(1+ (1+H)\deg(\mathscr{F}))} 2^{H(1+ (1+H)\deg(\mathscr{F}))}
\end{align*}
and $\beta_1 = H(1+ (2+H)\deg(\mathscr{F}))\alpha_1$. 

 Recall that for inequality \eqref{eq:cdn2}, $\beta_2$ must satisfy $\alpha_1 + \beta_1 < \beta_2 < 1/2$. This is true for $\beta_2 = [2+H(1+ (2+H)\deg(\mathscr{F}))]\alpha_1$ and $\alpha_1 < [2+H(1+ (2+H)\deg(\mathscr{F}))]^{-1}/2$.
\end{proof}

\subsection{Proof of Theorem \ref{thm:generalization_error}}
Let $u_0 = 0 \in \mathrm{NN}_H(D)$ be the neural network with parameter $\theta = (0, \hdots, 0)$. 
Obviously, $R_{n, n_e, n_r}^{(\mathrm{ridge})}(u_0) = R_{n, n_e, n_r}(u_0)$. Also, 
\begin{equation*}
R_{n, n_e, n_r}(u_0)  \leqslant \frac{\lambda_d}{n}\sum_{i=1}^{n} \|Y_i\|_2^2 + {\lambda_e}\|h\|_\infty  + \frac{1}{n_r}\sum_{k=1}^M \sum_{\ell=1}^{n_r} \|\mathscr{F}_k(0, \bX^{(r)}_\ell)\|_2^2.
\end{equation*}
Since each $\mathscr{F}_k$
 is a polynomial operator (see Definition \ref{defi:polyop}), it takes the form \[\mathscr{F}_k(u, \bx) = \sum_{\ell=1}^{N(P_k)} \phi_{\ell,k} \prod_{i=1}^{d_2}\prod_{j=1}^{s_k} (\partial^{\alpha_{i,j,k}}u_i(\bx))^{I_k(i,j,\ell)}.\] Therefore,
\begin{align}
R_{n, n_e, n_r}(u_0) 
& \leqslant \frac{\lambda_d}{n}\sum_{i=1}^{n} \|Y_i\|_2^2 + {\lambda_e}\|h\|_\infty  + \sum_{k=1}^M  \sum_{\ell = 1}^{N(P_k)}\|\phi_{\ell,k}\|_{\infty, \bar{\Omega}} \nonumber \\
&:= I, \label{eq:borne0}
\end{align}
where $I$ does not depend on $\lambda_{(\mathrm{ridge})}$, $n_e$, and $n_r$.

Let $(\hat{\theta}^{(\mathrm{ridge})}(p, n_e, n_r, D))_{p\in\mathbb{N}}$ be any minimizing sequence of the empirical risk of the ridge PINN, i.e.,  $\lim_{p \to \infty}R^{(\mathrm{ridge})}_{n, n_e, n_r}(u_{\hat{\theta}^{(\mathrm{ridge})}(p, n_e, n_r, D)}) = \inf_{\theta \in \Theta_{H,D}}\,R^{(\mathrm{ridge})}_{n, n_e, n_r}(u_\theta)$.
In the rest of the proof, we let $n_{r,e} = \min(n_r, n_e)$.
We will make use of the following three sets:
$\mathcal{E}_1(n_{r,e}) = \{\theta \in \Theta_{H,D},\ \|\theta\|_2 \geq n_{r,e}^{\kappa}\}$,
$\mathcal{E}_2(n_{r,e}) = \{\theta \in \Theta_{H,D}, \  n_{r,e}^{\kappa/4}\leq \|\theta\|_2 \leq n_{r,e}^{\kappa}\}$, and
    $\mathcal{E}_3(n_{r,e}) = \{\theta \in \Theta_{H,D}, \  \|\theta\|_2 \leq n_{r,e}^{\kappa/4}\}$.
Clearly, $\Theta_{H,D} = \mathcal{E}_1 \cup \mathcal{E}_2 \cup \mathcal{E}_3$. 
The proof relies on the argument that almost surely, given any $n_r$ and $n_e$, for all $p$ large enough, $\hat{\theta}^{(\mathrm{ridge})}(p, n_e, n_r, D) \in \mathcal{E}_2 \cup \mathcal{E}_3$. 
Moreover, on $\mathcal{E}_2 \cup \mathcal{E}_3$, the empirical risk function $R_{n, n_e, n_r}^{(\mathrm{ridge})}$ is close to the theoretical risk $\mathscr{R}_n$, when $n_{r,e}$ is large enough. 
For clarity, the proof is divided into four steps.
 
\paragraph{Step 1}
We start by observing that, for any $\theta \in \mathcal{E}_1(n_{r,e})$,
$R_{n, n_e, n_r}^{(\mathrm{ridge})}(\theta) \geqslant \lambda_{(\mathrm{ridge})}\|\theta\|_2^2\geqslant  n_{r,e}^{\kappa}$.
Therefore, according to \eqref{eq:borne0}, once $n_{r,e} \geq (I+1)^{1/\kappa}$,
\[ \inf_{\theta \in \mathcal{E}_3(n_{r,e})}R_{n, n_e,n_r}^{(\mathrm{ridge})}(u_\theta) + 1 \leqslant R_{n, n_e,n_r}^{(\mathrm{ridge})}(u_0) +1 \leqslant \inf_{\theta \in \mathcal{E}_1(n_{r,e})}R_{n, n_e,n_r}^{(\mathrm{ridge})}(u_\theta).\] This shows that, for all $n_{r,e}$ large enough and for all $p$ large enough, $\hat{\theta}^{(\mathrm{ridge})}(p, n_e, n_r, D) \notin \mathcal{E}_1(n_{r,e})$. 

\paragraph{Step 2} Applying Proposition \ref{prop:GKopH} and Proposition \ref{prop:GKopF} with $\alpha_1 = \kappa$ and $\beta_2 = (2+H(1+(2+H)\max_k\deg(\mathscr{F}_k)))\alpha_1$, and then Theorem \ref{thm:approx_integral}, we know that, almost surely, there exists $N \in \mathbb{N}^\star$ such that, for all $n_{r,e}\geqslant N$,
\begin{align}
    &\sup_{\theta\in \mathcal{E}_2(n_{r,e})\cup \mathcal{E}_3(n_{r,e})} \Big|\frac{1}{n_e}\sum_{j=1}^{n_e} \|u_\theta(\bX_j^{(e)})-h(\bX_j^{(e)})\|_2^2 - \mathbb{E}\|u_\theta(\bX^{(e)})-h(\bX^{(e)})\|_2^2 \Big| \nonumber\\
    &\quad \leqslant \log^2(n_{r,e}) n_{r,e}^{\beta_2-1/2}
    \label{eq:GKh2}
\end{align}
and, for each $1 \leqslant k \leqslant M$,
\begin{equation}
    \sup_{\theta\in \mathcal{E}_2(n_{r,e})\cup \mathcal{E}_3(n_{r,e})} \Big|\frac{1}{n_r}\sum_{\ell=1}^{n_r}  \mathscr{F}_k(u_\theta, \bX^{(r)}_\ell)^2 - \frac{1}{|\Omega|}\int_\Omega \mathscr{F}_k(u_\theta, \bx)^2 d\bx\Big| \leqslant \log^2(n_{r,e}) n_{r,e}^{\beta_2-1/2}.
    \label{eq:GKF2}
\end{equation} 
Thus, almost surely, for all $n_{r,e}$ large enough and for all $\theta \in \mathcal{E}_2(n_{r,e})$,
\begin{align*}
    R_{n, n_e, n_r}^{(\mathrm{ridge})}(u_\theta) &\geqslant  \mathscr{R}_n(u_\theta) + \lambda_{(\mathrm{ridge})}\|\theta\|_2^2  - (M+1) \log^2(n_{r,e}) n_{r,e}^{\beta_2-1/2}.
\end{align*}
But, for all $\theta \in \mathcal{E}_2(n_{r,e})$,
$\lambda_{(\mathrm{ridge})}\|\theta\|_2^2 \geqslant n_{e,r}^{-\kappa / 2}$. Upon noting that $-\nicefrac{\kappa}{2}  > \beta_2 -\nicefrac{1}{2}$, we conclude that, almost surely, for all $n_{r,e}$ large enough and for all $\theta \in \mathcal{E}_2(n_{r,e})$, $ R_{n, n_e,n_r}^{(\mathrm{ridge})}(u_\theta) \geqslant \mathscr{R}_n(u_\theta)$.

\paragraph{Step 3}
Clearly, for all $\theta \in \mathcal{E}_3(n_{r,e})$,
$\lambda_{(\mathrm{ridge})}\|\theta\|_2^2 \leqslant n_{e,r}^{-\kappa / 2}$. Using inequalities \eqref{eq:GKh2} and \eqref{eq:GKF2}, we deduce that, almost surely, for all $n_{r,e}$ large enough and for all $\theta \in \mathcal{E}_3(n_{r,e})$, $|R_{n, n_e,n_r}^{(\mathrm{ridge})}(u_\theta) - \mathscr{R}_n(u_\theta)| \leqslant (M+2)\log^2(n_{r,e}) n_{r,e}^{-\kappa/2}$.
 
\paragraph{Step 4} Fix $\varepsilon > 0$.
Let $(\theta_p)_{p\in\mathbb{N}}$ be any minimizing sequence of the theoretical risk function $\mathscr{R}_n$, that is, 
$\lim_{p\to\infty} \mathscr{R}_n(u_{\theta_p}) = \inf_{\theta \in \Theta_{H,D}}\mathscr{R}_n(u_\theta)$.
Thus, by definition, there exists some $P_\varepsilon \in \mathbb{N}$ such that $|\mathscr{R}_n(u_{\theta_{P_\varepsilon}}) - \inf_{\theta \in \Theta_{H,D}}\mathscr{R}_n(u_\theta)|\leqslant \varepsilon$. 

For fixed $n_{r,e}$, according to Step 1, we have, for all $p$ large enough, $\hat{\theta}^{(\mathrm{ridge})}(p, n_e, n_r, D) \in  \mathcal{E}_2(n_{r,e}) \cup \mathcal{E}_3(n_{r,e})$. So, according to Step 2 and Step 3,
\[\mathscr{R}_{n}(u_{\hat{\theta}^{(\mathrm{ridge})}(p, n_e, n_r, D)}) \leqslant R_{n, n_e,n_r}^{(\mathrm{ridge})}(u_{\hat{\theta}^{(\mathrm{ridge})}(p, n_e, n_r, D)})+ (M+2)\log^2(n_{r,e}) n_{r,e}^{-\kappa/2}.\] 
Now, by definition of the minimizing sequence $(\hat{\theta}^{(\mathrm{ridge})}(p, n_e, n_r, D))_{p\in\mathbb{N}}$, for all $p$ large enough, $R_{n, n_e,n_r}^{(\mathrm{ridge})}(u_{\hat{\theta}^{(\mathrm{ridge})}(p, n_e, n_r, D)}) \leqslant  \inf_{\theta \in \Theta_{H,D}} R_{n, n_e,n_r}^{(\mathrm{ridge})}(u_\theta) + \varepsilon$. Also, according to Step 3,
\begin{align*}
    \inf_{\theta \in \mathcal{E}_2(n_{r,e}) \cup \mathcal{E}_3(n_{r,e})} R_{n, n_e,n_r}^{(\mathrm{ridge})}(u_\theta) &\leqslant \inf_{\theta \in \mathcal{E}_3(n_{r,e})} R_{n, n_e,n_r}^{(\mathrm{ridge})}(u_\theta)\\
    & \leqslant \inf_{\theta \in \mathcal{E}_3(n_{r,e})} \mathscr{R}_{n}(u_\theta) + (M+2)\log^2(n_{r,e}) n_{r,e}^{-\kappa/2}.
\end{align*}
Observe that, for all $n_{r,e}$ large enough, $\theta_{P_\varepsilon} \in \mathcal{E}_3(n_{r,e})$. Therefore, $\inf_{\theta \in \mathcal{E}_3(n_{r,e})} \mathscr{R}_{n}(u_\theta) \leqslant \mathscr{R}_{n}(u_{\theta_{P_\varepsilon}})$. Combining the previous inequalities, we conclude that, almost surely, for all $n_{r,e}$ large enough and for all $p$ large enough, 
\[\mathscr{R}_{n}(u_{\hat{\theta}^{(\mathrm{ridge})}(p, n_e, n_r, D)}) \leqslant \inf_{\theta \in \Theta_{H,D}}\mathscr{R}_n(u_\theta) + 3\varepsilon.\]
Since $\varepsilon$ is arbitrary, almost surely,
$\lim_{n_e,n_r \to \infty} \lim_{p\to\infty} \mathscr{R}_{n}(u_{\hat{\theta}^{(\mathrm{ridge})}(p, n_e, n_r, D)}) = \inf_{\theta \in \Theta_{H,D}}\mathscr{R}_n(u_\theta)$.

\subsection{Proof of Theorem \ref{thm:approximation}}
The result is a direct consequence of Theorem \ref{thm:generalization_error}, Proposition \ref{prop:densite} and of the continuity of $\mathscr{R}_n$ with respect to the $C^K(\Omega)$ norm.

\section{Proofs of Section \ref{sec:functional}}
\subsection{Proof of Proposition \ref{prop:laxMLin}}
Since the functions in $H^{m+1}(\Omega, \mathbb{R}^{d_2})$ are only defined almost everywhere, we first have to give a meaning to the pointwise evaluations $u(\bX_i)$ when $u\in H^{m+1}(\Omega, \mathbb{R}^{d_2})$. 
Since $\Omega$ is a bounded Lipschitz domain and $(m +1)  > d_1/2$, we can use the Sobolev embedding of Theorem \ref{thm:sobIneq}.
Clearly, $\tilde{\Pi}$ is linear and $\|\tilde \Pi (u)\|_{\infty} \leqslant  C_{\Omega} \| u \|_{H^{m+1}(\Omega)}$. The natural choice to evaluate $u \in H^{m+1}(\Omega,\mathbb{R}^{d_2})$ at the point $\bX_i$ is therefore to evaluate its unique continuous modification $\tilde \Pi (u)$ at $\bX_i$.
    
By assumption, $\mathscr{F}_k(u, \cdot) = \mathscr{F}_k^{(\mathrm{lin})}(u, \cdot) + B_k$, where $\mathscr{F}_k^{(\mathrm{lin})}(u, \cdot) = \sum_{|\alpha|\leqslant K} \langle A_{k,\alpha}, \partial^\alpha u\rangle$ and $A_{k,\alpha} \in C^\infty(\bar{\Omega}, \mathbb{R}^{d_1})$.
Next, consider the symmetric bilinear form, defined for all $u,v \in H^{m+1}(\Omega, \mathbb{R}^{d_2})$ by
\begin{align*}
    \mathcal{A}_n(u,v) &= \frac{\lambda_d}{n} \sum_{i=1}^n \langle\tilde \Pi (u)(\bX_i) , \tilde \Pi(v)(\bX_i) \rangle+\lambda_e \mathbb{E}\langle \tilde \Pi(u)(\bX^{(e)}), \tilde \Pi(v)(\bX^{(e)})\rangle\\
    &\quad +\frac{1}{|\Omega|}\sum_{k=1}^{M}\int_{\Omega} \mathscr{F}_k^{(\mathrm{lin})}(u,\bx)\mathscr{F}_k^{(\mathrm{lin})}(v,\bx)d\bx + \frac{\lambda_t }{|\Omega|} \!\sum_{|\alpha|\leqslant m+1}\!\int_{\Omega}\langle \partial^\alpha u(\bx), \partial^\alpha v(\bx)\rangle d\bx,
    \end{align*}
along with the linear form defined for all $u \in H^{m+1}(\Omega, \mathbb{R}^{d_2})$ by
\begin{align*}
    \mathcal{B}_n(u) &= \frac{\lambda_d}{n} \sum_{i=1}^n \langle Y_i, \tilde \Pi(u)(\bX_i)\rangle + \lambda_e \mathbb{E}\langle \tilde \Pi(u)(\bX^{(e)}), h(\bX^{(e)})\rangle\\
    &\quad - \frac{1}{|\Omega|}\sum_{k=1}^{M}\int_{\Omega} B_k(\bx)\mathscr{F}_k^{(\mathrm{lin})}(v,\bx)d\bx .
\end{align*}
Observe that \[\mathcal{A}_n(u,u) -2\mathcal{B}_n(u) = \mathscr{R}_n^{(\mathrm{reg})}(u) - \frac{\lambda_d}{n} \sum_{i=1}^n \|Y_i\|_2^2 -\lambda_e \mathbb{E}\|h(\bX^{(e)})\|_2^2 - \frac{1}{|\Omega|}\sum_{k=1}^{M}\int_{\Omega} B_k(\bx)^2d\bx.\] 
In addition, $\mathcal{A}_n(u,u) \geqslant \lambda_t \|u\|_{H^{m+1}(\Omega)}^2$, where $\lambda_t >0$, so that $\mathcal{A}_n$ is coercive on the normed space $(H^{m+1}(\Omega), \|\cdot\|_{H^{m+1}(\Omega)})$.  
Since $(m+1) > \max(d_1/2, K)$, one has that 
    \[|\mathcal{A}_n(u,v)| \leqslant ((\lambda_d+\lambda_e) C_{\Omega}^2 + \sum_{1\leqslant k \leqslant M}(\sum_{|\alpha|\leqslant K} \|A_{k,\alpha}\|_{\infty,\Omega})^2 +\lambda_t)\| u \|_{H^{m+1}(\Omega)}\| v \|_{H^{m+1}(\Omega)},\]
    and 
    \[|\mathcal{B}_n(u)| \leqslant C_\Omega \Big(\frac{\lambda_d}{n} \sum_{i=1}^n \|Y_i\|_2  + \lambda_e \|h\|_{\infty} + \sum_{k=1}^M(\|B_{k}\|_{\infty,\Omega}\sum_{|\alpha|\leqslant K} \|A_{k,\alpha}\|_{\infty,\Omega})\Big) \| u \|_{H^{m+1}(\Omega)}.\]
    This shows that the operators $\mathcal{A}_n$ and $\mathcal{B}_n$ are continuous. Therefore, by the Lax-Milgram theorem \citep[e.g.,][Corollary 5.8]{brezis2010functional}, there exists a unique $\hat u \in H^{m+1}(\Omega, \mathbb{R}^{d_2})$ such that $\mathcal{A}_n(\hat u,\hat u) - 2\mathcal{B}_n(\hat u) = \min_{u \in H^{m+1}(\Omega, \mathbb{R}^{d_2})}\mathcal{A}_n(u,u) - 2\mathcal{B}_n(u)$. This directly implies that $\hat u$ is the unique minimizer of $\mathscr{R}_n^{(\mathrm{reg})}$ over $H^{m+1}(\Omega, \mathbb{R}^{d_2})$. Furthermore, the Lax-Milgram theorem also states that $\hat u$ is the unique element of $H^{m+1}(\Omega, \mathbb{R}^{d_2})$ such that, for all $v \in H^{m+1}(\Omega, \mathbb{R}^{d_2})$, $\mathcal{A}_n(\hat u,v) = \mathcal{B}_n(v)$. This concludes the proof of the proposition.

\subsection{Proof of Proposition \ref{prop:sequenceCvLin}}
Let $\hat u_n$ be the unique minimizer of the regularized theoretical risk $\mathscr{R}^{(\mathrm{reg})}_n$ over $H^{m+1}(\Omega, \mathbb{R}^{d_2})$ given by Proposition \ref{prop:laxMLin}. Notice that
\[\inf_{u \in C^{\infty}(\bar{\Omega}, \mathbb{R}^{d_2})} \mathscr{R}^{\mathrm{(reg)}}_n(u) = \inf_{u \in H^{m+1}(\Omega, \mathbb{R}^{d_2})} \mathscr{R}^{\mathrm{(reg)}}_n(u) = \mathscr{R}_n(\hat u_n).\]
The first equality is a consequence of the density of $ C^{\infty}(\bar{\Omega}, \mathbb{R}^{d_2})$ in $H^{m+1}(\Omega,\mathbb{R}^{d_2})$, together with the continuity of the function $\mathscr{R}^{\mathrm{(reg)}}_n : H^{m+1}(\Omega, \mathbb{R}^{d_2}) \to \mathbb{R}$ with respect to the $H^{m+1}(\Omega)$ norm (see the proof of Proposition \ref{prop:laxMLin}). The density argument follows from the extension theorem of \citet[][Chapter VI.3.3, Theorem 5]{stein1970lipschitz} and from \citet[Chapter 5.3, Theorem 3]{evans2010partial}.

Our goal is to show that the regularized theoretical risk satisfies some form of continuity, so that we can connect $\mathscr{R}^{(\mathrm{reg})}(u_p)$ and $\mathscr{R}^{(\mathrm{reg})}(\hat u_n)$. 
Recall that, by assumption, $\mathscr{F}_k(u, \cdot) = \mathscr{F}_k^{(\mathrm{lin})}(u, \cdot) + B_k$, where $\mathscr{F}_k^{(\mathrm{lin})}(u, \cdot) = \sum_{|\alpha|\leqslant K} \langle A_{k,\alpha}(\cdot), \partial^\alpha u(\cdot)\rangle$ and $A_{k,\alpha} \in C^\infty(\bar{\Omega}, \mathbb{R}^{d_1})$. Observe that 
    \begin{equation}
        \mathscr{R}_n^{(\mathrm{reg})}(u) = F(u)+ \frac{1}{|\Omega|}I(u),
        \label{eq:decomposition}
    \end{equation}
where \[F(u) =   \frac{\lambda_d}{n} \sum_{i=1}^n \|\tilde \Pi (u)(\bX_i) - Y_i\|_2^2 + \lambda_e \mathbb{E}\|\tilde \Pi (u)(\bX^{(e)})-h(\bX^{(e)})\|_2^2 ,\] 
\[I(u) = \int_{\Omega}L((\partial^{m+1}_{i_1,\hdots, i_{m+1}}u(\bx))_{1\leqslant i_1,\hdots, i_{m+1} \leqslant d_1},\hdots,u(\bx), \bx)d\bx,\] 
and where the function $L$ satisfies
\[L(x^{(m+1)}, \hdots, x^{(0)}, z) =  \sum_{k=1}^M\Big(B_k(z) + \sum_{|\alpha|\leqslant K} \langle A_{k,\alpha}(z), x^{(|\alpha|)}_{\alpha}\rangle\Big)^2 +  \lambda_t \sum_{j=0}^{m+1} \|x^{(j)}\|_2^2.\]
(The term $x^{(j)} \in \mathbb{R}^{\binom{d_1+j-1}{j-1}d_2}$ corresponds to the to the concatenation of all the partial derivatives of order $j$, i.e., to the term $(\partial^{j}_{i_1,\hdots, i_{j}}u(\bx))_{1\leqslant i_1,\hdots, i_{j} \leqslant d_1}$.)
Clearly, $L\geqslant 0$ and, since $(m+1) > K$, the Lagrangian $L$ is convex in $x^{(m+1)}$. 
Therefore, according to Lemma \ref{lem:lowerSemiC0Lin}, the function $I$ is weakly lower-semi continuous on $H^{m+1}(\Omega, \mathbb{R}^{d_2})$.

Now, let us proceed by contradiction and assume that there is a sequence $(u_p)_{p\in\mathbb{N}}$ of functions such that $(i)$ $u_p\in C^\infty(\bar{\Omega}, \mathbb{R}^{d_2})$, $(ii)$ $\lim_{p\to\infty}\mathscr{R}^{(\mathrm{reg})}_n(u_p)=\mathscr{R}^{(\mathrm{reg})}_n(\hat u_n)$, and $(iii)$ $(u_p)_{p\in\mathbb{N}}$ does not converge to $\hat u_n$ with respect to the $H^m(\Omega)$ norm. 
Therefore, upon passing to a subsequence, there exists $\varepsilon >0$ such that, for all $ p \geqslant 0$, $\|u_p-\hat u_n\|_{H^m(\Omega)} \geqslant \varepsilon$.

Since $\mathscr{R}^{(\mathrm{reg})}_n(u_p) \geqslant \lambda_t \|u_p\|_{H^{m+1}(\Omega)}$, $\lambda_t>0$, and $(u_p)_{p\in\mathbb{N}}$ is a minimizing sequence, $(u_p)_{p\in \mathbb{N}}$ is bounded in $H^{m+1}(\Omega, \mathbb{R}^{d_2})$.
Therefore, Theorem \ref{thm:rellichK} states that passing to a subsequence, $(u_p)_{p\in\mathbb{N}}$ converges to a limit, say $u_\infty$, both weakly in $H^{m+1}(\Omega, \mathbb{R}^{d_2})$ and with respect to the $H^m(\Omega)$ norm. 
Then, since $I$ is weakly lower-semi continuous on $H^{m+1}(\Omega, \mathbb{R}^{d_2})$, we deduce that 
\begin{equation}
    \label{eq:cvI}
    \lim_{p\to\infty}I(u_p) \geqslant  I(u_\infty).
\end{equation}
Recalling the definition of $\tilde \Pi$ in Theorem \ref{thm:sobIneq}, we know that there exists a constant $C_\Omega >0$ such that $\|u_p-\tilde \Pi(u_\infty)\|_{\infty, \Omega} = \|\tilde \Pi(u_p - u_\infty)\|_{\infty, \Omega} \leqslant C_\Omega \|u_p-u_\infty\|_{H^{m}(\Omega)}$.
We deduce that  $\lim_{p\to\infty}F(u_p) = F(u_\infty)$. Therefore, combining this result with \eqref{eq:decomposition} and \eqref{eq:cvI}, we deduce that $\lim_{p\to\infty}\mathscr{R}^{(\mathrm{reg})}_n(u_p) \geqslant  \mathscr{R}^{(\mathrm{reg})}_n(u_\infty)$.
However, recalling that $\lim_{p\to\infty}\mathscr{R}^{(\mathrm{reg})}_n(u_p)=\mathscr{R}^{(\mathrm{reg})}_n(\hat u_n)$ and that $\hat u_n$ is the unique minimizer of $\mathscr{R}^{(\mathrm{reg})}_n$ over $H^{m+1}(\Omega,\mathbb{R}^{d_2})$, we conclude that $u_\infty = \hat u_n$. 

We just proved that there exists a subsequence of $(u_p)_{p\in \mathbb{N}}$ which converges to $\hat u_n$ with respect to the $H^{m}(\Omega)$ norm.
This contradicts the assumption $\|u_p-\hat u_n\|_{H^m(\Omega)} \geqslant \varepsilon$ for all $ p \geqslant 0$.

\subsection{Proof of Theorem \ref{thm:functionalCv}}

The result is an immediate consequence of Theorem \ref{thm:approximation}, Propositions \ref{prop:laxMLin}, and Proposition \ref{prop:sequenceCvLin}.

\subsection{Proof of Theorem \ref{prop:pdeSolverFunctional}}
Throughout the proof, since no data are involved, we denote the regularized theoretical risk by $\mathscr{R}^{(\mathrm{reg})}$ instead of $\mathscr{R}_n^{(\mathrm{reg})}$. Also, to make the dependence in the hyperparameter $\lambda_t$ transparent, we denote by $u(\lambda_t)$ the unique minimizer of $\mathscr{R}^{(\mathrm{reg})}$ instead of $\hat u_n$. 

We proceed by contradiction and assume that 
$\lim_{\lambda_t \to 0}\| u(\lambda_t)-u^\star\|_{H^{m}(\Omega)} \neq 0$.
If this is true, then, upon passing to a subsequence $(\lambda_{t,p})_{p\in \mathbb{N}}$ such that $\lim_{p\to \infty} \lambda_{t,p} =0$, there exists $\varepsilon >0$ such that, for all $ p \geqslant 0$, $\| u(\lambda_{t,p})- u^\star\|_{H^m(\Omega)} \geqslant \varepsilon$.

Notice that $ \| u(\lambda_{t,p})\|_{H^{m+1}(\Omega)} \leqslant \mathscr{R}^{(\mathrm{reg})}(u^\star)/\lambda_{t,p} = \|u^\star\|_{H^{m+1}(\Omega)}$.
Theorem \ref{thm:rellichK} proves that upon passing to a subsequence, $(u(\lambda_{t,p}))_{p\in \mathbb{N}}$ converges with respect to the $H^m(\Omega)$ norm to a function $u_\infty \in H^{m+1}(\Omega, \mathbb{R}^{d_2})$.
Since $m \geqslant K$, the theoretical risk $\mathscr{R}$ is continuous with respect to the $H^m(\Omega)$ norm and we have that $\mathscr{R}(u_\infty) = \lim_{p\to \infty} \mathscr{R}(u(\lambda_{t,p}))$. 
Moreover, by definition of $u(\lambda_{t,p})$ and since $\mathscr{R}(u^\star) = 0$, we have that  $\mathscr{R}(u(\lambda_{t,p})) + \lambda_{t,p} \|u(\lambda_{t,p})\|_{H^{m+1}(\Omega)}\leqslant \lambda_{t,p} \|u^\star\|_{H^{m+1}(\Omega)}$.
Therefore, $\mathscr{R}(u_\infty) = 0$ and $u_\infty = u^\star$. 
This contradicts the assumption that for all $ p \geqslant 0$, $\| u(\lambda_{t,p})- u^\star\|_{H^m(\Omega)} \geqslant \varepsilon$.

\subsection{Proof of Proposition \ref{prop:consistencey}}
We prove the proposition in several steps. In the sequel, given a measure $\mu$ on $\Omega$ and a function $u \in H^{m+1}(\Omega, \mathbb{R}^{d_2})$, we let $\|u\|^2_{L^2(\mu)} = \int_\Omega \|\tilde \Pi(u)(\bx)\|_2^2 d\mu(\bx)$, where, as usual, $\tilde \Pi(u)$ is the unique continuous function such that $\tilde \Pi(u) = u$ almost everywhere.

\paragraph{Step 1: Decomposing the problem into two simpler ones} 
Following the framework of \citet{arnone2022spatialRegression}, the core idea is to decompose the problem into two simpler ones thanks to the linearity in $\hat u_n$ and in $Y_i$ of the identity 
\[\forall v \in H^{m+1}(\Omega, \mathbb{R}^{d_2}),\quad  \mathcal{A}_n(\hat u_n,v) = \mathcal{B}_n(v)\] 
of Proposition \ref{prop:laxMLin}.
Thus, recalling that $Y_i = u^\star(\bX_i)+\varepsilon_i$, we 
let
\begin{align*}
    \mathcal{B}^{\star}_n(v) &= \frac{\lambda_d}{n} \sum_{i=1}^n \langle u^\star(\bX_i), \tilde \Pi(v)(\bX_i)\rangle+ \lambda_e \mathbb{E}\langle\tilde \Pi(v)(\bX^{(e)}),h(\bX^{(e)})\rangle\\
    &\quad - \frac{1}{|\Omega|}\sum_{k=1}^{M}\int_{\Omega} B_k(\bx)\mathscr{F}_k^{(\mathrm{lin})}(v,\bx)d\bx 
\end{align*}
and
\[\mathcal{B}^{\mathrm{(noise)}}_n(v) = \frac{\lambda_d}{n} \sum_{i=1}^n \langle \varepsilon_i, \tilde \Pi(v)(\bX_i)\rangle.\]
Clearly, $\mathcal{B}_n = \mathcal{B}^{\star}_n + \mathcal{B}^{\mathrm{(noise)}}_n.$
Using Proposition \ref{prop:laxMLin} with $Y_i$ instead of $\varepsilon_i$, and setting $\lambda_e = 0$, we see that there exists a unique $\hat u ^{\mathrm{(noise)}}_n \in H^{m+1}(\Omega, \mathbb{R}^{d_2})$ such that, for all $v \in H^{m+1}(\Omega, \mathbb{R}^{d_2})$, $\mathcal{A}_n( \hat u ^{\mathrm{(noise)}}_n,v) = \mathcal{B}^{\mathrm{(noise)}}_n(v)$. Furthermore, $\hat u ^{\mathrm{(noise)}}_n$ is the unique minimizer over $H^{m+1}(\Omega, \mathbb{R}^{d_2})$ of
\begin{align*}
    \mathscr{R}_n^{\mathrm{(noise)}}(u) &= \frac{\lambda_d}{n} \sum_{i=1}^n \|\tilde \Pi (u)(\bX_i)-\varepsilon_i\|_2^2+\lambda_e \mathbb{E}\| u(\bX^{(e)})\|_2^2+\frac{1}{|\Omega|}\sum_{k=1}^{M}\int_{\Omega} \mathscr{F}_k^{(\mathrm{lin})}(u,\bx)^2d\bx  \\
    &\quad + \lambda_t\|u\|_{H^{m+1}(\Omega)}^2.
\end{align*}
Similarly, Proposition \ref{prop:laxMLin} shows that there exists a unique $\hat u ^{\star}_n \in H^{m+1}(\Omega, \mathbb{R}^{d_2})$ such that, for all $v \in H^{m+1}(\Omega, \mathbb{R}^{d_2})$, 
$\mathcal{A}_n( \hat u ^{\star}_n,v) = \mathcal{B}_n^{\star}(v)$, and
$\hat u ^{\star}_n$ is the unique minimizer over $H^{m+1}(\Omega, \mathbb{R}^{d_2})$ of 
\begin{align*}
    \mathscr{R}_n^\star(u) &= \frac{\lambda_d}{n} \sum_{i=1}^n \|\tilde \Pi (u-u^\star)(\bX_i)\|_2^2+\lambda_e \mathbb{E}\| \tilde \Pi(u)(\bX^{(e)}) - h(\bX^{(e)})\|_2^2\\
    &\quad +\frac{1}{|\Omega|}\sum_{k=1}^{M}\int_{\Omega} \mathscr{F}_k(u,\bx)^2d\bx + \lambda_t\|u\|_{H^{m+1}(\Omega)}^2.
\end{align*}
By the bilinearity of $\mathcal{A}_n$, one has, for all $ v \in H^{m+1}(\Omega, \mathbb{R}^{d_2})$, $\mathcal{A}_n( \hat u_n ^{\star}+ \hat u_n ^{\mathrm{(noise)}} ,v) = \mathcal{B}_n(v)$. However, according to Proposition \ref{prop:laxMLin}, $\hat u_n$ is the unique element of $H^{m+1}(\Omega, \mathbb{R}^{d_2})$ satisfying this property. Therefore, 
$\hat u_n = \hat u_n ^{\star}+ \hat u_n ^{\mathrm{(noise)}}$.
\paragraph{Step 2: Some properties of the minimizers} According to Lemma \ref{lem:measurability}, $\hat u_n$, $\hat u_n^\star$, and $\hat u_n^{\mathrm{(noise)}}$ are random variables. Our goal in this paragraph is to prove that $\mathbb{E}\|\hat u_n\|_{H^{m+1}(\Omega)}^2$, $\mathbb{E}\|\hat u_n^\star\|_{H^{m+1}(\Omega)}^2$, and $\mathbb{E}\|\hat u_n^{\mathrm{(noise)}}\|_{H^{m+1}(\Omega)}^2 $ are finite, so that we can safely use conditional expectations on $\hat u_n$,  $\hat u_n^\star$, and  $\hat u_n^\mathrm{(noise)}$. 
Recall that, since  $\lambda_t \|\hat u_n\|_{H^{m+1}(\Omega)}^2 
    \leqslant \mathscr{R}^{(\mathrm{reg})}_n(\hat u_n)  \leqslant \mathscr{R}^{(\mathrm{reg})}_n(0) $, and since $\mathscr{F}_k^{(\mathrm{lin})}(0, \cdot) = 0$, 
\[\lambda_t \|\hat u_n\|_{H^{m+1}(\Omega)}^2 \leqslant \frac{\lambda_d}{n} \sum_{i=1}^n \|Y_i\|_2^2 + \lambda_e \mathbb{E}\|h(\bX^{(e)})\|_2^2 + \frac{1}{|\Omega|} \sum_{k=1}^M \int_\Omega B_k(\bx)^2 d\bx.\]
Hence, 
\[\mathbb{E}\|\hat u_n\|_{H^{m+1}(\Omega)}^2 \leqslant \lambda_t^{-1}\Big(\lambda_d\mathbb{E}\|u^\star(\bX) + \varepsilon\|_2^2 + \lambda_e \mathbb{E}\|h(\bX^{(e)})\|_2^2 + \frac{1}{|\Omega|}\sum_{k=1}^M \int_\Omega B_k(\bx)^2 d\bx\Big).\]
Similarly,
\begin{equation*}
    \mathbb{E}\|\hat u_n^\star\|_{H^{m+1}(\Omega)}^2 \leqslant \lambda_t^{-1}\Big(\lambda_d\mathbb{E}\|u^\star(\bX)\|_2^2 + \lambda_e \mathbb{E}\|h(\bX^{(e)})\|_2^2 + \frac{1}{|\Omega|}\sum_{k=1}^M \int_\Omega B_k(\bx)^2 d\bx\Big),
    \label{eq:boundingOptimizerExpectancy}
\end{equation*} 
and $\mathbb{E}\|\hat u_n^{\mathrm{(noise)}}\|_{H^{m+1}(\Omega)}^2 \leqslant \lambda_t^{-1}\lambda_d\mathbb{E}\|\varepsilon\|_2^2$.

\paragraph{Step 3: Bias-variance decomposition} In this paragraph, we use the notation $\mathcal{A}_{(\bx, e)}(u,u)$ instead of $\mathcal{A}_n(u,u)$, to make the dependence of $\mathcal{A}_n$ in the random variables $\bx = (\bX_1, \hdots, \bX_n)$ and $e=(\varepsilon_1, \hdots, \varepsilon_n)$ more explicit. We do the same with $\mathcal{B}_n$ and $\hat u_n^{\mathrm{(noise)}}$. 
 Observe that, for any $(\bx,e) \in\Omega^n \times \mathbb{R}^{nd_2}$ and for any $u \in H^{m+1}(\Omega, \mathbb{R}^{d_2})$, one has \[\mathcal{A}_{(\bx,-e)}(u,u) - 2 \mathcal{B}_{(\bx,e)}^{\mathrm{(noise)}}(u) = \mathcal{A}_{(\bx,e)}(-u,-u) - 2 \mathcal{B}_{(\bx,-e)}^{\mathrm{(noise)}}(-u).\]
Therefore, $\hat u^{\mathrm{(noise)}}_{(\bx,e)} = -\hat u^{\mathrm{(noise)}}_{(\bx,-e)}$. 

Since, by assumption, $\varepsilon$ has the same law as $-\varepsilon$, this implies $\mathbb{E}(\hat u_n^{\mathrm{(noise)}}\mid \bX_1, \hdots, \bX_n) = 0$, and so $\mathbb{E}(\hat u_n^{\mathrm{(noise)}}) = 0$.
Moreover, since $\hat u_n^\star$ is a measurable function of $\bX_1, \hdots, \bX_n$, we have $\mathbb{E}(\hat u_n^\star\mid\bX_1, \hdots, \bX_n) = \hat u_n^\star$. Recalling (Step 1) that $\hat u_n = \hat u_n ^{\star}+ \hat u_n ^{\mathrm{(noise)}}$, we deduce the following bias-variance decomposition:  
\begin{equation}
    \label{eq:biasVariance}
    \mathbb{E}\|\hat u_n-u^\star\|^2_{L^2(\mu_\bX)} = \mathbb{E}\|\hat u^\star_n-u^\star\|^2_{L^2(\mu_\bX)}+ \mathbb{E}\|\hat u^{\mathrm{(noise)}}_n\|^2_{L^2(\mu_\bX)}.
\end{equation}

\paragraph{Step 4: Bounding the bias} 
Recall that $\hat u_n^\star$ minimizes $\mathscr{R}_n^\star$ over $H^{m+1}(\Omega, \mathbb{R}^{d_2})$, so that $\mathscr{R}_n^\star(u^\star) \geqslant \mathscr{R}_n^\star(\hat u^\star_n)$.
Therefore, $\mathrm{PI}(u^\star)+ \lambda_t\|u^\star\|_{H^{m+1}(\Omega)}^2 \geqslant \frac{\lambda_d}{n} \sum_{i=1}^n \|\tilde \Pi (\hat u_n^\star-u^\star)(\bX_i)\|_2^2$. 
We deduce that
\begin{align*}
    &\frac{1}{\lambda_d}\big(\mathrm{PI}(u^\star)+ \lambda_t\|u^\star\|_{H^{m+1}(\Omega)}^2\big) \\
    &\geqslant  \frac{\|\hat u_n^\star-u^\star\|_{H^{m+1}(\Omega)}^2}{n} \sum_{i=1}^n \Big\|\tilde \Pi \Big (\frac{\hat u_n^\star-u^\star}{\|\hat u_n^\star-u^\star\|_{H^{m+1}(\Omega)}}\Big)(\bX_i)\Big\|_2^2 \\
    & \geqslant  \|\hat u_n^\star-u^\star\|_{L^2(\mu_\bX)}^2 \\
    & \quad - \|\hat u_n^\star-u^\star\|_{H^{m+1}(\Omega)}^2  \sup_{\|u\|_{H^{m+1}(\Omega)}\leqslant 1} \Big( \mathbb{E}\|\tilde \Pi(u)(\bX)\|_2^2 - \frac{1}{n}\sum_{i=1}^n \|\tilde \Pi(u)(\bX_i)\|_2^2\Big) \\
  & \geqslant  \|\hat u_n^\star-u^\star\|_{L^2(\mu_\bX)}^2\\
  & \quad 
   - 2  \big(\|\hat u_n^\star\|_{H^{m+1}(\Omega)}^2+\|u^\star\|_{H^{m+1}(\Omega)}^2\big) \!\sup_{\|u\|_{H^{m+1}(\Omega)}\leqslant 1}\! \Big( \mathbb{E}\|\tilde \Pi(u)(\bX)\|_2^2- \frac{1}{n}\sum_{i=1}^n \|\tilde \Pi(u)(\bX_i)\|_2^2\Big).
\end{align*}
Moreover, $\mathrm{PI}(u^\star)  + \lambda_t\|u^\star\|_{H^{m+1}(\Omega)}^2 \geqslant \lambda_t \|\hat u_n^\star\|_{H^{m+1}(\Omega)}^2$.
Taking expectations, we conclude by Lemma \ref{lem:empiricalL2} that there exists a constant $C_\Omega '$, depending only on $\Omega$, such that
\[\mathbb{E}\|\hat u_n^\star-u^\star\|_{L^2(\mu_\bX)}^2\leqslant \frac{1}{\lambda_d}\big(\mathrm{PI}(u^\star) + \lambda_t\|u^\star\|_{H^{m+1}(\Omega)}^2\big) + \frac{C_\Omega 'd_2^{1/2}}{n^{1/2}} \Big(2\|u^\star\|_{H^{m+1}(\Omega)}^2 + \frac{\mathrm{PI}(u^\star)}{\lambda_t}\Big) .\]

\paragraph{Step 5: Bounding the variance} 
Since $\hat u_n^{\mathrm{(noise)}}$ minimizes $\mathscr{R}_n^{\mathrm{(noise)}}$ over $H^{m+1}(\Omega, \mathbb{R}^{d_2})$, we have $\mathscr{R}_n^{\mathrm{(noise)}}(0) \geqslant \mathscr{R}_n^{\mathrm{(noise)}}(\hat u^{\mathrm{(noise)}}_n)$.
So, 
\[\frac{\lambda_d}{n} \sum_{i=1}^n \|\varepsilon_i\|_2^2  \geqslant \frac{\lambda_d}{n} \sum_{i=1}^n \|\tilde \Pi (\hat u_n^\mathrm{(noise)})(\bX_i)- \varepsilon_i\|_2^2.\]
Observing  that  $\|\tilde \Pi (\hat u_n^\mathrm{(noise)})(\bX_i)- \varepsilon_i\|_2^2 = \|\tilde \Pi (\hat u_n^\mathrm{(noise)})(\bX_i)\|_2^2 -2\langle \tilde \Pi (\hat u_n^\mathrm{(noise)})(\bX_i),\varepsilon_i\rangle + \|\varepsilon_i\|_2^2$, we deduce that
\[\frac{2}{n} \sum_{i=1}^n \langle \tilde \Pi (\hat u_n^\mathrm{(noise)})(\bX_i),\varepsilon_i\rangle  \geqslant  \frac{1}{n} \sum_{i=1}^n \|\tilde \Pi (\hat u_n^\mathrm{(noise)})(\bX_i)\|_2^2,\]
and  
\begin{align*}
&\Big\langle \int_\Omega \tilde \Pi(\hat u_n^\mathrm{(noise)})d \mu_\bX, \frac{2}{n} \sum_{i=1}^n \varepsilon_i\Big\rangle + \frac{2}{n} \sum_{i=1}^n \Big\langle \tilde \Pi (\hat u_n^\mathrm{(noise)})(\bX_i) - \int_\Omega \tilde \Pi(\hat u_n^\mathrm{(noise)})d \mu_\bX,\varepsilon_i\Big\rangle  \\
& \quad \geqslant  \frac{1}{n} \sum_{i=1}^n \|\tilde \Pi (\hat u_n^\mathrm{(noise)})(\bX_i)\|_2^2.
\end{align*}
Therefore, 
\begin{align*}
     \|\hat u_n^\mathrm{(noise)}\|_{L^2(\mu_\bX)}^2  &\leqslant \Big\langle \int_\Omega \tilde \Pi(\hat u_n^\mathrm{(noise)})d \mu_\bX, \frac{2}{n} \sum_{i=1}^n \varepsilon_i\Big\rangle  \\
    &+ \|\hat u_n^\mathrm{(noise)}\|_{H^{m+1}(\Omega)}\sup_{\|u\|_{H^{m+1}(\Omega)}\leqslant 1} \frac{1}{n} \sum_{j=1}^n \langle \tilde \Pi (u)(\bX_j) - \mathbb{E}(\tilde \Pi(u)(\bX)),\varepsilon_j\rangle \\ 
    &+ \|\hat u_n^\mathrm{(noise)}\|_{H^{m+1}(\Omega)}^2\sup_{\|u\|_{H^{m+1}(\Omega)}\leqslant 1} \Big(\mathbb{E}\|\tilde \Pi(u)(\bX_i)\|_2^2 - \frac{1}{n}\sum_{i=1}^n \|\tilde \Pi(u)(\bX_i)\|_2^2\Big)\\
    &:= A+B+C.
\end{align*}
According to the Cauchy-Schwarz inequality,
\[\mathbb{E}(A)  \leqslant \Big(\mathbb{E}\Big\| \int_\Omega \tilde \Pi(\hat u_n^\mathrm{(noise)})d \mu_\bX \Big\|_2^2\Big)^{1/2}  \times \frac{2 (\mathbb{E}\|\varepsilon\|_2^2)^{1/2}}{n^{1/2}},\]
and so, by Jensen's inequality,
\[\mathbb{E}(A) \leqslant \big(\mathbb{E}\|\hat u_n^\mathrm{(noise)}\|_{L^2(\mu_\bX)}^2\big)^{1/2}  \times \frac{2 (\mathbb{E}\|\varepsilon\|_2^2)^{1/2}}{n^{1/2}}.\]
The inequality $\mathscr{R}_n^{\mathrm{(noise)}}(0) \geqslant \mathscr{R}_n^{\mathrm{(noise)}}(\hat u^{\mathrm{(noise)}}_n)$ also implies that
\[\frac{\lambda_d}{n} \sum_{i=1}^n \|\varepsilon_i\|_2^2  \geqslant \frac{\lambda_d}{n} \sum_{i=1}^n \|\tilde \Pi (\hat u_n^\mathrm{(noise)})(\bX_i)- \varepsilon_i\|_2^2 + \lambda_t \|\hat u_n^\mathrm{(noise)}\|_{H^{m+1}(\Omega)}^2.\]
Therefore,
\[\frac{\lambda_d}{n \lambda_t} \sum_{i=1}^n 2\langle \tilde \Pi (\hat u_n^\mathrm{(noise)})(\bX_i),\varepsilon_i\rangle \geqslant  \|\hat u_n^\mathrm{(noise)}\|_{H^{m+1}(\Omega)}^2,\]
and
\[\frac{\lambda_d}{\lambda_t} \sup_{\|u\|_{H^{m+1}(\Omega)}\leqslant 1} \frac{1}{n} \sum_{j=1}^n \langle \tilde \Pi (u)(\bX_j),\varepsilon_j\rangle  \geqslant   \|\hat u_n^\mathrm{(noise)}\|_{H^{m+1}(\Omega)}.\]
By Theorem \ref{thm:sobIneq}, if $\|u\|_{H^{m+1}(\Omega)}\leqslant 1$, then  $ \langle \mathbb{E}(\tilde \Pi(u)(\bX)),\frac{1}{n} \sum_{j=1}^n\varepsilon_j\rangle \leqslant \frac{C_\Omega d_2^{1/2}}{n} \|\sum_{i=1}^n \varepsilon_i\|_2$. Thus, 
\begin{align*}
    &\|\hat u_n^\mathrm{(noise)}\|_{H^{m+1}(\Omega)}\\
    &\leqslant\frac{\lambda_d}{\lambda_t} \Big( \frac{C_\Omega d_2^{1/2}}{n} \|\sum_{i=1}^n \varepsilon_i\|_2 + \sup_{\|u\|_{H^{m+1}(\Omega)}\leqslant 1} \frac{1}{n} \sum_{j=1}^n \langle \tilde \Pi (u)(\bX_j) - \mathbb{E}(\tilde \Pi(u)(\bX)),\varepsilon_j\rangle \Big) .
\end{align*} 
Using Lemma \ref{lem:empiricalProcess} together with the fact that, for all $\bx, \by \in \mathbb{R}$, $(\bx+\by)^2 \leqslant 2(\bx^2 + \by^2)$, 
\begin{equation*}
    \mathbb E \|\hat u_n^\mathrm{(noise)}\|_{H^{m+1}(\Omega)}^2 \leqslant \frac{4\lambda_d^2}{n\lambda_t^2} C_\Omega^2 d_2 \mathbb E\|\varepsilon\|_2^2.
\end{equation*} 
Similarly, observing that for all random variables $X, Y \in \mathbb{R}$, $\mathbb E(XY)^2 \leqslant \mathbb E(X^2) \mathbb E(Y^2) $, 
\[\mathbb E(B) \leqslant \frac{4\lambda_d}{n\lambda_t} C_\Omega^2 d_2 \mathbb E\|\varepsilon\|_2^2.\]
Moreover, by Lemma \ref{lem:empiricalL2} and the inequality $\mathbb E(XYZ)^2 \leqslant \mathbb E(X^2) \mathbb E(Y^2) \mathbb E(Z^2) $, 
\[\mathbb E(C) \leqslant \frac{\lambda_d^2}{n^{3/2} \lambda_t^2}  C_\Omega^2 d_2^{3/2} \mathbb E\|\varepsilon\|_2^2.\]
Therefore, we conclude that there exists a constant $C_\Omega > 0$, depending only on $\Omega$, such that 
\begin{align*}
    \mathbb{E}\|\hat u_n^\mathrm{(noise)}\|_{L^2(\mu_\bX)}^2 &\leqslant \big(\mathbb{E}\|\hat u_n^\mathrm{(noise)}\|_{L^2(\mu_\bX)}^2\big)^{1/2}  \frac{2 (\mathbb{E}\|\varepsilon\|_2^2)^{1/2}}{n^{1/2}} \\
    &\quad + \frac{4\lambda_d}{n\lambda_t} C_\Omega^2 d_2 \mathbb E\|\varepsilon\|_2^2 +  \frac{\lambda_d^2}{n^{3/2} \lambda_t^2}  C_\Omega^2 d_2^{3/2} \mathbb E\|\varepsilon\|_2^2 .
\end{align*}
Hence, using elementary algebra,
\[\big(\mathbb{E}\|\hat u_n^\mathrm{(noise)}\|_{L^2(\mu_\bX)}^2\big)^{1/2} \leqslant \frac{(\mathbb{E}\|\varepsilon\|_2^2)^{1/2}}{n^{1/2}}\Big(2+2C_\Omega d_2^{3/4}\Big(\frac{\lambda_d^{1/2}}{\lambda_t^{1/2}} +\frac{\lambda_d}{\lambda_tn^{1/4}}\Big)\Big)\]
and 
\[\mathbb{E}\|\hat u_n^\mathrm{(noise)}\|_{L^2(\mu_\bX)}^2 \leqslant \frac{8\mathbb{E}\|\varepsilon\|_2^2}{n}\Big(1+C_\Omega d_2^{3/2}\Big(\frac{\lambda_d}{\lambda_t} +\frac{\lambda_d^2}{\lambda_t^2n^{1/2}}\Big)\Big).\]

\paragraph{Step 6: Putting everything together} Combining Steps 3, 4, and 5, we conclude that
\begin{align*}
    \mathbb{E}\|\hat u_n-u^\star\|^2_{L^2(\mu_\bX)} &\leqslant \frac{1}{\lambda_d}\big(\mathrm{PI}(u^\star) + \lambda_t\|u^\star\|_{H^{m+1}(\Omega)}^2\big)  + \frac{C_\Omega 'd_2^{1/2}}{n^{1/2}} \Big(2\|u^\star\|_{H^{m+1}(\Omega)}^2 + \frac{\mathrm{PI}(u^\star)}{\lambda_t}\Big)\\
    &\quad +\frac{8\mathbb{E}\|\varepsilon\|_2^2}{n}\Big(1+C_\Omega d_2^{3/2}\Big(\frac{\lambda_d}{\lambda_t} +\frac{\lambda_d^2}{\lambda_t^2n^{1/2}}\Big)\Big).
\end{align*}
\subsection{Proof of Proposition \ref{thm:phyCst}}
By definition, $\hat u_n$ minimizes $\mathscr{R}_n^{(\mathrm{reg})}$ over $H^{m+1}(\Omega, \mathbb{R}^{d_2})$. So, $\mathscr{R}_n^{(\mathrm{reg})}(u^\star) \geqslant \mathscr{R}_n^{(\mathrm{reg})}(\hat u_n)$. 
Moreover, since 
\[
\|\tilde \Pi (\hat u_n)(\bX_i) - Y_i\|_2^2 = \|\tilde \Pi (\hat u_n-u^\star )(\bX_i)\|_2^2 -2 \langle \tilde \Pi (\hat u_n-u^\star )(\bX_i), \varepsilon_i\rangle + \|\varepsilon_i\|_2^2 ,\]
one has 
\begin{align*}
    &\frac{1}{n}\sum_{i=1}^n \|\tilde \Pi (\hat u_n)(\bX_i) - Y_i\|_2^2 \\
    &\quad \geqslant  -2 \|\hat u_n-u^\star\|_{H^{m+1}(\Omega)}\times \sup_{\|u\|_{H^{m+1}(\Omega)}\leqslant 1} \frac{1}{n} \sum_{j=1}^n \langle \tilde \Pi (u)(\bX_j) - \mathbb{E}(\tilde \Pi (u)(\bX)),\varepsilon_j\rangle\\
    &\qquad -2 \Big\langle \int_\Omega \tilde \Pi (\hat u_n-u^\star )d\mu_\bX, \frac{1}{n}\sum_{i=1}^n \varepsilon_i\Big\rangle + \frac{1}{n}\sum_{i=1}^n \|\varepsilon_i\|_2^2.
\end{align*}
Thus,
\begin{align}
    &\frac{1}{n}\sum_{i=1}^n \|\tilde \Pi (\hat u_n)(\bX_i) - Y_i\|_2^2 \nonumber\\
    &\quad \geqslant -2 (\|\hat u_n\|_{H^{m+1}(\Omega)}+\|u^\star\|_{H^{m+1}(\Omega)}) \sup_{\|u\|_{H^{m+1}(\Omega)}\leqslant 1} \frac{1}{n} \sum_{j=1}^n \langle \tilde \Pi (u)(\bX_j) - \mathbb{E}(\tilde \Pi (u)(\bX)),\varepsilon_j\rangle \nonumber\\
    &\qquad -2 \Big\langle \int_\Omega \tilde \Pi (\hat u_n-u^\star )d\mu_\bX, \frac{1}{n}\sum_{i=1}^n \varepsilon_i\Big\rangle + \frac{1}{n}\sum_{i=1}^n \|\varepsilon_i\|_2^2. \label{eq:phyConsistencyIneg}
\end{align}
Recall from Steps 4 and 5 of the proof of Theorem \ref{prop:consistencey} that
\begin{align*}
    \mathbb{E}\|\hat u_n\|_{H^{m+1}(\Omega)}^2 
    &\leqslant 2\mathbb{E}\|\hat u_n^\star\|_{H^{m+1}(\Omega)}^2 + 2\mathbb{E}\|\hat u_n^{\mathrm{(noise)}}\|_{H^{m+1}(\Omega)}^2\\
    & \leqslant 2\Big(\frac{\mathrm{PI}(u^\star)}{\lambda_t} + \|u^\star\|_{H^{m+1}(\Omega)}^2\Big) + \frac{8\lambda_d^2}{n\lambda_t^2} C_\Omega^2 d_2 \mathbb E\|\varepsilon\|_2^2 
\end{align*}
Therefore, Lemma \ref{lem:empiricalProcess} and the inequality $\mathbb E(XY)^2 \leqslant \mathbb E(X)^2 \mathbb E(Y)^2 $ show that 
\begin{align*}
    &\mathbb{E}\Big( \|\hat u_n\|_{H^{m+1}(\Omega)} \sup_{\|u\|_{H^{m+1}(\Omega)}\leqslant 1} \frac{1}{n} \sum_{j=1}^n \langle \tilde \Pi (u)(\bX_j) - \mathbb{E}(\tilde \Pi (u)(\bX)),\varepsilon_j\rangle \Big) = \Oequivalent_{n\to \infty}\Big(\frac{\lambda_d}{n\lambda_t}\Big). 
\end{align*}
By Theorem \ref{prop:consistencey},
\begin{align*}
    \mathbb{E}\Big|\Big\langle \int_\Omega \tilde \Pi (\hat u_n-u^\star )d\mu_\bX, \frac{1}{n}\sum_{i=1}^n \varepsilon_i\Big\rangle \Big| &\leqslant \big(\mathbb{E}\|u^\star-\hat u_n\|_{L^2(\mu_\bX)}^2\big)^{1/2} \frac{\mathbb{E}\|\varepsilon\|_2^2}{n^{1/2}}= \Oequivalent_{n\to \infty}\Big(\frac{\lambda_d}{n^{2}\lambda_t}\Big)^{1/2}.
\end{align*}
Combining these three results with \eqref{eq:phyConsistencyIneg}, we conclude that
\begin{align*}
    \mathbb{E}\Big(\frac{1}{n}\sum_{i=1}^n \|\tilde \Pi (\hat u_n)(\bX_i) - Y_i\|_2^2\Big) \geqslant \mathbb{E}\|\varepsilon\|_2^2 + \Oequivalent_{n\to \infty}\Big(\frac{\lambda_d}{n\lambda_t}\Big).
\end{align*}
Therefore, since $\lim_{n \to \infty }\frac{\lambda_d^2}{n\lambda_t} = 0$ and since $\mathscr{R}_n^{(\mathrm{reg})}(\hat u_n) = \frac{\lambda_d}{n}\sum_{i=1}^n \|\tilde \Pi (\hat u_n)(\bX_i) - Y_i\|_2^2+\mathrm{PI}(\hat u_n) + \lambda_t \|\hat u_n\|_{H^{m+1}(\Omega)}^2$,
\[\mathbb{E}\big(\mathscr{R}_n^{(\mathrm{reg})}(\hat u_n)\big) \geqslant  \lambda_d \mathbb{E}\|\varepsilon\|_2^2 + \mathbb{E}(\mathrm{PI}(\hat u_n)) + \oequivalent_{n\to \infty}(1).\]
Similarly, almost everywhere,
\[\frac{1}{n}\sum_{i=1}^n \|\tilde \Pi (\hat u^\star)(\bX_i) - Y_i\|_2^2 = \frac{1}{n}\sum_{i=1}^n \|\varepsilon_i\|_2^2.\]
Hence, 
\[\mathbb{E}\big(\mathscr{R}_n^{(\mathrm{reg})}(u^\star)\big) = \lambda_d \mathbb{E}\|\varepsilon\|_2^2 + \mathrm{PI}( u^\star) + \lambda_t \|u^\star\|_{H^{m+1}(\Omega)}^2.\]
Since $ \mathbb{E}(\mathscr{R}_n^{(\mathrm{reg})}(\hat u_n)) \leqslant \mathbb{E}(\mathscr{R}_n^{(\mathrm{reg})}(u^\star))$ and since $\lambda_t \to 0$, we are led to 
\[\mathbb{E}(\mathrm{PI}(\hat u_n)) \leqslant \mathrm{PI}(u^\star) + \oequivalent_{n\to \infty}(1),\]
which is the desired result.
\end{appendix}

\end{document}